\documentclass[10pt]{article}

\usepackage{amsfonts}
\usepackage{amssymb}
\usepackage{amsthm}
\usepackage{amsmath}
\usepackage{bm}
\usepackage{amscd} 
\usepackage{tikz-cd}

\usepackage[mathcal]{euscript}
\usepackage{txfonts}

\usepackage[all]{xy}

\usepackage{hyperref}

\makeatletter
\newcommand*{\forbackarrow}[2]{\mathrel{
  \settowidth{\@tempdima}{$\scriptstyle#1$}
  \settowidth{\@tempdimb}{$\scriptstyle#2$}
  \ifdim\@tempdimb>\@tempdima \@tempdima=\@tempdimb\fi
  \mathop{\vcenter{
    \offinterlineskip\ialign{\hbox to\dimexpr\@tempdima+1em{##}\cr
    \rightarrowfill\cr\noalign{\kern.5ex}
    \leftarrowfill\cr}}}\limits^{\!#1}_{\!#2}}}

\numberwithin{equation}{section} \DeclareMathSizes{2}{10}{12}{13}
\newtheorem{thm}{Proposition}[section]

\newtheorem{cor}[thm]{Corollary}
\newtheorem{lem}[thm]{Lemma}
\newtheorem{defn}[thm]{Definition}
\newtheorem{exm}[thm]{Example}

\title{Classifying subcategories, local algebra and the Rosenberg  spectrum of a locally noetherian category}

\smallskip 

\smallskip 
\author{Abhishek Banerjee \footnote{Dept. of Mathematics, Indian Institute of Science, Bangalore-560012, Karnataka, India. Email: abhishekbanerjee1313@gmail.com }}
\date{ }

\begin{document}

\maketitle

\begin{abstract} Let $\mathcal A$ be a locally noetherian Grothendieck category. In this paper, we study subcategories 
of $\mathcal A$ using subsets of the Rosenberg spectrum $\mathfrak Spec(\mathcal A)$. Along the way, we also develop results in  local algebra with respect to the category $\mathcal A$ that we believe to be of independent interest. 
\end{abstract}

\medskip

\medskip
\emph{MSC(2010) Subject Classification: 16B50, 16D90, 16P40, 18E15} 


\section{Introduction}

\medskip

\medskip
Let $R$ be a commutative ring having Zariski spectrum $Spec(R)$ and let $R-Mod$ be the category of $R$-modules. Then, the study of subcategories of $R-Mod$ using the
subsets of $Spec(R)$ is a topic that is well developed in the literature, especially when $R$ is noetherian (see, for example,  
\cite{Four}, \cite{Nee}, \cite{SW}, \cite{Taka}). 
For instance, a theorem of Garkusha and Prest \cite{GP1}, \cite{GP2} shows that torsion classes of finite type in 
$R-Mod$ correspond to complements of intersections of quasi-compact open subsets of $Spec(R)$.  If $R$ is a  commutative and noetherian, a theorem of H\"{u}gel,  Posp\'{i}\v{s}il, \v{S}\v{t}ov\'{i}\v{c}ek and Trlifaj \cite{Four}  shows that  that there is a one-one correspondence between
specialization closed subsets of the Zariski spectrum $Spec(R)$ and hereditary torsion pairs on the category of
$R$-modules. The connection between supports, torsion classes and subsets of $Spec(R)$ is a highly developed topic in the literature, studied by several
authors (see, for instance, Bravo, Odaba\c{s}i, Parra and P\'{e}rez \cite{BOPP}, Enochs, Garcia Rozas and Oyonarte \cite{EGRO}, Izurdiaga \cite{Izur}, Lindo \cite{Hayd}, Parra and Saor\'{i}n \cite{PS}). 

\smallskip
In this paper, our objective is to take a more general approach to this genre of results by working with a locally noetherian
Grothendieck category $\mathcal A$ instead of the category of modules over a  commutative noetherian ring. We feel
that this categorical approach vastly expands the scope of such results beyond commutative algebra. Consequently, 
most of our conclusions in this paper apply to large classes of noncommutative rings. Further, they also apply to the category
of quasi-coherent sheaves over a separated noetherian scheme.  We mention here our recent work in \cite{ABJp}, where we gave conditions for a Grothendieck
category to be locally noetherian, thus generalizing two criteria due to  Guil Asensio, Izurdiaga and Torrecillas \cite{GAIT} and Mohamed and M\"{u}ller \cite{MaMu} respectively. We therefore hope to continue our efforts to study properties of modules over a commutative noetherian ring in the more general context of locally noetherian Grothendieck categories.

\smallskip
In place of the Zariski spectrum $Spec(R)$ of a commutative ring, we will use the spectrum $\mathfrak Spec(\mathcal A)$ of an abelian category $\mathcal A$ constructed by
Rosenberg  \cite{Rose2}, \cite{R3}, \cite{Rose1} in order to study noncommutative algebraic geometry. A famous
result of Gabriel \cite{Gab} from the 1960s shows that a noetherian scheme may be reconstructed from its category
of quasi-coherent sheaves. The spectrum construction $\mathfrak Spec(\mathcal A)$ was used by Rosenberg \cite{R3}
to prove this for separated schemes, i.e., two separated schemes $X$ and $Y$ are isomorphic if and only
if they have equivalent categories $QCoh(X)$ and $QCoh(Y)$ of quasi-coherent sheaves.

\smallskip
 Each point of the spectrum $\mathfrak Spec(\mathcal A)$ corresponds to an equivalence
class $\langle P\rangle$ of some spectral object $P\in Spec(\mathcal A)$ (see Rosenberg \cite{Rose2}, \cite{Rose1}). The striking fact is that when $\mathcal A$
is the category $R-Mod$ of modules over a commutative ring, every spectral object of $R-Mod$ is equivalent to
some quotient $R/\mathfrak{p}$, with $\mathfrak{p}$ being a prime ideal in $R$. Here, we also refer the reader to the series of papers by Kontsevich
and Rosenberg on noncommutative spaces and noncommutative schemes (see \cite{KR1}, \cite{KR2}, \cite{KR3}, \cite{KR4}).  We also note here the approach of  Kanda towards studying commutative algebra
over Grothendieck categories by means of the atom spectrum (see  \cite{Kanda1}, 
\cite{Kanda2}, \cite{Kanda3}, \cite{Kanda4}).   When $\mathcal A$ is a locally noetherian Grothendieck category,
we know (see \cite[Theorem 5.9]{Kanda1}) that the atom spectrum of $\mathcal A$ is homeomorphic to the Ziegler spectrum $Zg(\mathcal A)$ of $\mathcal A$. Under suitable conditions, we have shown (see Proposition \ref{P8.65jq}) that the points of the Rosenberg spectrum $\mathfrak Spec(\mathcal A)$ are bijective
to the points of $Zg(\mathcal A)$ and hence to the points of the atom spectrum. 

\smallskip
Each point $\langle P\rangle\in \mathfrak Spec(\mathcal A)$ corresponds to a localizing subcategory
of $\mathcal A$ which is still denoted by $\langle P\rangle$. Accordingly, there is a canonical functor
$L_{\langle P\rangle}:\mathcal A\longrightarrow \mathcal A/\langle P\rangle$ as well as its right adjoint
$i_{\langle P\rangle}:\mathcal A/\langle P\rangle\longrightarrow \mathcal A$. With this in hand, Rosenberg \cite{Rose2} 
was able to construct good notions of associated points $Ass(M)$ and support $Supp(M)$ for each object $M\in \mathcal A$. This
categorical form of local algebra from \cite{Rose2} will be the main tool for our results in this paper. 

\smallskip
If $\mathcal A$ is a locally noetherian Grothendieck category, it is well known that  the subcategory $\mathcal A_{fg}$ of finitely generated
objects of $\mathcal A$ is an abelian subcategory. Under certain conditions, we show in Proposition \ref{P3.9} that there is a bijection between Serre subcategories
of $\mathcal A_{fg}$ and specialization closed subsets of $\mathfrak Spec(\mathcal A)$ (also compare \cite[Theorem 5.5]{Kanda1}).    For a commutative noetherian ring $R$, it was shown by Stanley and Wang \cite{SW} that Serre
subcategories of the category  $R-mod$ of finitely generated $R$-modules coincide with torsion classes, as also with thick and narrow subcategories of $R-mod$. We prove a counterpart of this result in Proposition \ref{AShah}. 

\smallskip
Since $\mathcal A$ is a Grothendieck category, every object $M\in \mathcal A$ has an injective hull $E(M)$. In
particular, the injective
hulls of spectral objects of $\mathcal A$ will play an important role in connecting specialization closed
subsets of $\mathfrak Spec(\mathcal A)$ to torsion pairs on $\mathcal A$. In particular, we give in Proposition \ref{P5.1} a one-one correspondence between hereditary torsion pairs on 
$\mathcal A$ and specialization closed subsets of the spectrum $\mathfrak Spec(\mathcal A)$.  The explicit description of these torsion pairs is given in terms of spectral objects of 
$\mathcal A$. 
We also develop  in Section 5 results on supports, associated
points and essential extensions of objects of $\mathcal A$ that we feel are of independent interest.

\smallskip
Thereafter, in Section 6, we begin to study resolving subcategories (as defined by  \v{S}\v{t}ov\'{i}\v{c}ek \cite{Stov}) with the help of subsets of
$\mathfrak Spec(\mathcal A)$. The main property of a resolving subcategory $\mathcal C\subseteq \mathcal A$
is that every object $M\in \mathcal A$ is the target of an epimorphism $C\twoheadrightarrow M$ with 
$C\in \mathcal C$.  This allows us to study syzygy like objects for each $M\in \mathcal A$ with respect to the 
resolving subcategory $\mathcal C\subseteq \mathcal A$. For given $n\geq 1$, we say that a resolving subcategory
$\mathcal C\subseteq \mathcal A$ is $n$-resolving if every object $M\in \mathcal A$ fits into an exact sequence
\begin{equation*}
0\longrightarrow C_n\longrightarrow C_{n-1}\longrightarrow \dots \longrightarrow C_0\longrightarrow M\longrightarrow 0
\end{equation*}
with $C_i\in \mathcal C$ for each $0\leq i\leq n$.  We show in Proposition \ref{P6.5} that   $1$-resolving subcategories $\mathcal C\subseteq \mathcal A$ that are torsion free classes of hereditary
torsion pairs correspond to certain 
specialization closed subsets $S\subseteq \mathfrak Spec(\mathcal A)$.

\smallskip
In order to study $n$-resolving subcategories, we develop a better understanding of local algebra in the 
category  $\mathcal A$. Since $\mathcal A$ is a Grothendieck category, we can consider for any object 
$M\in \mathcal A$ the cosyzygy $\mho_1(M):=E(M)/M$ as well as higher cosyzygies $\mho_k(M):=\mho_1(\mho_{k-1}(M))$
for $k>1$. The main result of Section 7 is a characterization of associated points of higher cosyzygies of objects
of $\mathcal A$ under certain conditions. 
\begin{equation*}
\begin{array}{c}
Ass(M)=\{\mbox{$\langle P\rangle $ $\vert$ $P$ is spectral and $Hom_{\mathcal A/\langle P\rangle}(L_{\langle P\rangle}(P),L_{\langle P\rangle}(M))\ne 0$}\}\\
 Ass(\mho_k(M))=\{\mbox{$\langle P\rangle$ $\vert$ $P$ is spectral and 
 $Ext^k_{\mathcal A/\langle P\rangle}(L_{\langle P\rangle}(P),L_{\langle P\rangle}(M))\ne 0$}\}\\
\end{array}
\end{equation*} These results are motivated by a famous lemma of Bass that gives a criterion for associated primes of higher cosyzygies
of modules over commutative noetherian rings using $Ext$ groups. 

\smallskip

We now fix a generator $G\in \mathcal A$ and consider decreasing sequences $\tilde{Y}=(Y_1\supseteq ... \supseteq Y_n)$
of subsets of $\mathfrak Spec(\mathcal A)$ satisfying $Ass(\mho_{i-1}(G))\cap Y_i=\phi$ for $1\leq i\leq n$. We refer 
to these as $G$-sequences of length $n$. On the other hand, any $n$-resolving subcategory $\mathcal C\subseteq
\mathcal A$ leads to a chain
\begin{equation*}
\mathcal A=\mathcal C_{(n+1)}\supseteq \mathcal C_{(n)} \supseteq ... \supseteq \mathcal C_{(1)}=\mathcal C
\end{equation*}  of subcategories with $\mathcal C_{(j)}$ being an $(n-j+1)$-resolving subcategory. In Section 8, we give conditions under which  $G$-sequences
of length $n$ are in one-one correspondence with certain  $n$-resolving subcategories of $\mathcal A$ (see Proposition \ref{PR8.11}). Although not an analogue, this result 
is motivated by the techniques of  H\"{u}gel,  Posp\'{i}\v{s}il, \v{S}\v{t}ov\'{i}\v{c}ek and Trlifaj \cite[Theorem 3.7]{Four} in showing a bijection between certain sequences of
specialization closed subsets of $Spec(R)$ and $n$-cotilting classes of $R$-modules. 

 \smallskip
In Section 9, we deal with arbitrary subsets of $\mathfrak Spec(\mathcal A)$ (i.e., not necessarily specialization closed). Our main result  in this respect is to describe a correspondence
between arbitrary  subsets of $\mathfrak Spec(\mathcal A)$ and subcategories of finitely generated objects that are closed under subobjects, finite direct sums and essential
extensions. We conclude in Section 10 by giving examples of abelian categories whose subcategories may be studied using the theory developed in the paper.
 
\smallskip
{\bf Acknowledgements:} The author is grateful for the hospitality of the Fields Institute in Toronto, where some of this paper was written.

\medskip

\section{Preliminaries on the spectrum of an abelian category}

\medskip

\medskip
Let $\mathcal A$ be an abelian category that satisfies (AB5), i.e., the category $\mathcal A$ has arbitrary direct
sums and filtered colimits commute with finite limits. We recall the following three kinds of objects in $\mathcal A$ 
(see \cite{AdRos} and \cite{Bo} for definitions) 

\smallskip
(1) An object $X\in \mathcal A$ is said to be finitely generated if we have 
\begin{equation*}
\underset{\lambda\in\Lambda}{\varinjlim}\textrm{ }Hom(X,Y_\lambda)\cong Hom(X,\underset{\lambda\in\Lambda}{\varinjlim}\textrm{ }
Y_\lambda)
\end{equation*} for any filtered system $\{Y_\lambda\}_{\lambda\in \Lambda}$ of objects in $\mathcal A$
connected by monomorphisms. 

\smallskip
(2) An object $X\in \mathcal A$ is said to be finitely presented if the functor $Hom(X,\_\_):\mathcal A
\longrightarrow \mathbf{Ab}$ to the category $\mathbf{Ab}$ of abelian groups preserves filtered colimits. 

\smallskip
(3) An object $X\in \mathcal A$ is said to be noetherian if every subobject of $X$ is finitely generated. 

\smallskip
The (AB5) category $\mathcal A$ is said to be locally noetherian if it has a small generating family $\{G_i\}_{i\in I}$ of noetherian objects. In that case, the object 
$G:=\underset{i\in I}{\bigoplus}G_i$ is a generator for $\mathcal A$ and  it follows that
$\mathcal A$ is a Grothendieck category. Then, (see \cite[Proposition 1.9.1]{Tohoku})  every object $X\in \mathcal A$
may be expressed as the quotient of a direct sum of (possibly infinitely many copies) of the generator $G$. This implies, in particular,
that $\mathcal A$ is locally finitely generated, i.e., every object in $\mathcal A$ may be expressed as the filtered
colimit of its finitely generated subobjects.

\smallskip
In a locally noetherian category $\mathcal A$, the finitely generated objects coincide with the noetherian objects and also with the finitely presented objects. Further, the full subcategory $\mathcal A_{fg}$ of finitely generated
objects  is an abelian category. 

\smallskip
We will now describe the spectrum $\mathfrak Spec(\mathcal A)$ of an abelian category $\mathcal A$ constructed by Rosenberg \cite{Rose2}, \cite{Rose1}. This begins with the construction of a relation  on the objects of
the category: $X\prec Y$ for objects $X$, $Y\in \mathcal A$ if $X$ is the subquotient of a direct sum of finitely
many copies
of $Y$. This gives rise to an equivalence relation $\approx$ as follows : $X\approx Y$ if and only if 
$X\prec Y\prec X$. The equivalence class of an object $X\in \mathcal A$ is denoted by $[X]$. Further, for 
any $X\in \mathcal A$, we set $\langle X\rangle$ to be the full subcategory of $\mathcal A$ whose objects are 
given by:
\begin{equation}
Ob(\langle X\rangle) :=Ob(\mathcal A) - \{\mbox{$Y\in Ob(\mathcal A)$ $\vert$ $X\prec Y$}\}
\end{equation} It may be easily verified that $X\prec Y$ if and only if $\langle X\rangle \subseteq \langle Y
\rangle$. 

\smallskip
An object $P\in \mathcal A$ is said to be spectral if $P\ne 0$ and   $Q\approx P$ for any nonzero subobject
$0\ne Q\subseteq P$. The collection of spectral objects of $\mathcal A$ will be denoted by $Spec(\mathcal A)$. The remarkable fact is that when $\mathcal A=R-Mod$, the category of modules over a commutative
ring $R$, each spectral object of $R-Mod$ is equivalent to the quotient $R/\mathfrak{p}$ over a prime ideal $
\mathfrak{p}\subseteq R$. 

\smallskip The
spectrum of the abelian category is now defined as follows:
\begin{equation}\label{spectrum}
\mathfrak Spec(\mathcal A):=\{\mbox{$\langle P\rangle$ $\vert$ $P\in \mathcal A$ is spectral }\}
\end{equation} The support of an object $M\in \mathcal A$ is defined as follows:
\begin{equation}\label{support}
Supp(M):=\{\mbox{$\langle P\rangle\in \mathfrak Spec(\mathcal A) $ $\vert$ $M\notin Ob(\langle P\rangle)$ }\}
=\{\mbox{$\langle P\rangle$ $\vert$ $P\in Spec(\mathcal A)$ and $P\prec M$}\}
\end{equation} In other words, we have:
\begin{equation}\label{supt}
\langle P\rangle \in Supp(M) \qquad \Leftrightarrow \qquad P\prec M 
\end{equation} for any $P\in Spec(\mathcal A)$. 
The associated points of an object $M\in \mathcal A$ are defined as follows (see \cite[$\S$ 8]{Rose2}):
\begin{equation}\label{ass}
Ass(M):=\{\mbox{$\langle P\rangle\in \mathfrak Spec(\mathcal A) $ $\vert$ $P\in Spec(\mathcal A)$ and $P\subseteq M$ }\}
\end{equation} It is evident that $Ass(M)\subseteq Supp(M)$ for each $M\in \mathcal A$.

\smallskip 
There are several topologies that may be considered on $\mathfrak Spec(\mathcal A)$ that may be studied in 
various ways, but the facts that we will need most in the paper are as follows (see \cite[$\S$ 1.6]{Rose1}) : 

\smallskip
(1) The topology $\tau_*$ on $\mathfrak Spec(\mathcal A)$ : A base of closed sets for $\tau_*$ is given by
the collection of all $Supp(M)$, where $M$ varies over all finite direct sums of spectral objects of $\mathcal A$. 

\smallskip
(2) The topology $\tau^*$ on $\mathfrak Spec(\mathcal A)$ : A base of closed sets for $\tau^*$ is given by 
the collection of all $Supp(M)$, where $M$ varies over all finitely generated objects of $\mathcal A$. 

\smallskip
(3) When $\mathcal A$ is a locally noetherian category such that every nonzero object in $\mathcal A$ has an associated
point, i.e., $Ass(M)\ne\phi$ for each $0\ne M\in \mathcal A$, the topologies $\tau^*$ and $\tau_*$ on $\mathfrak Spec(\mathcal A)$ coincide (see \cite[Corollary 1.6.4.3]{Rose1}). 

\section{Serre subcategories in terms of supports}

\smallskip

Throughout this section and the rest of this paper, we will assume that $\mathcal A$ is a locally noetherian Grothendieck category. Further, we always assume that every nonzero object in $\mathcal A$ has an associated point. Accordingly, we consider
the  spectrum $\mathfrak Spec(\mathcal A)$ with the topology $\tau_*=\tau^*$ explained as in Section 2. 

\smallskip
\begin{lem}\label{L3.1} The collection of 
subcategories of $\mathcal A_{fg}$ that are both full and replete forms a set.
\end{lem}

\begin{proof} Let $\mathcal C\subseteq \mathcal A_{fg}$ be a full subcategory. Then 
$\mathcal C$ can be described completely by specifying its objects. Additionally, suppose that $\mathcal C$
is replete, i.e., if $X\in \mathcal C$ then $Y\in \mathcal C$ for any object $Y\in \mathcal A_{fg}$ isomorphic
to $X$. Hence, $\mathcal C$ can be described by specifying the isomorphism classes of objects of $\mathcal A_{fg}$
lying in $\mathcal C$.

\smallskip We now choose a generator $G$ for the locally noetherian (hence Grothendieck) category
$\mathcal A$.  We have mentioned before in Section 2 that every object $X\in \mathcal A$ may be expressed
as the quotient of a direct sum of (possibly infinitely many) copies of $G$. Additionally, if $X\in \mathcal A_{fg}$, i.e., 
$X$ is finitely generated, it follows that there exists an integer $n\geq 1$ such that $X$ is a quotient of a 
direct sum of $n$-copies of $G$. 

\smallskip
Finally, since a Grothendieck category is well powered (see, for instance, \cite[Proposition IV.6.6]{Bo}), the subobjects 
(and hence quotients) of each $G^n$ form a set. It follows that 
the isomorphism classes of objects in $\mathcal A_{fg}$ also form a set. This proves the result. 
\end{proof}

\smallskip
\begin{lem}\label{L3.2} Let $P\in Spec(\mathcal A)$ be a spectral object. Then, $P$ is finitely generated, i.e., 
$P\in \mathcal A_{fg}$. 
\end{lem}

\begin{proof} We consider a generating family $\{G_i\}_{i\in I}$ of noetherian objects for the category
$\mathcal A$. Then, there exists $i_0\in I$ such that we have a nonzero morphism $\phi: G_{i_0}\longrightarrow P$. 
We set $Q:=Im(\phi)$. Since $G_{i_0}$ is finitely generated, so is the object $Q\subseteq P$. 

\smallskip
On the other hand, since $P$ is spectral, we must have $P\prec Q$, i.e., $P$ is a subquotient of a direct sum of finitely
many copies of $Q$. Since the finitely generated objects in $\mathcal A$ coincide with the noetherian objects, this shows that $P$ is finitely generated. 
\end{proof}

\begin{lem}\label{L3.3} Let $P\in Spec(\mathcal A)$ be a spectral object. Then, the closure of the point
$\langle P\rangle\in \mathfrak Spec(\mathcal A)$ is given by :
\begin{equation}
\overline{\langle P\rangle}=Supp(P)
\end{equation} 
\end{lem}

\begin{proof} From \eqref{support}, it is clear that  $\langle P\rangle 
\in Supp(P)$. Further, we know from Lemma \ref{L3.2} that $P$ is finitely generated and hence
$Supp(P)\subseteq \mathfrak Spec(\mathcal A)$ is a closed set. 

\smallskip
If $M\in \mathcal A_{fg}$ is such that $\langle P\rangle \in Supp(M)$, i.e., $M\notin \langle P\rangle$, we now claim 
that $Supp(P)\subseteq Supp(M)$. Indeed, if $Q\in Spec(\mathcal A)$ is such that $\langle Q\rangle \in Supp(P)$, then
$P\notin \langle Q\rangle$, i.e., $Q\prec P$. Then, if $M\notin \langle P\rangle$, i.e., $P\prec M$, we must have
$Q\prec M$. Hence, $M\notin \langle Q\rangle$ and  $\langle Q\rangle \in Supp(M)$. 
\end{proof}

\begin{defn}\label{D3.4}  For a full and replete subcategory $\mathcal C\subseteq \mathcal A_{fg}$, set:
\begin{equation}
Supp(\mathcal C):=\underset{M\in\mathcal C}{\bigcup}\textrm{ }Supp(M)
\end{equation} Conversely, for a subset $S\subseteq \mathfrak Spec(\mathcal A)$, set:
\begin{equation}\label{eq3.3}
Supp^{-1}(S):=\{\mbox{$M\in \mathcal A_{fg}$ $\vert$ $Supp(M)\subseteq S$}\}\subseteq \mathcal A_{fg}
\end{equation}
\end{defn}

We know from Lemma \ref{L3.1} that the collection of full and replete subcategories of
$\mathcal A_{fg}$ forms a set. 
From the definition of the topology on $\mathfrak Spec(\mathcal A)$, it is clear that $Supp$ defines a map :
\begin{equation*}
Supp : \{\mbox{Full \& replete subcategories of $\mathcal A_{fg}$}\}\longrightarrow \{\mbox{Specialization closed subsets
of $\mathfrak Spec(\mathcal A)$}\}
\end{equation*} We recall that a subset  is specialization closed if it is a union of
closed sets. If $R$ is a commutative noetherian ring and $R-mod$ is the category of finitely generated $R$-modules,
Stanley and Wang \cite{SW} have constructed criteria on sets of subcategories of $R-mod$ so that the restriction 
of the map $Supp$ as in Definition \ref{D3.4} is a bijection with inverse $Supp^{-1}$. Our first purpose in this section
is to obtain analogous criteria for the locally noetherian category $\mathcal A$. We begin with the following result.

\begin{thm}\label{P3.5} Let $S\subseteq \mathfrak Spec(\mathcal A)$ be a specialization closed subset. Then, 
$Supp(Supp^{-1}(S))=S$. 
\end{thm}

\begin{proof} We consider a spectral object $P\in Spec(\mathcal A)$ such that $\langle P\rangle \in 
S\subseteq \mathfrak Spec(\mathcal A)$. Since $S$ is specialization closed, we must have
$\overline{\langle P\rangle}\subseteq S$ and it follows from Lemma \ref{L3.3} that $Supp(P)\subseteq S$. Further,
from Lemma \ref{L3.2} we know that $P\in \mathcal A_{fg}$. It follows that $P\in Supp^{-1}(S)$. Now since 
$\langle P\rangle \in Supp(P)$, we get $\langle P\rangle \in Supp(Supp^{-1}(S))$. 

\smallskip
Conversely, we consider some $P\in Spec(\mathcal A)$ such that $\langle P\rangle \in Supp(Supp^{-1}(S))$. Then, 
there exists some $M\in Supp^{-1}(S)$ such that $\langle P\rangle \in Supp(M)$. Using the definition in \eqref{eq3.3}, $M\in Supp^{-1}(S)$
implies that $Supp(M)\subseteq S$ and hence $\langle P\rangle \in S$. This proves the result.

\end{proof}

\begin{lem}\label{L3.6} (a) Let $\mathcal C\subseteq \mathcal A_{fg}$ be a full subcategory that contains
$0$ and is closed under extensions. Suppose that $M\in \mathcal A_{fg}$ is such that $P\in \mathcal C$ for 
every spectral object $P\in Spec(\mathcal A)$ with $\langle P\rangle \in Supp(M)$. Then, $M\in \mathcal C$. 

\smallskip
(b) Let $\mathcal C\subseteq \mathcal A_{fg}$ be a replete and full subcategory satisfying the following two conditions:

\smallskip
(1) $\mathcal C$ is closed under extensions and $0\in\mathcal C$. 

\smallskip
(2) If $M\in \mathcal C$, then $P\in \mathcal C$ for each spectral object $P$ such that $\langle P\rangle \in Supp(M)$. 

\smallskip
Then, $Supp^{-1}(Supp(\mathcal C))=\mathcal C$. 
\end{lem}

\begin{proof} (a) From \cite[1.6.4.1]{Rose1}, we know that any $M\in \mathcal A_{fg}$ can be given a finite
filtration $0=M_0\subseteq M_1 \subseteq \dots \subseteq M_n=M$ such that $Q_i:=M_i/M_{i-1}\in Spec(\mathcal A)$
for each $1\leq i\leq n$.  It is clear that $Q_i\prec M_i$ and hence  
$\langle Q_i\rangle\in Supp(M_i)\subseteq Supp(M)$.   From the assumption on $M$, it now follows that 
$Q_i\in \mathcal C$. In particular, $M_1=Q_1\in \mathcal C$. Since $\mathcal C$ is closed under extensions,
the successive  short exact sequences 
\begin{equation}
0\longrightarrow M_{i-1}\longrightarrow M_i\longrightarrow Q_i\longrightarrow 0\qquad 2\leq i\leq n
\end{equation} show that each $M_i\in \mathcal C$. Then, $M=M_n\in \mathcal C$. 

\smallskip
(b) It is evident from Definition \ref{D3.4} that $\mathcal C\subseteq Supp^{-1}(Supp(\mathcal C))$. Conversely,
we consider an object $M\in Supp^{-1}(Supp(\mathcal C))\subseteq \mathcal A_{fg}$. Then,  $Supp(M)\subseteq Supp(\mathcal C)$. 

\smallskip
We now claim that $P\in \mathcal C$ for any spectral object  $P\in Spec(\mathcal A)$ with $\langle P\rangle \in Supp(M)$. Indeed, if $\langle P\rangle 
\in Supp(M)\subseteq Supp(\mathcal C)$, we can choose some $N\in \mathcal C$ such that $\langle P\rangle \in Supp(N)$. From
 assumption (2) on $\mathcal C$, it follows that $P\in \mathcal C$. 
It now follows from part (a) that $M\in \mathcal C$. 
\end{proof}

\begin{thm}\label{P3.7}  Let $\mathfrak T$ be a collection of subcategories
of $\mathcal A_{fg}$ satisfying the following conditions:

\smallskip
(1) For each specialization closed subset $S\subseteq \mathfrak Spec(\mathcal A)$, $Supp^{-1}(S)\in \mathfrak T$. 

(2) Each subcategory $\mathcal C\in \mathfrak T$ is a full and replete subcategory of 
$\mathcal A_{fg}$   that contains zero and is closed under extensions. 

(3) If $\mathcal C\in\mathfrak T$ and $M\in \mathcal C$, then $P\in \mathcal C$ for each spectral object
$P$ such that $\langle P\rangle \in Supp(M)$. 

\smallskip
Then, the restriction of the map $Supp$ :
\begin{equation}
Supp_{\mathfrak T}:=Supp | \mathfrak T : \mathfrak T\longrightarrow \{\mbox{Specialization closed subsets
of $\mathfrak Spec(\mathcal A)$}\}
\end{equation} is a bijection having inverse $Supp^{-1}$. 
\end{thm} 

\begin{proof} We consider a specialization closed subset $S\subseteq \mathfrak Spec(\mathcal A)$. By assumption,
$Supp^{-1}(S)\in \mathfrak T$. From Proposition \ref{P3.5}, we now get $Supp_{\mathfrak T}(Supp^{-1}(S))=Supp(Supp^{-1}(S))=S$. 

\smallskip
On the other hand, any subcategory $\mathcal C\in\mathfrak T$ satisfies the conditions in part (b) of 
Lemma \ref{L3.6}. It follows from Lemma \ref{L3.6}(b) that $Supp^{-1}(Supp_{\mathfrak T}(\mathcal C))
=Supp^{-1}(Supp(\mathcal C))=\mathcal C$.
\end{proof}

We now recall the following definition 
(see, for instance, \cite[Tag 02MN]{Stacks}) 

\begin{defn}\label{D3.8} Let $\mathcal B$ be an abelian category.  A   subcategory $\mathcal B'\subseteq \mathcal B$ is said to be an abelian subcategory if $\mathcal B'$ is also abelian
and the inclusion functor $\mathcal B'\hookrightarrow \mathcal B$ is exact. 

\smallskip A full subcategory $\mathcal B'\subseteq \mathcal B$ is a Serre subcategory if it is an abelian subcategory 
that is closed under extensions and 
subobjects. Equivalently a full and replete subcategory $\mathcal B'\subseteq \mathcal B$ containing $0$ is said to be a Serre
subcategory if it satisfies the following conditions. 

\smallskip
(a) $\mathcal B'$ is closed under taking subobjects and quotients in $\mathcal B$.

\smallskip
(b) $\mathcal B'$ is closed under taking extensions in $\mathcal B$. 
\end{defn}

We have noted before that since $\mathcal A$ is locally noetherian, the category $\mathcal A_{fg}$ of finitely 
generated objects is itself an abelian subcategory of $\mathcal A$. 
 We will now use Proposition \ref{P3.7} to show that Serre subcategories of $\mathcal A_{fg}$
correspond to specialization closed subsets of $\mathfrak Spec(\mathcal A)$. 

\smallskip
\begin{thm}\label{P3.9}  There is a bijection : 

\begin{equation*}
Supp : \{\mbox{Serre subcategories of $\mathcal A_{fg}$}\}\forbackarrow{}{} \{\mbox{Specialization closed subsets of
$\mathfrak Spec(\mathcal A)$}\} : Supp^{-1}
\end{equation*}

\end{thm}

\begin{proof} We will have to verify that the collection $\mathfrak T$ of Serre subcategories of $\mathcal A_{fg}$
satisfies the conditions of Proposition \ref{P3.7}. First, we recall that given a short exact sequence 
\begin{equation}
0\longrightarrow M'\longrightarrow M\longrightarrow M''\longrightarrow 0
\end{equation} of objects in $\mathcal A$, we must have $Supp(M)=Supp(M')\cup Supp(M'')$ (see \cite[$\S$ 5.2.2]{Rose2}). 
Accordingly, it is clear from Definition \ref{D3.8} that the  collection $\mathfrak T$ of Serre subcategories of $\mathcal A_{fg}$
satisfies conditions  (1) and (2) appearing in Proposition \ref{P3.7}.

\smallskip It remains to check condition (3) in Proposition \ref{P3.7}. For this, we consider a Serre subcategory $\mathcal C\subseteq \mathcal A_{fg}$
and some $M\in \mathcal C$. If $P$ is a spectral object such that $\langle P\rangle \in Supp(M)$, it follows from
\eqref{supt} that $P\prec M$. Accordingly, $P$ is a subquotient of a direct sum of finitely many copies of 
$M\in \mathcal C$. Since $\mathcal C$ is a Serre subcategory, it follows that $P\in \mathcal C$. 

\end{proof}

\section{Classifying subcategories and restricting the spectrum}

\smallskip
In this section, we will restrict to a subcollection of spectral objects of $\mathcal A$ that satisfy a certain property 
with respect to the objects  in whose support they appear. To motivate the key definitions in this section, we will need the following observation from commutative algebra. 

\begin{thm}\label{P4.1} Let $R$ be a commutative noetherian ring and let $N$ be a spectral object in the abelian category
 $R-Mod$ of $R$-modules. Let $M$ be a finitely generated $R$-module such that $\langle N\rangle \in Supp(M)$. Then,
$Hom_{R-Mod}(M,N)\ne 0$.
\end{thm}

\begin{proof} Since $N$ is spectral, we know from  Section 2 that there is a prime ideal $\mathfrak{p}\subseteq R$ such
that $R/\mathfrak{p}$ is 
equivalent to $N$, i.e., $N\prec R/\mathfrak{p}\prec N$. We consider some nonzero $n\in N$ and denote by $N'\subseteq N$ the submodule
generated by $n$. Since $N$ is spectral, we have an equivalence $N'\approx N$ and hence $N'\approx N\approx R/
\mathfrak{p}$. 

\smallskip
On the other hand, since $N'$ is generated by a single element, we know that $N'$ is isomorphic to the quotient
$R/I$ for some ideal $I\subseteq R$. Then, $R/I\cong N'\approx R/\mathfrak{p}$, i.e., $R/I\prec R/
\mathfrak{p}\prec R/I$. Then, the annihilators
satisfy 
$I=Ann(R/I)\supseteq \mathfrak{p}=Ann(R/\mathfrak{p})\supseteq Ann(R/I)=I$ and hence $I=\mathfrak{p}$. It follows that we have a monomorphism
$R/\mathfrak{p}=R/I\cong N'\hookrightarrow N$. 

\smallskip
Further since $\langle N\rangle \in Supp(M)$, we have $R/\mathfrak{p}\prec N\prec M$ and hence $
\mathfrak{p}\supseteq Ann(M)$. Since $M$
is finitely generated, it follows
from \cite[Lemma 3.5]{SW} that there is a nonzero morphism $\phi:M\longrightarrow R/\mathfrak{p}$. Finally, composing 
$\phi:M\longrightarrow R/\mathfrak{p}$ with the monomorphism $R/\mathfrak{p}\hookrightarrow N$ gives us a nonzero morphism
from $M$ to $N$. 
\end{proof}

\smallskip
We now come back to the category $\mathcal A$ and consider some point $\langle P\rangle \in \mathfrak Spec(\mathcal A)$. We will
say that the point $\langle P\rangle \in \mathfrak Spec(\mathcal A)$ satisfies property $(\ast)$ if given any 
$Q\in Spec(\mathcal A)$ and $M\in \mathcal A_{fg}$ such that $\langle Q\rangle =\langle P\rangle \in Supp(M)$, we must
have $Hom_{\mathcal A}(M,Q)\ne 0$. From Proposition \ref{P4.1}, we see that when $R$ is a commutative noetherian ring, all the points in $\mathfrak Spec(R-Mod)$ 
satisfy this condition $(\ast)$.  

\smallskip
We now consider the full subcategory $\mathcal A'\subseteq \mathcal A_{fg}$ defined as follows:
\begin{equation}
Ob(\mathcal A'):=\{\mbox{$M\in \mathcal A_{fg}$ $\vert$ Every $\langle P\rangle \in Supp(M)$ satisfies condition $(\ast)$}\}
\end{equation} Along with this, we also consider the following specialization closed subset of $\mathfrak Spec(\mathcal A)$:
\begin{equation}
\mathfrak Spec'(\mathcal A):=\underset{M\in \mathcal A'}{\bigcup}\textrm{ }Supp(M)
\end{equation} Since $\mathfrak Spec'(\mathcal A)\subseteq \mathfrak Spec(\mathcal A)$ is specialization closed, it  is clear that the maps $Supp$ and $Supp^{-1}$ in Definition \ref{D3.4} restrict to maps:
\begin{equation*}
\begin{array}{c}
Supp : \{\mbox{Full \& replete subcategories of $\mathcal A'$}\}\longrightarrow \{\mbox{Specialization closed subsets
of $\mathfrak Spec'(\mathcal A)$}\}\\
Supp^{-1}:\{\mbox{Specialization closed subsets
of $\mathfrak Spec'(\mathcal A)$}\}
\longrightarrow \{\mbox{Full \& replete subcategories of $\mathcal A'$}\}\\
\end{array}
\end{equation*}
We also notice here that $\mathcal A'=Supp^{-1}(\mathfrak Spec'(\mathcal A))$ is a Serre subcategory of
$\mathcal A_{fg}$. In particular, $\mathcal A'$  is an abelian subcategory of $\mathcal A_{fg}$ and hence an abelian
subcategory of $\mathcal A$. The following result will help us understand the condition $(\ast)$ and the 
subcategory $\mathcal A'$ better. 

\begin{thm}\label{P4.15}  For every object $M\in \mathcal A'$ and each nonzero subobject $N\subseteq M$, we have
$Hom_{\mathcal A}(M,N)\ne 0$.
\end{thm}

\begin{proof}We consider $M\in \mathcal A'$ and a subobject $0\ne N\subseteq M$.  Since $N\ne 0$, we can choose a spectral object
$P\in Spec(\mathcal A)$ having an inclusion $P\hookrightarrow N$. It is clear that $\langle P\rangle \in Supp(N)
\subseteq Supp(M)$. Hence, $\langle P\rangle$ satisfies condition $(\ast)$ and there exists a morphism
$0\ne \phi:M\longrightarrow P$. Composing $\phi$ with the monomorphism $P\hookrightarrow N$ gives a nonzero
morphism from $M$ to $N$. 
\end{proof}

We also record here the following corollary in the case of modules over a commutative noetherian ring.

\begin{cor}\label{C4.16} Let $R$ be a commutative noetherian ring and let $M$ be a finitely generated
$R$-module. Then, for any nonzero submodule $N\subseteq M$, we have $Hom_{R-Mod}(M,N)\ne 0$. 
\end{cor}

\begin{proof} From Proposition \ref{P4.1}, we know that every point in $\mathfrak Spec(R-Mod)$ satisfies
condition $(\ast)$. The result now follows immediately from Proposition \ref{P4.15}. 

\end{proof}

\begin{thm}\label{P4.2} Let $\mathfrak T$ be a collection of subcategories
of $\mathcal A'$ satisfying the following conditions:

\smallskip
(1) For each specialization closed subset $S\subseteq \mathfrak Spec'(\mathcal A)$, $Supp^{-1}(S)\in \mathfrak T$. 

(2) Each subcategory $\mathcal C\in \mathfrak T$ is a full and replete subcategory of 
$\mathcal A'$   that contains zero and is closed under extensions. 

(3) If $\mathcal C\in\mathfrak T$ and $M\in \mathcal C$, then $P\in \mathcal C$ for each spectral object
$P$ such that $\langle P\rangle \in Supp(M)$. 

\smallskip
Then, the restriction of the map $Supp$ :
\begin{equation}
Supp_{\mathfrak T}:=Supp | \mathfrak T : \mathfrak T\longrightarrow \{\mbox{Specialization closed subsets
of $\mathfrak Spec'(\mathcal A)$}\}
\end{equation} is a bijection having inverse $Supp^{-1}$. 

\end{thm}

\begin{proof} We consider a specialization closed subset $S\subseteq \mathfrak Spec'(\mathcal A)$. Since $\mathfrak Spec'(
\mathcal A)$ is specialization closed in $\mathfrak Spec(\mathcal A)$, it follows that $S$ is specialization closed
in $\mathfrak Spec(\mathcal A)$. By assumption, $Supp^{-1}(S)\in \mathfrak T$. From Proposition \ref{P3.5}, we now
obtain $Supp(Supp^{-1}(S))=S$. 

\smallskip
Conversely, any category $\mathcal C\in \mathfrak T$ satisfies the two conditions in Lemma \ref{L3.6}(b). Further,
since $\mathcal C\subseteq \mathcal A'$, we know that $Supp(\mathcal C)\subseteq \mathfrak Spec'(\mathcal A)$. 
From Lemma \ref{L3.6}(b), it now follows that $Supp^{-1}(Supp(\mathcal C))=\mathcal C$. 

\end{proof} 

\begin{thm}\label{P4.3}  There  is a bijection : 

\begin{equation*}
Supp : \{\mbox{Serre subcategories of $\mathcal A'$}\}\forbackarrow{}{} \{\mbox{Specialization closed subsets of
$\mathfrak Spec'(\mathcal A)$}\} : Supp^{-1}
\end{equation*}

\end{thm}

\begin{proof}
Using the criterion in Proposition \ref{P4.2}, this result follows exactly as in the proof of Proposition \ref{P3.9}. 
\end{proof}

\begin{defn}\label{D4.4} Let $\mathcal B$ be an abelian category. 

\smallskip
(a) A full subcategory $\mathcal B'\subseteq \mathcal B$ is said to be a torison class if
it is closed under extensions and quotients. 

(b) A full abelian subcategory $\mathcal B'\subseteq \mathcal B$ is said to be a thick subcategory
if it is closed under extensions. 

(c) A narrow subcategory $\mathcal B'\subseteq \mathcal B$ is a full subcategory that is closed under extensions and cokernels. 
\end{defn}

 The notion of a narrow subcategory of an abelian category was introduced in \cite{SW}. For the notion of a  thick subcategory, we refer the reader to a standard reference such
as \cite[Tag 02MN]{Stacks}. The notion of a torsion class in Definition \ref{D4.4} is taken from \cite[$\S$ 2]{SW}. In the special
case of abelian categories in which all objects are noetherian, we will show that it agrees with the standard notion of the torsion class of a torsion pair in the literature (see \cite[Chapter 1]{BI}). 

\smallskip For a general abelian category $\mathcal B$, the notions
in Definition \ref{D3.8} and Definition \ref{D4.4}   are related as follows (see \cite[$\S$ 2]{SW}): 

\begin{equation}\label{rel} 
\begin{array}{ccc}
\{\mbox{Serre subcategories of $\mathcal B$}\} & \implies & \{\mbox{Thick subcategories of $\mathcal B$}\} \\
\Downarrow && \Downarrow \\
\{\mbox{Torsion classes of $\mathcal B$}\} & \implies & \{\mbox{Narrow subcategories of $\mathcal B$}\} \\
\end{array}
\end{equation} It is also shown in \cite[$\S$ 2]{SW} that, in general, none of the types of subcategories appearing 
in \eqref{rel} is equivalent to any other. 

\begin{defn}\label{D4.39} (see, for instance, \cite[Chapter 1]{BI} Let $\mathcal B$ be an abelian category. A  torsion
pair $\tau$ on $\mathcal B$ consists of a pair $\tau=(\mathcal T,\mathcal F)$
of replete and full subcategories
of $\mathcal B$ satisfying the following two conditions:

\medskip
(a) Given any objects $X\in \mathcal T$ and 
$Y\in \mathcal F$, we have $Hom_{\mathcal B}(X,Y)=0$. 

\medskip
(b) For any object $M\in \mathcal B$, there exists a short exact
sequence:
\begin{equation}\label{eq2.1}
0\longrightarrow X_M\longrightarrow M\longrightarrow Y^M\longrightarrow 0
\end{equation} in $\mathcal B$ such that $X_M\in \mathcal T$ and 
$Y^M\in \mathcal F$. 

\medskip
The subcategory $\mathcal T$ is referred to as the torsion
class whereas the subcategory $\mathcal F$ is referred to as the
torsion free class of the torsion pair.
\end{defn} 

\smallskip
It is well known in the literature (see, for instance, \cite[Chapter 1]{BI}) that if $\mathcal B$ is an abelian category that is both complete and cocomplete,
any replete and full subcategory $\mathcal C\subseteq \mathcal B$ that is closed under extensions, quotients and coproducts
must arise as the torsion class of some torsion pair on $\mathcal B$. 

\smallskip However, if $\mathcal A$ is a locally noetherian abelian category and $\mathcal A'\subseteq \mathcal A_{fg}$
is as above,  every object in $\mathcal A'$ is finitely generated. Hence the abelian category $\mathcal A'$ does
not contain all coproducts (i.e. it is not cocomplete). Therefore, in order to characterize torsion classes inside
$\mathcal A'$, we will need the following result. 

\begin{thm}\label{P4.391} Let $\mathcal B$ be an abelian category such that every object in $\mathcal B$
is noetherian. Let $\mathcal C\subseteq \mathcal B$ be a full and replete subcategory that is closed under extensions
and quotients. Let  $\mathcal C^\perp\subseteq \mathcal B$  be the full subcategory given by
\begin{equation}\label{eqs4.5}
Ob(\mathcal C^\perp) :=\{\mbox{$N\in \mathcal B$ $\vert$ $Hom_{\mathcal B}(C,N)=0$ for all $C\in \mathcal C$}\}
\end{equation} Then $(\mathcal C,\mathcal C^\perp)$ is a torsion pair on $\mathcal B$. 
\end{thm}

\begin{proof} From \eqref{eqs4.5}, it is clear that the pair $(\mathcal C,\mathcal C^\perp)$ satisfies condition (a)
in Definition \ref{D4.39}. To show that it also satisfies condition (b), we choose some $M\in \mathcal B$.

\smallskip Since $M$ is noetherian, we can pick a subobject $X_M
\subseteq M$ that is a maximal object among the images of morphisms $C\longrightarrow M$ with $C\in \mathcal C$. We note
that it is possible that $X_M=0$. Since $\mathcal C$ is closed under quotients, it is clear that $X_M\in \mathcal C$. 

\smallskip If $\phi':C'\longrightarrow M$ is any morphism with $C'\in \mathcal C$, we now claim 
that $Im(\phi')\subseteq X_M$. Indeed, if $X_M=Im(\phi_0)$ for some morphism $\phi_0:C_0\longrightarrow M$
with $C_0\in \mathcal C$, we can consider the induced map $\phi_0+\phi':C_0\oplus C'\longrightarrow M$. 
It is clear that both $X_M=Im(\phi_0)\subseteq Im(\phi_0+\phi')$ and $Im(\phi')\subseteq Im(\phi_0+\phi')$. Since
$\mathcal C$ is closed under extensions, we know that $C_0\oplus C'\in \mathcal C$. From the maximality
of $X_M$, it follows that $X_M=Im(\phi_0+\phi')$ and hence $Im(\phi')\subseteq Im(\phi_0+\phi')=X_M$. 

\smallskip
It now remains to show that the quotient $M/X_M\in \mathcal C^\perp$. Let $\psi:C''\longrightarrow M/X_M$ be a morphism
with source $C''\in \mathcal C$. Since $\mathcal C$ is closed under quotients, we see that $Im(\psi)\in \mathcal C$. On the other
hand, we can express $Im(\psi)\subseteq M/X_M$ as a quotient $X'/X_M$, where $X_M\subseteq X'\subseteq M$. This 
gives a short exact sequence
\begin{equation}
0\longrightarrow X_M\longrightarrow X'\longrightarrow X'/X_M=Im(\psi)\longrightarrow 0
\end{equation} Since $X_M\in \mathcal C$ and $X'/X_M=Im(\psi)\in \mathcal C$ and $\mathcal C$ is closed
under extensions, we obtain $X'\in \mathcal C$. It follows from the maximality of $X_M$ that $X'=X_M$. Hence, $\psi=0$ and 
the result follows. 
\end{proof}

\begin{thm}\label{P4.5}
There is a bijection : 

\begin{equation*}
Supp : \{\mbox{Torsion classes in $\mathcal A'$}\}\forbackarrow{}{} \{\mbox{Specialization closed subsets of
$\mathfrak Spec'(\mathcal A)$}\} : Supp^{-1}
\end{equation*}

\end{thm}

\begin{proof} We have to show that the collection $\mathfrak T$ of all torsion classes in $\mathcal A'$ satisfies
the three conditions in 
Proposition \ref{P4.2}. For any specialization closed $S\subseteq \mathfrak Spec'(\mathcal A)$, we know that
$Supp^{-1}(S)$ is a Serre subcategory of $\mathcal A'$ and hence a torsion class. By definition,   torsion classes
are closed under extensions. It remains therefore to check that $\mathfrak T$ satisfies condition (3) 
in Proposition \ref{P4.2}. 

\smallskip
Accordingly, let $\mathcal C\in \mathfrak T$ be a torsion class inside the abelian category $\mathcal A'$. We know
that every object in $\mathcal A'\subseteq \mathcal A_{fg}$ is noetherian.
Then, it follows from Proposition \ref{P4.391} that $(\mathcal C,\mathcal C^\perp)$ is a torsion pair on 
$\mathcal A'$,  where $\mathcal C^\perp\subseteq \mathcal A'$ is the full subcategory given by 
\begin{equation}
Ob(\mathcal C^\perp) :=\{\mbox{$N\in \mathcal A'$ $\vert$ $Hom_{\mathcal A'}(C,N)=0$ for all $C\in \mathcal C$}\}
\end{equation}  Since $(\mathcal C,\mathcal C^\perp)$ is a torsion pair, it now follows (see, for instance, 
\cite[Chapter 1]{BI}) that there exist functors $\mathbf T:\mathcal A'\longrightarrow \mathcal C$ and 
$\mathbf F: \mathcal A'\longrightarrow \mathcal C^\perp$ such that for every $X\in \mathcal A'$ we have a short exact sequence
\begin{equation}
0\longrightarrow \mathbf T(X)\longrightarrow X\longrightarrow \mathbf F(X)\longrightarrow 0
\end{equation}
 We now choose 
$M\in \mathcal C$ and  consider some spectral object $P$ such that $\langle P\rangle \in Supp(M)$. Since 
$Supp(M)$ is specialization closed, it follows from Lemma \ref{L3.3} that $Supp(P)=\overline{\langle P\rangle}
\subseteq Supp(M)\subseteq Supp(\mathcal C)$. We will show that $P\in \mathcal C$ by checking that
$\mathbf F(P)=0$. 

\smallskip
Indeed, suppose that $\mathbf F(P)\ne 0$. Then, $Ass(\mathbf F(P))\ne \phi$ and there exists some 
$Q\in Spec(\mathcal A)$ with $Q\subseteq \mathbf F(P)$. We have:
\begin{equation}
\langle Q\rangle \in Ass(\mathbf F(P))\subseteq Supp(\mathbf F(P))\subseteq Supp(P)\subseteq Supp(\mathcal C)
\end{equation} Here, the inclusion $Supp(\mathbf F(P))\subseteq Supp(P)$ follows from the fact that
$P\longrightarrow \mathbf F(P)$ is an epimorphism. Now since $\langle Q\rangle \in Supp(\mathcal C)$, we can 
find $M'\in \mathcal C$ such that $\langle Q\rangle \in Supp(M')$. Since $M'\in \mathcal C\subseteq \mathcal A'$, it follows that we can
find a nonzero morphism $M'\longrightarrow Q$. Composing this morphism with the inclusion $Q\hookrightarrow 
\mathbf F(P)$ gives a nonzero morphism $M'\longrightarrow \mathbf F(P)$. However, this is a contradiction since
$M'\in\mathcal C$ and $\mathbf F(P)\in \mathcal C^\perp$. Hence, $ \mathbf F(P)= 0$ and we must
have $P=\mathbf T(P)\in \mathcal C$. 
\end{proof}

\begin{thm}\label{AShah} Let  $\mathcal A'\subseteq \mathcal A$ be the abelian subcategory as defined above. Then, for a full
and replete subcategory $\mathcal C\subseteq \mathcal A'$ the following conditions are equivalent.

\smallskip
(1) $\mathcal C$ is a Serre subcategory of $\mathcal A'$. 

(2) $\mathcal C$ is a torsion class in $\mathcal A'$. 

(3) $\mathcal C$ is a thick subcategory of $\mathcal A'$. 

(4) $\mathcal C$ is a narrow subcategory of $\mathcal A'$.

\end{thm} 

\begin{proof} (1) $\Leftrightarrow$ (2) : Using Proposition \ref{P4.3} and Proposition \ref{P4.5},  this follows from the fact that both the collection of Serre subcategories of $\mathcal A'$ and the collection of 
torsion classes in $\mathcal A'$ are the image of the map  
\begin{equation*}
Supp^{-1}:\{\mbox{Specialization closed subsets
of $\mathfrak Spec'(\mathcal A)$}\}
\longrightarrow \{\mbox{Full \& replete subcategories of $\mathcal A'$}\}
\end{equation*}

\smallskip (4) $\Rightarrow$ (2) : Let $\mathcal C\subseteq \mathcal A'$ be a narrow subcategory. Then,
$\mathcal C$ is closed under cokernels.  If $\mathcal C$ is not a torsion class, there exists an epimorphism 
$\pi_0: M_0\longrightarrow N$ with $M_0\in \mathcal C$ and $N\notin\mathcal C$. Then, $0\ne Ker(\pi_0)
\subseteq M_0$ and using Proposition \ref{P4.15}, we get a nonzero morphism $\epsilon_0:
M_0\longrightarrow Ker(\pi_0)$. We now set 
\begin{equation}\label{eq4.11s}
M_1:=Cok(M_0\overset{\epsilon_0}{\longrightarrow}Ker(\pi_0)\hookrightarrow M_0)
\end{equation} and denote by $\rho_0$ the canonical epimorphism $\rho_0:M_0\longrightarrow M_1$ to the cokernel. 
Since $\mathcal C$ is closed under cokernels, we must have $M_1\in \mathcal C$. From 
\eqref{eq4.11s}, it is also clear that $Ker(\rho_0)\ne 0$ and the morphism $\pi_0:M_0\longrightarrow N$ factors through $M_1$. This gives
us an epimorphism $\pi_1:M_1\longrightarrow N$ such that $\pi_1\circ\rho_0=\pi_0$. 

\smallskip
Since $N\notin \mathcal C$ and $M_1\in \mathcal C$, the epimorphism $\pi_1:M_1\longrightarrow N$ must
satisfy $Ker(\pi_1)\ne 0$. Repeating the above process over and over again, we can obtain for each $n\geq 1$ an epimorphism 
$\pi_n:M_n\longrightarrow N$ with $M_n\in \mathcal C$ and an epimorphism $\rho_{n-1}:
M_{n-1}\longrightarrow M_n$ with $Ker(\rho_{n-1})\ne 0$ such that $\pi_n\circ \rho_{n-1}=\pi_{n-1}$. Since
$M_n\in \mathcal C$ and $N\notin\mathcal C$, we must have $Ker(\pi_n)\ne 0$. 

\smallskip
We now consider the increasing chain of subobjects $Ker(\rho_0)\subseteq Ker(\rho_1\circ \rho_0)
\subseteq Ker(\rho_2\circ\rho_1\circ \rho_0)\subseteq \dots$ of $M_0$.  A simple application of the Snake Lemma
shows that
\begin{equation}
\frac{Ker(\rho_n\circ \rho_{n-1}\circ ...\circ \rho_0)}{Ker(\rho_{n-1}\circ ...\circ \rho_0)}\cong Ker(\rho_n)\ne 0\qquad\forall\textrm{ }n\geq 1
\end{equation} This shows that the chain is strictly increasing, which contradicts the fact that $M_0\in \mathcal C
\subseteq \mathcal A'\subseteq \mathcal A_{fg}$
is Noetherian.

\smallskip 
(2) $\Rightarrow$ (4) : This follows from the diagram of relations in \eqref{rel}. 

\smallskip
(3) $\Leftrightarrow$ (4) : We have already shown that (1) $\Leftrightarrow$ (2) 
$\Leftrightarrow$ (4). The result now follows by noting that the diagram of relations in \eqref{rel} gives 
(3) $\Rightarrow$ (4) and that (1) $\Rightarrow$ (3). 

\smallskip
\end{proof}  

\section{Subsets of the spectrum and torsion theories on $\mathcal A$}

\smallskip
 In the previous section, we have shown that certain specialization closed subsets of
$\mathfrak Spec(\mathcal A)$ correspond to Serre subcategories, torsion classes as well as thick and narrow
subcategories of a certain abelian subcategory $\mathcal A'\subseteq \mathcal A_{fg}$. 

\smallskip
In this section, we will be concerned with showing that specialization closed subsets in the spectrum
$\mathfrak Spec(\mathcal A)$ of the locally noetherian category correspond to hereditary torsion pairs 
on $\mathcal A$. 
If $R$ is a  commutative noetherian ring, it is shown in \cite[Proposition 2.3]{Four} (see also \cite[VI]{Bo} and 
\cite[Theorem 4.1]{Taka}) that there is a one-one correspondence between
specialization closed subsets of the Zariski spectrum $Spec(R)$ and hereditary torsion pairs on the category of
$R$-modules. The authors in \cite{Four} also describe how these torsion pairs relate to   injective hulls of the quotients
$R/\mathfrak{p}$ as $\mathfrak{p}$ varies over all the prime ideals of $R$. Accordingly, we will  describe the 
hereditary torsion pairs on $\mathcal A$ and their connection to  injective hulls of spectral objects of $\mathcal A$.

\begin{thm}\label{P5.1} Let $S\subseteq \mathfrak Spec(\mathcal A)$ be a specialization closed subset. Then, the pair
$(\mathcal T(S),\mathcal F(S))$ of subcategories defined as follows
\begin{equation}\label{eq5.1}
\mathcal T(S):=\{\mbox{$M\in \mathcal A$ $\vert$ $Supp(M)\subseteq S$}\}\qquad 
\mathcal F(S):=\{\mbox{$M\in \mathcal A$ $\vert$ $Ass(M)\cap S=\phi$}\}
\end{equation} gives  a torsion pair on the abelian category $\mathcal A$. Further, the torsion pair
$(\mathcal T(S),\mathcal F(S))$ is hereditary, i.e., the torsion class $\mathcal T(S)$ is closed under subobjects.
\end{thm}

\begin{proof} Given a short exact sequence in $\mathcal A$
\begin{equation}
0\longrightarrow M'\longrightarrow M\longrightarrow M''\longrightarrow 0
\end{equation} it follows from \cite[$\S$ 5.2.2]{Rose2} that $Supp(M)=Supp(M')\cup Supp(M'')$. Further, if 
$\{M_i\}_{i\in I}$ is a family of objects in $\mathcal A$, it follows from \cite[$\S$ 5.2.3]{Rose2} that 
\begin{equation}
Supp\left(\underset{i\in I}{\bigoplus}M_i\right)=\underset{i\in I}{\bigcup}Supp(M_i)
\end{equation} Combining these with the definition in \eqref{eq5.1}, it is immediately clear that $\mathcal T(S)$
is closed under subobjects, extensions, quotients and coproducts. Since $\mathcal A$ is a Grothendieck category, it is  both complete and
cocomplete. It now follows from \cite[Chapter 1]{BI} that $\mathcal T(S)$ is the torsion class of a hereditary
torsion pair on $\mathcal A$. 

\smallskip
Suppose now that there is a morphism $\psi:M\longrightarrow N$ with $M\in \mathcal T(S)$ and 
$N\in \mathcal F(S)$. Then, $Ass(Im(\psi))\subseteq Ass(N)$ (see \cite[$\S$ 8.2]{Rose2}). On the other hand,
we have $Ass(Im(\psi))\subseteq Supp(Im(\psi))\subseteq Supp(M)\subseteq S$. Since $Ass(N)\cap S=\phi$,
we get $Ass(Im(\psi))=\phi$ and hence $Im(\psi)=0$. 

\smallskip
Conversely, suppose that $N\in \mathcal A$ is such that $Hom_{\mathcal A}(M,N)=0$ for all 
$M\in \mathcal T(S)$. Suppose that $Ass(N)\cap S\ne \phi$, i.e., there exists a spectral object
$P\in Spec(\mathcal A)$ along with a monomorphism $P\hookrightarrow N$ such that $\langle P\rangle \in 
S$. Since $S$ is specialization closed, we obtain from Lemma \ref{L3.3} that $\overline{\langle P\rangle }=Supp(P)\subseteq S$. 
Then, $P\in \mathcal T(S)$ and the monomorphism $P\hookrightarrow N$ must be zero, which is a contradiction.
\end{proof}

\begin{thm}\label{P5corr} The association $
S\mapsto \tau(S):=(\mathcal T(S),\mathcal F(S))
$ defines a one-one correspondence between hereditary torsion pairs on 
$\mathcal A$ and specialization closed subsets of the spectrum $\mathfrak Spec(\mathcal A)$. 
\end{thm}

\begin{proof} If $\mathcal C\subseteq \mathcal A$ is the torsion class of a hereditary torsion pair
on $\mathcal A$, we associate to $\mathcal C$ a specialization closed subset of $\mathfrak Spec(\mathcal A)$
as follows:
\begin{equation}\label{eq5.4l}
\mathcal C\mapsto Z(\mathcal C):=\underset{M\in \mathcal C\cap \mathcal A_{fg}}{\bigcup}\textrm{ }Supp(M)
\end{equation} We claim that the associations $S\mapsto \mathcal T(S)$ and $\mathcal C\mapsto Z(\mathcal C)$ are inverse
to each other. 

\smallskip
We begin with a specialization closed subset $S\subseteq \mathfrak Spec(\mathcal A)$. By definition, we know that
\begin{equation}\label{eq5.5l}
\mathcal T(S)=\{\mbox{$M\in \mathcal A$ $\vert$ $Supp(M)\subseteq S$}\}
\end{equation} From \eqref{eq5.5l}, it is clear that $Z(\mathcal T(S))\subseteq S$. Conversely, let $P\in Spec(\mathcal A)$
be a spectral object with $\langle P\rangle \in S$. Since $S$ is specialization closed, we
see that $Supp(P)=\overline{\langle P\rangle}\subseteq S$ and hence $P\in \mathcal T(S)$. From Lemma \ref{L3.2}, we know that the spectral object
$P$ is finitely generated and hence $P\in \mathcal T(S)\cap \mathcal A_{fg}$. Since $\langle P\rangle 
\in Supp(P)$, it follows from  \eqref{eq5.4l}  that $S\subseteq Z(\mathcal T(S))$. 

\smallskip
On the other hand, let $\mathcal C$ be the torsion class of a hereditary torsion pair on $\mathcal A$. From the definitions
in \eqref{eq5.4l} and \eqref{eq5.5l}, it is immediate that $\mathcal C\cap\mathcal A_{fg}\subseteq \mathcal T(Z(\mathcal C))$. 
We choose some $M\in \mathcal C$. 
Since $\mathcal A$ is locally finitely generated (see Section 2), the object $M$ may be expressed as the filtered colimit
\begin{equation}
M=\underset{M'\in \mathcal A_{fg},M'\subseteq M}{\varinjlim}\textrm{ }M'
\end{equation} Since $\mathcal C$ is the torsion class of a hereditary torsion pair, it is closed under
subobjects and hence every finitely generated subobject $M'\subseteq M$ lies in $\mathcal C$ and hence
in $\mathcal C\cap \mathcal A_{fg}$. Then, $M'\in \mathcal T(Z(\mathcal C))$ for each finitely generated
$M'\subseteq M$. But $\mathcal T(Z(\mathcal C))$ being a torsion class, it is closed under colimits and 
we see that $M\in \mathcal T(Z(\mathcal C))$. Hence, $\mathcal C\subseteq \mathcal T(Z(\mathcal C))$.

\smallskip
It remains to show that $\mathcal T(Z(\mathcal C))\subseteq \mathcal C$. First, we consider some spectral
object $P\in Spec(\mathcal A)$ such that $P\in \mathcal T(Z(\mathcal C))$.  From the definitions 
in \eqref{eq5.4l} and \eqref{eq5.5l}, we obtain : 
\begin{equation}
\overline{\langle P\rangle}=Supp(P)\subseteq  \underset{M\in \mathcal C\cap \mathcal A_{fg}}{\bigcup}\textrm{ }Supp(M)
\end{equation} Hence, we can find some $M\in \mathcal C\cap\mathcal A_{fg}$ with $\langle P\rangle \in Supp(M)$, i.e.,
$P\prec M$. But $\mathcal C$ is closed under subobjects, quotients and coproducts and hence $P\prec M$
implies that $P\in \mathcal C$. 

\smallskip We now choose some 
$N\in \mathcal T(Z(\mathcal C))$ and consider a finitely generated subobject $N'\subseteq N$. From Proposition 
\ref{P5.1}, we already know that $\mathcal T(Z(\mathcal C))$ is closed under subobjects and hence 
$N'\in \mathcal T(Z(\mathcal C))$. From the definitions 
in \eqref{eq5.4l} and \eqref{eq5.5l}, we now obtain : 
\begin{equation}
Supp(N')\subseteq Supp(N)\subseteq \underset{M\in \mathcal C\cap \mathcal A_{fg}}{\bigcup}\textrm{ }Supp(M)
\end{equation} From \cite[$\S$ 1.6.4.1]{Rose1}, we know that the finitely generated object $N'$ can be filtered as follows
\begin{equation}
N'=N'_k\supseteq N'_{k-1}\supseteq \dots \supseteq N'_0=0
\end{equation} with each successive quotient $N'_i/N'_{i-1}\in Spec(\mathcal A)$. Again since $\mathcal T(Z(\mathcal C))$
is closed under subobjects and quotients, we see that each spectral object $N'_i/N'_{i-1}\in \mathcal T(Z(\mathcal C))$. It
follows that each $N'_i/N'_{i-1}\in \mathcal C$. Since $\mathcal C$ is closed under extensions, we can now  show by a succession
of short exact sequences that $N'\in \mathcal C$. Finally, since the object $N$ is the colimit of its finitely generated
subobjects and $\mathcal C$ is closed under colimits, it follows that $N\in \mathcal C$. Hence, $\mathcal T(Z(\mathcal C))\subseteq \mathcal C$ and the result follows. 
\end{proof}

Since $\mathcal A$ is a Grothendieck category, it is well known that every object $M\in \mathcal A$ has an injective hull 
(see, for instance, \cite[$\S$ V.2]{Bo}),
which we denote by $E(M)$. We will now describe the categories $\mathcal T(S)$ and $\mathcal F(S)$ in terms of
injective hulls of spectral objects of $\mathcal A$.

\begin{thm}\label{P5.3} Let $S\subseteq \mathfrak Spec(\mathcal A)$ be a specialization closed subset. We set 
\begin{equation}
\begin{array}{c}
\mathcal T'(S)=\{\mbox{$M\in \mathcal A$ $\vert$ $Hom_{\mathcal A}(M,E(P))=0$ for all $P\in Spec(A)$ with $
\langle P\rangle \notin S$}\} \\
\mathcal F'(S)=\{\mbox{$M\in\mathcal A$ $\vert$ $Hom_{\mathcal A}(P,M)=0$ for all $P
\in Spec(\mathcal A)$ with $\langle P\rangle \in S$}\}\\
\end{array}
\end{equation}
Then, $\mathcal T(S)=\mathcal T'(S)$ and $\mathcal F(S)=\mathcal F'(S)$. 
\end{thm}

\begin{proof} If $P\in Spec(\mathcal A)$ is such that $\langle P\rangle \in S$, it follows from Lemma \ref{L3.3} and the
fact that $S$ is specialization closed that $Supp(P)=\overline{\langle P\rangle}\subseteq S$. Then, $P\in \mathcal T(S)$ and it
follows that $Hom_{\mathcal A}(P,M)=0$ for any $M\in \mathcal F(S)$. Hence, $\mathcal F(S)\subseteq \mathcal F'(S)$.

\smallskip
Conversely, we consider $M\in \mathcal F'(S)$. If $M\ne 0$, we  consider some 
$P\in Spec(\mathcal A)$ along with a monomorphism $P\hookrightarrow M$. Then, $\langle P\rangle \notin S$, i.e.,
$Ass(M)\cap S=\phi$. This gives $\mathcal F'(S)\subseteq \mathcal F(S)$ and hence $\mathcal F(S)=\mathcal F'(S)$. 

\smallskip
On the other hand, let us consider $M\in \mathcal T(S)$ along with $P\in Spec(\mathcal A)$ with 
$\langle P\rangle \notin S$. Suppose there exists a nonzero morphism $\phi: M\longrightarrow E(P)$. We set
$I:=Im(\phi)\subseteq E(P)$. Since $P\subseteq E(P)$ is an essential subobject, this gives a nonzero subobject
$I\cap P\subseteq P$. Since $P$ is spectral, it follows that $\langle P\rangle =\langle I\cap P\rangle$. The morphism
$\phi:M\longrightarrow E(P)$ induces an epimorphism to its image $I\subseteq E(P)$ which we denote
by $\phi':M\longrightarrow I$. We now form a pullback square:
\begin{equation}
\begin{CD}
N @>>> I\cap P \\
@VVV @VVV \\
M @>\phi' >> I \\
\end{CD}
\end{equation} Since $\mathcal A$ is an abelian category, the pullback of $\phi':M\longrightarrow I$ gives an 
epimorphism $N\twoheadrightarrow I\cap P$. It is also clear that $N\subseteq M$. Then, $Supp(I\cap P)
\subseteq Supp(N)\subseteq Supp(M)\subseteq S$ since $M\in \mathcal T(S)$. However, since $\langle P\rangle =\langle I\cap P\rangle$, we obtain 
\begin{equation}
\langle P\rangle \in \overline{\langle P\rangle}=\overline{\langle I\cap P\rangle}=Supp(I\cap P)\subseteq S
\end{equation} This contradicts the assumption that $\langle P\rangle \notin S$. Hence, $\mathcal T(S)\subseteq 
\mathcal T'(S)$. 

\smallskip
Finally, we consider some $M\in \mathcal T'(S)$. Suppose $M\notin\mathcal T(S)$, i.e., $Supp(M)\not\subseteq S$. 
In other words, we can choose $P\in Spec(\mathcal A)$ with $P\prec M$ such that $\langle P\rangle \notin S$. Since
$P\prec M$, it follows that we can choose  a direct sum $M^n$ of $n$-copies of $M$ and a subobject 
$N\subseteq M^n$ along with an epimorphism $\pi:N\longrightarrow P$. Composing $\pi:N\longrightarrow P$
with the inclusion $P\hookrightarrow E(P)$ gives a nonzero morphism $\phi:N\longrightarrow E(P)$. However,
since $E(P)$ is injective, the morphism $\phi:N\longrightarrow E(P)$ extends to a nonzero morphism $M^n\longrightarrow E(P)$. 
This gives $Hom_{\mathcal A}(M^n,E(P))\ne 0$, whence it follows that $Hom_{\mathcal A}(M,E(P))\ne 0$. This is a contradiction.
Hence, $\mathcal T'(S)\subseteq \mathcal T(S)$.

\end{proof}

\begin{cor}\label{C5.4}  Let $S\subseteq \mathfrak Spec(\mathcal A)$ be a specialization closed subset. Then,
 for any $Q\in Spec(\mathcal A)$ with 
$\langle Q\rangle \notin S$, the injective hull $E(Q)$ lies in $\mathcal F(S)$. 
\end{cor}

\begin{proof} We consider $Q\in Spec(\mathcal A)$ with $\langle Q\rangle \notin S$. From Proposition 
\ref{P5.3}, we know that every object $M\in \mathcal T(S)=\mathcal T'(S)$ satisfies 
$Hom_{\mathcal A}(M,E(Q))=0$. Hence, $E(Q)$ lies in the torsion free class $\mathcal F(S)$. 

\end{proof}

\begin{cor}\label{C5.5} Let $S\subseteq \mathfrak Spec(\mathcal A)$ be a specialization closed subset and let
$\tau(S)=(\mathcal T(S),\mathcal F(S))$ be the hereditary torsion pair on $\mathcal A$ associated to $S$. Then, 
\begin{equation}
S=\mathfrak Spec(\mathcal A)\backslash Ass(\mathcal F(S))=\mathfrak Spec(\mathcal A)\backslash 
\left(\underset{M\in \mathcal F(S)}{\bigcup}\textrm{ }Ass(M)\right)
\end{equation}
\end{cor}
\begin{proof} From Proposition \ref{P5.1}, we know that $\mathcal F(S):=\{\mbox{$M\in \mathcal A$ $\vert$ $Ass(M)\cap S=\phi$}\}$ and hence $S\cap Ass(\mathcal F(S))=\phi$. Conversely, if $Q\in Spec(\mathcal A)$ is such that
$\langle Q\rangle \notin S$, then $Ass(Q)\cap S=\{\langle Q\rangle\}\cap S=\phi$. Hence, $Q\in \mathcal F(S)$ and 
$\langle Q\rangle \in Ass(Q)\subseteq Ass(\mathcal F(S))$. 

\end{proof}

For a specialization closed subset $S\subseteq \mathfrak Spec(\mathcal A)$, we would like to show that the subcategory
$\mathcal T(S)$ contains all the injective hulls $E(P)$ with $P\in Spec(\mathcal A)$ such that $\langle P
\rangle \in S$. However, for this we will first need to do some commutative algebra in the framework of locally noetherian abelian 
categories.

\smallskip
We recall (see \cite[$\S$ 2.3.3]{Rose2}) that for any $P\in Spec(\mathcal A)$, the full subcategory $\langle P
\rangle$ is a Serre subcategory of $\mathcal A$. Hence, we may construct the quotient $\mathcal A/\langle P
\rangle$ along with the canonical functor $L_{\langle P\rangle}:\mathcal A\longrightarrow \mathcal A/\langle P\rangle$ which
is exact (see \cite[Proposition III.1]{Gab}). Further, since $\mathcal A$ is a Grothendieck category, the functor
\begin{equation}
L_{\langle P\rangle}:\mathcal A\longrightarrow \mathcal A/\langle P\rangle
\end{equation} is a flat localization (see \cite[$\S$ 2.4.8.2]{Rose2}) of  $\mathcal A$, i.e., 
$\langle P\rangle$ is a localizing subcategory of $\mathcal A$ (see Gabriel \cite[$\S$ III.2]{Gab}). In particular, this implies the following: 

\smallskip
(1) The exact functor $L_{\langle P\rangle}:\mathcal A\longrightarrow \mathcal A/\langle P\rangle$ admits a right adjoint
$i_{\langle P\rangle }:\mathcal A/\langle P\rangle \longrightarrow \mathcal A$ which is referred to as the section functor. 

\smallskip 
(2) The section functor $i_{\langle P\rangle}$ is fully faithful. Further, the counit $L_{\langle P\rangle}\circ i_{\langle P\rangle}\longrightarrow id_{\mathcal 
A/\langle P\rangle}$ of the adjunction is an isomorphism (see \cite[Proposition III.3]{Gab}). 

\smallskip
(3) An object $M\in \mathcal A$ is said to be $\langle P\rangle$-closed (see \cite[$\S$ III.2]{Gab}) if the morphism $M\longrightarrow 
i_{\langle P\rangle}\circ L_{\langle P\rangle}(M)$ induced by the unit of the adjunction is an isomorphism. Then, 
$i_{\langle P\rangle}$ induces an equivalence between $\mathcal A/\langle P\rangle$ and the full subcategory
of $\langle P\rangle$-closed objects in $\mathcal A$. 

\smallskip
(5) If $\{G_i\}_{i\in I}$ is a generating family for $\mathcal A$, then $\{L_{\langle P\rangle}(G_i)\}_{i\in I}$ is a generating
family for $\mathcal A/\langle P\rangle$ (see \cite[Lemme III.4]{Gab}). Further,  $\mathcal A/\langle P\rangle$ is a 
Grothendieck category (see \cite[Proposition III.9]{Gab}). 

\smallskip
(6) Since $\mathcal A$ is locally noetherian, the quotient $\mathcal A/\langle P\rangle$ is also a locally noetherian category
(see \cite[Corollaire III.1]{Gab}).

\begin{lem}\label{L5.39} Let $P\in Spec(\mathcal A)$ be a spectral object. Then, every nonzero object
in $\mathcal A/\langle P\rangle$ has an associated point.
\end{lem}

\begin{proof}
We pick  a nonzero object $N\in \mathcal A/\langle P\rangle$ and consider $i_{\langle P\rangle}(N)\in \mathcal A$.
Since $i_{\langle P\rangle}$ is full and faithful, we must have $i_{\langle P\rangle}(N)\ne 0$. Then, we can pick
some $Q\in Spec(\mathcal A)$ with a monomorphism $Q\hookrightarrow i_{\langle P\rangle}(N)$. We know that
$L_{\langle P\rangle}( i_{\langle P\rangle}(N))=N$. Then, from 
\cite[$\S$ 8.5.4]{Rose2}, it follows that if $Q\notin Ker(L_{\langle P\rangle})$, we must have $\langle L_{\langle P
\rangle}(Q)\rangle\in Ass(N)$. 

\smallskip
It remains therefore to verify that $Q\notin Ker(L_{\langle P\rangle})$. We notice that the image of the monomorphism
$Q\hookrightarrow i_{\langle P\rangle}(N)$ under the isomorphism $Hom_{\mathcal A/\langle P\rangle}(L_{\langle P\rangle}(Q),
N)\cong Hom_{\mathcal A}(Q,i_{\langle P\rangle}(N))$ gives a nonzero  morphism $L_{\langle P\rangle}(Q)\longrightarrow 
N$ in $\mathcal A/\langle P\rangle$. Hence, $L_{\langle P\rangle}(Q)\ne 0$. 
\end{proof}

\smallskip
The next two results will describe the spectrum $\mathfrak Spec(\mathcal A/\langle P\rangle)$ as a subset of $Spec(\mathcal A)$.

\begin{lem}\label{siL5.37}  Let $P\in Spec(\mathcal A)$ be a spectral object. Then, $L_{\langle P\rangle}:\mathcal A\longrightarrow \mathcal A/\langle P\rangle$
restricts to a map $L_{\langle P\rangle}:Spec(\mathcal A)-Ker(L_{\langle P\rangle})\longrightarrow Spec(\mathcal A/\langle P\rangle)$. The subset 
$Spec(\mathcal A)-Ker(L_{\langle P\rangle})\subseteq Spec(\mathcal A)$ is a union of equivalence classes of objects of $Spec(\mathcal A)$. Further,  $L_{\langle P\rangle}:Spec(\mathcal A)-Ker(L_{\langle P\rangle})\longrightarrow Spec(\mathcal A/\langle P\rangle)$ descends to a map
on equivalence classes
\begin{equation}
\begin{CD}
Spec(\mathcal A)-Ker(L_{\langle P\rangle}) @>L_{\langle P\rangle}>> Spec(\mathcal A/\langle P\rangle)\\
@VVV @VVV \\
\mathfrak Spec(\mathcal A)-Ker(L_{\langle P\rangle}) @>\overline{L}_{\langle P\rangle}>> \mathfrak Spec(\mathcal A/\langle P\rangle)\\
\end{CD}
\end{equation}

\end{lem}

\begin{proof} Let $Q\in Spec(\mathcal A)-Ker(L_{\langle P\rangle})$. Then, $\langle Q\rangle \in Ass(Q)-Ker(L_{\langle P\rangle})$ and it follows from \cite[$\S$ 8.5.4]{Rose2} that
$L_{\langle P\rangle}(Q)\in Spec(\mathcal A/\langle P\rangle)$. 

\smallskip
Next, we observe that  if $\langle M\rangle =\langle N\rangle$ for objects $M$, $N\in \mathcal A$, we must have
$\langle L_{\langle P\rangle }(M)\rangle =\langle L_{\langle P\rangle}(N)\rangle$ using the fact that $L$ is exact. In particular,  if $M\in Ker(L_{\langle P\rangle})$, i.e.,
$L_{\langle P\rangle}(M)=0$, we must also have $L_{\langle P\rangle}(N)=0$.   Hence, the subset 
$Spec(\mathcal A)-Ker(L_{\langle P\rangle})\subseteq Spec(\mathcal A)$ is a union of equivalence classes of objects of $Spec(\mathcal A)$. The same reasoning also shows that the map $L_{\langle P\rangle}$ descends to a map on equivalence classes.
\end{proof}

\begin{thm}\label{siP5.38} Let $P\in Spec(\mathcal A)$ be a spectral object. Then, we have a bijection $\overline{L}_{\langle P\rangle}: \mathfrak Spec(\mathcal A)-Ker(L_{\langle P\rangle})
\longrightarrow  \mathfrak Spec(\mathcal A/\langle P\rangle)$. In particular, $ \mathfrak Spec(\mathcal A/\langle P\rangle)$ may be treated as a subset of
$\mathfrak Spec(\mathcal A)$. 
\end{thm}

\begin{proof} We will first show  that $\overline{L}_{\langle P\rangle}$ is one-one. We consider therefore $Q_1$, $Q_2\in Spec(\mathcal A)-Ker(L_{\langle P\rangle})$ such that
$\langle L_{\langle P\rangle}(Q_1)\rangle=\langle L_{\langle P\rangle}(Q_2)\rangle$. Since $ L_{\langle P\rangle}(Q_1)\ne 0$, we have $P\prec Q_1$ or $\langle P\rangle \subseteq 
\langle Q_1\rangle$. From Lemma \ref{siL5.37}, we know that $L_{\langle P\rangle}(Q_1)\in Spec(\mathcal A/\langle P\rangle)$. This allows us to factor $L_{\langle Q_1\rangle}: \mathcal A\longrightarrow \mathcal A/\langle Q_1\rangle $ as a composition of exact
localization functors $L_{\langle Q_1\rangle}=L'_{\langle Q_1\rangle}\circ L_{\langle P\rangle}$.  Applying Lemma \ref{siL5.37} again, we get commutative diagrams
\begin{equation}\label{sieq5.16}
\begin{array}{c}
\begin{tikzcd}
Spec(\mathcal A)-Ker(L_{\langle Q_1\rangle}) \arrow{r}{L_{\langle Q_1\rangle}} \arrow[swap]{d}{L_{\langle P\rangle}} & Spec(\mathcal A/\langle Q_1\rangle) \\
Spec(\mathcal A/\langle P\rangle) - Ker(L'_{\langle Q_1\rangle}) \arrow{ur}[swap]{L'_{\langle Q_1\rangle}} 
\end{tikzcd} \\  \\ \\

\begin{tikzcd}
\mathfrak Spec(\mathcal A)-Ker(L_{\langle Q_1\rangle}) \arrow{r}{\overline{L}_{\langle Q_1\rangle}} \arrow[swap]{d}{\overline{L}_{\langle P\rangle}} & \mathfrak Spec(\mathcal A/\langle Q_1\rangle) \\
\mathfrak Spec(\mathcal A/\langle P\rangle) - Ker(L'_{\langle Q_1\rangle}) \arrow{ur}[swap]{\overline{L'}_{\langle Q_1\rangle}} 
\end{tikzcd}
\end{array}
\end{equation} From \eqref{sieq5.16}, we see that
\begin{equation*}
0\ne \langle L_{\langle Q_1\rangle}(Q_1)\rangle=\overline{L}_{\langle Q_1\rangle}\langle Q_1\rangle =\overline{L'}_{\langle Q_1\rangle}\overline{L}_{\langle P\rangle}\langle Q_1\rangle =
\overline{L'}_{\langle Q_1\rangle} \langle L_{\langle P\rangle}(Q_1)\rangle=\overline{L'}_{\langle Q_1\rangle} \langle L_{\langle P\rangle}(Q_2)\rangle= \langle L_{\langle Q_1\rangle}(Q_2)\rangle
\end{equation*} Then, $ L_{\langle Q_1\rangle}(Q_2)\ne 0$, which gives  $Q_1\prec Q_2$. We can similarly show that $Q_2\prec Q_1$
and it follows that $\langle Q_1\rangle =\langle Q_2\rangle$. 

\smallskip
It remains to show  that $\overline{L}_{\langle P\rangle}$ is onto. For this, we consider $Q\in Spec(\mathcal A/\langle P\rangle)$. Then, $Q\ne 0$. Since $i_{\langle P\rangle}$
is fully faithful, we know that $i_{\langle P\rangle}(Q)\ne 0$ in $\mathcal A$. By the assumption on $\mathcal A$, we may choose $u:Q'\hookrightarrow i_{\langle P\rangle}(Q)$ with $Q'\in Spec(\mathcal A)$ such that
$\langle Q'\rangle \in Ass(i_{\langle P\rangle}(Q))$. We now have a commutative diagram
\begin{equation}\label{sieq5.17}
\begin{CD}
Q' @>mono>u> i_{\langle P\rangle}Q \\
@VVV @V\approx VV \\
i_{\langle P\rangle}L_{\langle P\rangle}Q' @>i_{\langle P\rangle}L_{\langle P\rangle}(u)>>i_{\langle P\rangle}L_{\langle P\rangle} i_{\langle P\rangle}Q \\
\end{CD}
\end{equation} Since $i_{\langle P\rangle}Q$ is $\langle P\rangle$-closed, the right vertical map in \eqref{sieq5.17} is an isomorphism. It now follows from 
\eqref{sieq5.17} that $Q'\longrightarrow i_{\langle P\rangle}L_{\langle P\rangle}Q'$ is a monomorphism. In particular, $L_{\langle P\rangle}Q'\ne 0$. Since $L_{\langle P\rangle}$
is exact and $L_{\langle P\rangle}\circ i_{\langle P\rangle}=id$, we now obtain a monomorphism $0\ne L_{\langle P\rangle}Q'\hookrightarrow Q$ in $\mathcal A/\langle P\rangle$. Since 
$Q\in Spec(\mathcal A/\langle P\rangle)$, we now have $\langle L_{\langle P\rangle}Q'\rangle =\langle Q\rangle$. This proves the result.

\end{proof}

\begin{thm}\label{siP5.9vk}   Let $Q\in Spec(\mathcal A)$ be a spectral object. Then, for any object $M\in \mathcal A$, we have
\begin{equation}
Ass(M)\cap \mathfrak Spec(\mathcal A/\langle Q\rangle)\subseteq Ass(L_{\langle Q\rangle}(M))
\end{equation} where $\mathfrak Spec(\mathcal A/\langle Q\rangle)$ is treated as a subset of $\mathfrak Spec(\mathcal A)$. 
\end{thm}

\begin{proof}
We consider a monomorphism $P\hookrightarrow M$ with $P\in Spec(\mathcal A)$ such that $\langle P\rangle \in Ass(M)\cap \mathfrak Spec(\mathcal A/\langle Q\rangle)$. Since 
$\langle P\rangle \in  \mathfrak Spec(\mathcal A/\langle Q\rangle)$, it follows from the bijection described in Proposition \ref{siP5.38} that
$P\notin Ker(L_{\langle Q\rangle})$. Then, it follows from \cite[III.8.5.4]{R3} that $\langle L_{\langle Q\rangle}(P)\rangle\in Ass(L_{\langle Q\rangle}(M))$ and the result follows.
\end{proof}

Following \cite[V.C2.3]{R3}, we now consider the set $LAss(M)$
of locally associated points of $M\in \mathcal A$, defined as follows:
\begin{equation*}
LAss(M)=\{\mbox{$\langle P\rangle\in \mathfrak Spec(\mathcal A)$ $\vert$ $P\in Spec(\mathcal A)$ and $\langle L_{\langle P\rangle}(P)\rangle\in Ass(L_{\langle P\rangle}(M))$ }\}
\end{equation*} We now have the following result.

\begin{thm}\label{sIPx5.10l}  Let $M\in \mathcal A$ be such that $LAss(M)=Ass(M)$. Then, for any spectral object
$Q$ of $\mathcal A$, we must have
\begin{equation}
Ass(L_{\langle Q\rangle}(M))=Ass(M)\cap \mathfrak Spec(\mathcal A/\langle Q\rangle)
\end{equation} 
\end{thm}

\begin{proof} From Proposition \ref{siP5.9vk}, we already know that $Ass(M)\cap \mathfrak Spec(\mathcal A/\langle Q\rangle)\subseteq Ass(L_{\langle Q\rangle}(M))$. By assumption, we have  $LAss(M)=Ass(M)$. It suffices therefore to show that $ Ass(L_{\langle Q\rangle}(M))\subseteq LAss(M)$. 

\smallskip
We consider therefore $P'\in Spec(\mathcal A/\langle Q\rangle)$ along with a monomorphism $P'\hookrightarrow L_{\langle Q\rangle}(M)$. Using 
Proposition \ref{siP5.38}, we obtain $P\in Spec(\mathcal A)-Ker(L_{\langle Q\rangle})$ such that $\langle L_{\langle Q\rangle }(P)\rangle=\langle P'\rangle$.  
We claim that $\langle P\rangle \in LAss(M)$. 

\smallskip Indeed, since $L_{\langle Q\rangle}(P)\ne 0$, it follows that $Q\prec P$. From Lemma \ref{siL5.37}, we know that $L_{\langle Q\rangle}(P)\in Spec(\mathcal A/\langle Q\rangle)$. This allows us to factor $L_{\langle P\rangle}: \mathcal A\longrightarrow \mathcal A/\langle P\rangle $ as a composition of exact
localization functors $L_{\langle P\rangle}=L'_{\langle P\rangle}\circ L_{\langle Q\rangle}$. This induces as in \eqref{sieq5.16},a commutative diagram
\begin{equation}
\begin{tikzcd}
\mathfrak Spec(\mathcal A)-Ker(L_{\langle P\rangle}) \arrow{r}{\overline{L}_{\langle P\rangle}} \arrow[swap]{d}{\overline{L}_{\langle Q\rangle}} & \mathfrak Spec(\mathcal A/\langle P\rangle) \\
\mathfrak Spec(\mathcal A/\langle P\rangle) - Ker(L'_{\langle P\rangle}) \arrow{ur}[swap]{\overline{L'}_{\langle P\rangle}} 
\end{tikzcd}
\end{equation} We already know that $\langle L_{\langle Q\rangle }(P)\rangle=\langle P'\rangle \in Ass(L_{\langle Q\rangle}(M))$. We now see that 
we have
\begin{equation*}
0\ne \langle L_{\langle P\rangle}(P)\rangle =\overline{L'}_{\langle P\rangle} \overline{L}_{\langle Q\rangle}(\langle P\rangle)=\overline{L'}_{\langle P\rangle} \langle {L}_{\langle Q\rangle}(P)\rangle= \langle {L'}_{\langle P\rangle}(P')\rangle
\end{equation*}  whence it follows that $ {L'}_{\langle P\rangle}(P')\ne 0$. From \cite[III.8.5.4]{R3}, it now follows that 
$\langle L'_{\langle P\rangle}(P')\rangle \in Ass(L'_{\langle P\rangle}L_{\langle Q\rangle}(M))$. In other words, we have
$\langle L_{\langle P\rangle}(P)\rangle= \langle {L'}_{\langle P\rangle}(P')\rangle \in Ass(L_{\langle P\rangle}(M))$, i.e., $\langle P
\rangle \in LAss(M)$. This proves the result.
\end{proof}

For the remainder of this section, we will always assume that  for every spectral object $Q\in Spec(\mathcal A)$
and every $M\in \mathcal A$, we have
\begin{equation}\label{tcondc1}
Ass(L_{\langle Q\rangle}(M))=Ass(M)\cap \mathfrak Spec(\mathcal A/\langle Q\rangle)
\end{equation} In particular,  Proposition \ref{sIPx5.10l} gives a sufficient condition for this to happen. 

\begin{lem}\label{L5.4} For any finitely generated object $M\in \mathcal A$, we have:
\begin{equation}\label{eq5.14}
Supp(M)=\underset{\langle Q\rangle \in Ass(M)}{\bigcup}\textrm{ }\overline{\langle Q\rangle}
\end{equation}
\end{lem}

\begin{proof} Since $Ass(M)\subseteq Supp(M)$, the inclusion, $\underset{\langle Q\rangle \in Ass(M)}{\bigcup}\textrm{ }\overline{\langle Q\rangle}\subseteq Supp(M)$ is obvious. On the other hand, since $M$ is finitely generated, there
exists a finite filtration
\begin{equation}
M=M_k\supseteq M_{k-1}\supseteq ... \supseteq M_1\supseteq M_0=0
\end{equation} with each quotient $Q_i=M_i/M_{i-1}\in Spec(\mathcal A)$. We let $Max\{Q_1,...,Q_k\}$ be the subcollection
of elements in $\{Q_1,...,Q_k\}$ maximal with respect to the relation ``$\prec$'', i.e., $Q\in Max\{Q_1,...,Q_k\}$ if and only
if $Q\prec Q_i$ for some $1\leq i\leq k$ implies $\langle Q\rangle =\langle Q_i\rangle$.  We notice that:
\begin{equation}
Supp(M)=\underset{i=1}{\overset{k}{\bigcup}}\textrm{ }Supp(Q_i)=\underset{Q\in Max\{Q_1,...,Q_k\}}{\bigcup}
\overline{\langle Q\rangle}
\end{equation} Hence, in order to prove the equality in \eqref{eq5.14}, it suffices to show that  if $
Q\in Max\{Q_1,...,Q_k\}$, then $\langle Q\rangle \in Ass(M)$.  We now consider the flat localization:
\begin{equation}
L_{\langle Q\rangle}:\mathcal A\longrightarrow \mathcal A/\langle Q\rangle
\end{equation} for some $Q\in Max\{Q_1,...,Q_k\}$.  Since $\langle Q\rangle \in Supp(M)$, we know that $L_{\langle Q\rangle}(M)\ne 0$ (see 
\cite[$\S$ 5.2.1]{Rose2}). Hence, $Ass(L_{\langle Q\rangle}(M))\ne \phi$ by Lemma \ref{L5.39}. However, from condition \eqref{tcondc1}, we know that
\begin{equation}
Ass(L_{\langle Q\rangle}(M))=Ass(M)\cap \mathfrak Spec(\mathcal A/\langle Q\rangle)
\end{equation} where $\mathfrak Spec(\mathcal A/\langle Q\rangle)$ is treated as a subset 
of $\mathfrak Spec(\mathcal A)$. Since $Ass(M)\subseteq Supp(M)$, we now have
\begin{equation}
\phi\ne Ass(L_{\langle Q\rangle}(M))=Ass(M)\cap \mathfrak Spec(\mathcal A/\langle Q\rangle)\subseteq Supp(M)\cap 
\mathfrak Spec(\mathcal A/\langle Q\rangle)
\end{equation} Finally, we consider a spectral object $Q'\in Spec(\mathcal A)$ such that 
$\langle Q'\rangle\in  Supp(M)\cap \mathfrak Spec(\mathcal A/\langle Q\rangle)$. Since $\langle Q'\rangle\in   \mathfrak Spec(\mathcal A/\langle Q\rangle)$, we have $Q\prec Q'$. However, since $\langle Q'\rangle \in Supp(M)$, we can find
some $Q_i\in \{Q_1,...,Q_k\}$ such that $Q'\prec Q_i$. Then, $Q\prec Q'\prec Q_i$. However, since $Q\in Max\{Q_1,...,
Q_k\}$, it follows that $\langle Q_i\rangle =\langle Q\rangle$ and we must have $\langle Q'\rangle =\langle Q\rangle$. We now
see that
\begin{equation}\label{eq5.20}
\phi\ne Ass(L_{\langle Q\rangle}(M))=Ass(M)\cap \mathfrak Spec(\mathcal A/\langle Q\rangle)\subseteq Supp(M)\cap 
\mathfrak Spec(\mathcal A/\langle Q\rangle)=\{\langle Q\rangle\}
\end{equation} From \eqref{eq5.20}, it is clear that $\langle Q\rangle\in Ass(M)$. This proves the result. 
\end{proof}

\begin{lem}\label{L5.41}  For any object $M\in \mathcal A$, we have:
\begin{equation}\label{eq5.14r}
Supp(M)=\underset{\langle Q\rangle \in Ass(M)}{\bigcup}\textrm{ }\overline{\langle Q\rangle}
\end{equation}

\begin{proof} We consider a spectral object $P\in Spec(\mathcal A)$ with $\langle P\rangle \in Supp(M)$. Since $M$ is the union
of the system of its finitely generated subobjects, we can find some finitely generated $M'\subseteq M$ such that 
$\langle P\rangle \in Supp(M')$. From Lemma \ref{L5.4}, it follows that there is some $\langle Q'\rangle \in Ass(M')$
such that $\langle P\rangle \in \overline{\langle Q'\rangle}$. But $Ass(M')\subseteq Ass(M)$ and hence we get
$ Supp(M)\subseteq \underset{\langle Q\rangle \in Ass(M)}{\bigcup}\textrm{ }\overline{\langle Q\rangle}$. The reverse
inclusion is obvious and the result follows. 
\end{proof}

\end{lem}

\begin{thm}\label{L5.5} Let $M$ be an object of $\mathcal A$ and let $N\subseteq M$ be an essential subobject. Then, we have: 

\smallskip
(a) $M$ and $N$ have the same associated points, i.e., $Ass(M)=Ass(N)$. 

(b)  $M$ and $N$ have the same support, i.e., $Supp(M)=Supp(N)$. 
\end{thm}
\begin{proof} (a) We consider a point $\langle Q\rangle$ of $Ass(M)$ given by a spectral object
$Q\in Spec(\mathcal A)$ along with an inclusion $Q\hookrightarrow M$. Since $N$ is an essential subobject of $M$,
we must have $Q\cap N\ne 0$. Since $Q$ is spectral, the nonzero subobject $Q\cap N$ is also spectral
and we have $\langle Q\cap N\rangle =\langle Q\rangle$. Then, $\langle Q\rangle =\langle Q\cap N
\rangle\in Ass(N)$. This gives $Ass(M)\subseteq Ass(N)$. Since $N\subseteq M$, we already know that
$Ass(N)\subseteq Ass(M)$ and the result follows. The result of (b) follows from part (a) and Lemma \ref{L5.41}. 
\end{proof}

\begin{thm}\label{P5.7}  Let
$S\subseteq \mathfrak Spec(\mathcal A)$ be a specialization closed subset and let $(\mathcal T(S),
\mathcal F(S))$ be the torsion pair on $\mathcal A$ corresponding to $S$. 
Then, for any $P\in Spec(\mathcal A)$ with $\langle P\rangle \in S$, the injective hull 
$E(P)$ lies in $  \mathcal T(S)$. 
\end{thm}

\begin{proof} We consider $P\in Spec(\mathcal A)$ with $\langle P\rangle \in S$. Since $S$ is specialization closed,
we have $Supp(P)\subseteq S$. From Proposition \ref{L5.5}, we know that $Supp(E(P))=Supp(P)$. Then, 
$Supp(E(P))\subseteq S$ and we have $E(P)\in\mathcal T(S)$. 
\end{proof}

\section{Subsets of the spectrum and $\mathbf 1$-resolving subcategories}

In this section, we will describe the connection between resolving subcategories of $\mathcal A$ and certain  specialization
closed subsets of the spectrum of $\mathcal A$. We begin by recalling the following definition from \cite{Stov}. 

\begin{defn}\label{D6.1} (see \cite[$\S$ Definition 2.1]{Stov}) Let $\mathcal B$ be an abelian category. A full subcategory 
$\mathcal C\subseteq \mathcal B$ is said to be a resolving subcategory if it satisfies the following conditions:

\smallskip
(1) The subcategory $\mathcal C$ is closed under taking direct summands. 

(2) Suppose that $0\longrightarrow M'\longrightarrow M\longrightarrow M''\longrightarrow 0$ is a short exact sequence in
$\mathcal B$ and that $M''\in \mathcal C$. Then, $M\in \mathcal C$ if and only if $M'\in \mathcal C$. 

(3) The subcategory $\mathcal C$ is generating. In other words, for every object $B\in \mathcal B$, there exists an  epimorphism
$C\twoheadrightarrow B$ in $\mathcal B$ with $C\in \mathcal C$. 
\end{defn}

For a resolving subcategory $\mathcal C$ of an abelian category $\mathcal B$, \v{S}\v{t}ov\'{i}\v{c}ek
establishes the following key result. 

\begin{thm}\label{P6.2} (see \cite[Proposition 2.3]{Stov} Let $\mathcal B$ be an abelian category and let $\mathcal C$ be a 
resolving subcategory. For some integer $n\geq 0$, let $B\in \mathcal B$ be an object such that there are exact sequences 
\begin{equation}
\begin{array}{c}
0\longrightarrow K\longrightarrow C_{n-1}\longrightarrow \dots \longrightarrow C_1\longrightarrow C_0
\longrightarrow B\longrightarrow 0\\
0\longrightarrow K'\longrightarrow C'_{n-1}\longrightarrow \dots \longrightarrow C'_1\longrightarrow C'_0
\longrightarrow B\longrightarrow 0\\
\end{array}
\end{equation}
in $\mathcal B$ with $C_i$, $C'_{i}\in \mathcal C$ for each $0\leq i\leq n-1$. Then, the object 
$K\in \mathcal C$ if and only if $K'\in \mathcal C$. 

\end{thm} 

The result of Proposition \ref{P6.2} motivates us to make the following definition. 

\begin{defn}\label{D6.3} 
 Let $\mathcal B$ be an abelian category and let $\mathcal C$ be a 
resolving subcategory. For a given integer $n\geq 0$, we will say that the subcategory $\mathcal C\subseteq \mathcal B$
is $n$-resolving if for each object $B\in \mathcal B$, there exists an exact sequence 
\begin{equation}
0\longrightarrow C_n\longrightarrow C_{n-1}\longrightarrow \dots \longrightarrow C_1\longrightarrow C_0
\longrightarrow B\longrightarrow 0\\
\end{equation} in $\mathcal B$ with $C_i\in \mathcal C$ for each $0\leq i\leq n$. 
\end{defn}

We now come back to our locally noetherian Grothendieck category $\mathcal A$ and let $\mathcal C\subseteq 
\mathcal A$ be a resolving subcategory. If $G\in \mathcal A$ is a generator, there is an epimorphism $G_0
\longrightarrow G$ with $G_0\in \mathcal C$. Consequently, for every object $M\in \mathcal A$, there is an epimorphism
$G_0^{I}\longrightarrow M$ from a direct sum of copies of $G_0$. It follows from \cite[Proposition 1.9.1]{Tohoku} that $G_0$
is also a generator for $\mathcal A$. Hence, every resolving subcategory $\mathcal C\subseteq \mathcal A$ must
contain a generator for $\mathcal A$. 

\smallskip

We remark here that Definition \ref{D6.3} essentially says that an $n$-resolving subcategory  
$\mathcal C\subseteq \mathcal A$ ``contains all 
$n$-th syzygy objects" of $\mathcal A$ with respect to $\mathcal C$. For a subcategory of modules over a commutative noetherian ring $R$, the property 
of containing all $n$-th syzygy objects is related to being an $n$-cotilting class (see the equivalence 
in \cite[Proposition 3.14]{Four}).  The relationship between $n$-cotilting classes in module categories and systems of
specialization closed subsets of the Zariski spectrum has been described in \cite[Theorem 2.7 \& Theorem 3.7]{Four}. For Grothendieck categories that contain enough projectives,  cotilting
classes have been considered by Colpi \cite{Colpi}. For more on tilting
or cotilting classes of modules over a ring, we refer the reader to \cite{Bazz}, \cite{CT}. 

\begin{lem}\label{L6.4} Let $\mathcal A$ be a locally noetherian Grothendieck category  and let $(\mathcal T,\mathcal F)$
be a hereditary torsion pair on $\mathcal A$. Then, the torsion free class $\mathcal F$ contains all coproducts. 
\end{lem}

\begin{proof} We consider a collection $\{N_i\}_{i\in I}$ of objects of $\mathcal F$ and set $N:=\underset{i\in I}{
\bigoplus}N_i$. We denote by $Fin(I)$ the collection of all finite subsets of $I$ and for any $J\in Fin(I)$, we set
$N_J:=\underset{j\in J}{\bigoplus}N_j$. It is clear that the finite direct sum (finite direct product)  $N_J$ lies in $\mathcal F$ and 
that $N$ may be expressed as the filtered colimit $N=\underset{J\in Fin(I)}{\varinjlim}N_J$. 

\smallskip
We need to show that $N\in \mathcal F$. For this, we consider some $M\in \mathcal T$ and a finitely generated
subobject $M'\subseteq M$. Since the torsion pair $(\mathcal T,\mathcal F)$ is hereditary, we must have
$M'\in \mathcal T$. Then since $M'$ is finitely generated, we have:
\begin{equation}
Hom_{\mathcal A}(M',N)=\underset{J\in Fin(I)}{\varinjlim}Hom_{\mathcal A}(M',N_J)=0
\end{equation} Finally, since the category $\mathcal A$ is locally finitely generated, the object $M$
may be expressed as the colimit of its finitely generated subobjects. This shows that 
$Hom_{\mathcal A}(M,N)=0$ for each $M\in \mathcal T$ and hence $N\in \mathcal F$. 
\end{proof}

\begin{thm}\label{P6.5} The associations :
\begin{equation}\label{eq6.3}
\mathcal C\mapsto S_{\mathcal C}:=\mathfrak Spec(\mathcal A)\backslash Ass(\mathcal C)
\qquad S\mapsto \mathcal C_S:=\{\mbox{$M\in \mathcal A$ $\vert$ $Ass(M)\cap S=\phi$}\}
\end{equation} give a one-one correspondence between the following : 

\smallskip
(1) $1$-resolving subcategories $\mathcal C\subseteq \mathcal A$ that are torsion free classes of hereditary
torsion pairs on $\mathcal A$ 

(2) Specialization closed subsets $S\subseteq \mathfrak Spec(\mathcal A)$ such that there exists at least one generator
$G$ of $\mathcal A$ such that $S\cap Ass(G)=\phi$.  
\end{thm}

\begin{proof} From Corollary \ref{C5.5},  it is clear that both the associations in \eqref{eq6.3} are induced
by the one-one correspondence in Proposition \ref{P5corr} between torsion free classes of hereditary torsion pairs on $\mathcal A$ and 
specialization closed subsets of $\mathfrak Spec(\mathcal A)$. 

\smallskip
We now consider a subcategory $\mathcal C\subseteq \mathcal A$ as in (1). Since $\mathcal C$ is resolving, 
it must contain a generator $G$ of the Grothendieck category $\mathcal A$. Hence, $S_{\mathcal C}=\mathfrak Spec(\mathcal A)\backslash Ass(\mathcal C)$ satisfies $S_{\mathcal C}\cap Ass(G)=\phi$. 

\smallskip
Conversely, we consider a specialization closed subset $S\subseteq \mathfrak Spec(\mathcal A)$ as in 
(2). By assumption, there is a generator $G$ of $\mathcal A$ such that $Ass(G)\cap S=\phi$. Hence, 
$G\in \mathcal C_S$. Then, for any object $M\in \mathcal A$, there is an epimorphism
$G^{I}\longrightarrow M$ from a direct sum of copies of $G$. Since $\mathcal C_S$ is the torsion free class
of a hereditary torsion pair on $\mathcal A$, it follows from Lemma \ref{L6.4} that the  direct sum $G^{I}$ lies in $\mathcal C_S$.  Further, since the torsion free class $\mathcal C_S$
is closed under subobjects, the kernel $K$ of the epimorphism $G^{I}\longrightarrow M$ lies in 
$\mathcal C_S$. This gives a short exact sequence 
\begin{equation}
0\longrightarrow K\longrightarrow G^{I}\longrightarrow M\longrightarrow 0
\end{equation} with both $G^{I}$, $K\in \mathcal C_S$.
 Since the torsion free class $\mathcal C_S$ is also closed under extensions and subobjects, given a short
exact sequence $0\longrightarrow N'\longrightarrow N\longrightarrow N''\longrightarrow 0$ with $N''\in \mathcal C_S$,
then $N'\in \mathcal C_S$ if and only if $N\in \mathcal C_S$. In particular, the torsion free class
$\mathcal C_S$ is also closed under taking direct summands. It follows that $\mathcal C_S$ is a $1$-resolving subcategory. 
\end{proof}

\section{Injective hulls, cosyzygies and their associated points}

\smallskip
In order to proceed further, we will need to generalize some classical results in commutative algebra to the 
locally noetherian Grothendieck category $\mathcal A$. The most important among these is a characterization of the associated points
of higher cosyzygies of an object of $\mathcal A$ in terms of $Ext$ groups. This is motivated by a result of Bass (see, for instance, 
\cite[Proposition 3.2.9]{BH})
which describes the associated primes of higher cosyzygies of a module over a commutative noetherian ring. 

\smallskip
Given an object $M\in \mathcal A$ and its injective hull $E(M)$, we have the following short exact sequence in $\mathcal A$:
\begin{equation}\label{eq7.1}
0\longrightarrow M\longrightarrow E(M)\longrightarrow \mho(M):=E(M)/M\longrightarrow 0
\end{equation}

\begin{defn}\label{D7.1} For $M\in \mathcal A$, put $E_0(M):=E(M)$ and $\mho_0(M):=M$. For $k\geq 1$, $E_k(M)$ and
$\mho_k(M)$ are defined by setting:
\begin{equation}
\mho_k(M):=\mho(\mho_{k-1}(M))=E(\mho_{k-1}(M))/\mho_{k-1}(M) \qquad E_k(M):=E(\mho_k(M))
\end{equation} For any $k<0$, we always set $\mho_k(M)=0$.  For each $k\in \mathbb Z$, the object $\mho_k(M)$ is referred to as the $k$-th cosyzygy of $M$. The sequence 
\begin{equation}
0\longrightarrow M\longrightarrow E_0(M)\longrightarrow E_1(M)\longrightarrow E_2(M)\longrightarrow \dots
\end{equation} is referred to as the minimal injective resolution of $M$. 
\end{defn}

The short exact sequence in \eqref{eq7.1} shows that for any object $T\in \mathcal A$ and any $k\geq 1$, we have:
\begin{equation}\label{ert7.4}
Ext^k_{\mathcal A}(T,\mho(M))=Ext^{k+1}_{\mathcal A}(T,M)
\end{equation}

\begin{lem}\label{L7.1}
If $M$, $N\in \mathcal A$ are objects such that $Hom_{\mathcal A}(M,N)\ne 0$, we must have
$Supp(M)\cap Ass(N)\ne \phi$.
\end{lem}

\begin{proof}
Let $\psi: M\longrightarrow N$ be a nonzero morphism in $\mathcal A$. Then, $Ass(Im(\psi))\ne \phi$. We see that 
$M/Ker(\psi)=Im(\psi)\subseteq N$ which gives us 
$Ass(Im(\psi))\subseteq Ass(N)$ and $Ass(Im(\psi))\subseteq Supp(Im(\psi))\subseteq Supp(M)$. Hence the result. 
\end{proof}

From now onwards, for any spectral object $P\in Spec(\mathcal A)$, we will denote by $K(P)$ the object 
$L_{\langle P\rangle}(P)\in \mathcal A/\langle P\rangle$. In the remainder of this paper, we will always assume that  $\mathcal A$ satisfies the condition in \eqref{tcondc1}. We also recall from Section 5 that $\mathcal A/\langle P\rangle$
must be a locally noetherian Grothendieck category.

\begin{lem}\label{L7.2} For any spectral object $P\in Spec(\mathcal A)$, the support of 
$K(P)$ is a single point, i.e., $Supp(K(P))=\{\langle P\rangle \}$.
\end{lem}

\begin{proof}
Since $K(P)=L_{\langle P\rangle}(P)$, it follows from \eqref{tcondc1} that
\begin{equation}
Ass(K(P))=Ass(P)\cap \mathfrak Spec(\mathcal A/\langle P\rangle)\subseteq Ass(P)
\end{equation} where $ \mathfrak Spec(\mathcal A/\langle P\rangle)$ is treated as a subset 
of $\mathfrak Spec(\mathcal A)$.  From \cite[$\S$ 8.1]{Rose2}, we know that $Ass(P)=\{\langle P\rangle\}$ and from Lemma 
\ref{L5.39} we know that $K(P)$ being a nonzero object of $\mathcal A/\langle P\rangle$ must have an associated
point, i.e., $Ass(K(P))\ne \phi$. It follows that $Ass(K(P))=\{\langle P\rangle\}$. 

\smallskip
From Lemma \ref{L5.41} it now follows that $Supp(K(P))$ is the closure of the point $\langle L_{\langle P\rangle}(P)
\rangle \in \mathfrak
Spec(\mathcal A/\langle P\rangle)$. However, for any nonzero object $T\in \mathcal A/\langle P\rangle$, we know that
$L_{\langle P\rangle}(P)\prec T$ (see the proof of \cite[Proposition 3.3.1]{Rose2}). Hence, the closure of $L_{\langle P\rangle}(P) \in \mathfrak
Spec(\mathcal A/\langle P\rangle)$ is the single point $\{\langle P\rangle\}$.

\end{proof}

\begin{lem}
\label{L7.4} Let $M$ be an object of $\mathcal A$. Then, the associated points of $M$ may be described as follows: 
\begin{equation}
Ass(M)=\{\mbox{$\langle P\rangle $ $\vert$ $P\in Spec(\mathcal A)$ and $Hom_{\mathcal A/\langle P\rangle}(K(P),L_{\langle P\rangle}(M))\ne 0$}\}
\end{equation} 
\end{lem}

\begin{proof} We choose a point   $\langle P\rangle \in Ass(M)$ represented by a spectral object 
$P\in Spec(\mathcal A)$ such that there is a monomorphism 
$P\hookrightarrow M$. Since $L_{\langle P\rangle}:\mathcal A\longrightarrow \mathcal A/\langle P\rangle$ is exact,
we obtain a monomorphism $K(P)=L_{\langle P\rangle}(P)\hookrightarrow L_{\langle P\rangle}(M)$. Then, 
$Hom_{\mathcal A/\langle P\rangle}(K(P),L_{\langle P\rangle}(M))\ne 0$.

\smallskip
Conversely, let $P\in Spec(\mathcal A)$ be such that $Hom_{\mathcal A/\langle P\rangle}(K(P),L_{\langle P\rangle}(M))\ne 0$. From Lemma \ref{L5.39}, we know
that every nonzero object in $\mathcal A/\langle P\rangle$  has an associated point. Applying Lemma \ref{L7.1}, we see
that $Supp(K(P))\cap Ass(L_{\langle P\rangle}(M))\ne \phi$. From Lemma \ref{L7.2}, we know that 
$Supp(K(P))=\{\langle P\rangle\}$ is a single point and hence $\langle P\rangle\in Ass(L_{\langle P\rangle}(M))$. Again,
$Ass(L_{\langle P\rangle}(M))=Ass(M)\cap \mathfrak Spec(\mathcal A/\langle P\rangle)\subseteq Ass(M)$ and hence
$\langle P\rangle \in Ass(M)$. 
\end{proof}

\begin{lem}\label{L7.5} Let $M$ be an object of $\mathcal A$ and $P\in Spec(\mathcal A)$ a spectral object. Then:

\smallskip
(a) Let $N\in \mathcal A/\langle P\rangle$ be such that $M$ is an essential subobject
of $i_{\langle P\rangle}(N)$. Then, $L_{\langle P\rangle}(M)\subseteq 
L_{\langle P\rangle}(i_{\langle P\rangle}(N))=N$ is an essential subobject.

(b) 
Let $E(M)$ be the injective hull of $M$ in $\mathcal A$. Then, 
the localization $L_{\langle P\rangle}(E(M))$ is the injective hull of $L_{\langle P\rangle}(M)$ in $\mathcal A/\langle P\rangle$. 

(c) Let $0\longrightarrow L_{\langle P\rangle}(M)\longrightarrow E_0(L_{\langle P\rangle}(M)) 
\longrightarrow E_1(L_{\langle P\rangle}(M))\longrightarrow \dots  $ be a minimal injective resolution of
$L_{\langle P\rangle}(M)$ in $\mathcal A/\langle P\rangle$ and let $\{\mho_k(L_{\langle P\rangle}(M))\}_{k\geq 0}$ be
the corresponding cosyzygies. Then, for each $k\geq 0$,  we have
\begin{equation}\label{er7.7}
\mho_k(L_{\langle P\rangle}(M))=L_{\langle P\rangle}(\mho_k(M))\qquad E_k(L_{\langle P\rangle}(M))=
L_{\langle P\rangle}(E_k(M))
\end{equation} 
\end{lem}

\begin{proof}  (a) We consider a subobject $0\ne K\subseteq N$. Since $i_{\langle P\rangle}:
\mathcal A/\langle P\rangle \longrightarrow \mathcal A$ is fully faithful and a right
adjoint, we get  $0\ne i_{\langle P\rangle}(K)\subseteq i_{\langle P\rangle}(N)$. Since $M$ is  an essential subobject
of $i_{\langle P\rangle}(N)$, we see that $T:= M\cap i_{\langle P\rangle}(K)\ne 0$. Since $L_{\langle P\rangle}$
is exact, it follows that $L_{\langle P\rangle}(T)\subseteq L_{\langle P\rangle}(M)$ and $L_{\langle P
\rangle}(T)\subseteq L_{\langle P\rangle}\circ i_{\langle P\rangle}(K)=K$. 

\smallskip
It remains to show that $L_{\langle P\rangle}(T)\ne 0$. For this, we notice that since $M$ is an essential subobject
of $i_{\langle P\rangle}(N)$, the intersection $T= M\cap i_{\langle P\rangle}(K)$ must be an essential subobject of $i_{\langle P\rangle}(K)$. But then we know from  Lemma \ref{L5.5} that $Supp(T)=Supp(i_{\langle P\rangle}(K))$. Since 
$L_{\langle P\rangle}\circ i_{\langle P\rangle}(K)=K\ne 0$, we get  $L_{\langle P\rangle}(T)\ne 0$. 

\smallskip
(b) From \cite[Corollaire III.2]{Gab}, we know that the injective object $E(M)$ in $\mathcal A$ may be expressed
as a direct sum $E(M)=E_1\oplus i_{\langle P\rangle}(E_2)$ where $P\not\prec E_1$ and $E_2$ is an injective
object of $\mathcal A/\langle P\rangle$. Then, $L_{\langle P\rangle}(E(M))=L_{\langle P\rangle}(E_1)
\oplus L_{\langle P\rangle}\circ i_{\langle P\rangle}(E_2)=E_2$ must be injective. 

\smallskip
It remains to show that $L_{\langle P\rangle}(M)\subseteq L_{\langle P\rangle}(E(M))=E_2$ is an essential extension. For this,
we note that $M\cap i_{\langle P\rangle}(E_2)$ is an essential subobject of $i_{\langle P\rangle}(E_2)$. From part (a),
it follows that $L_{\langle P\rangle}(M\cap i_{\langle P\rangle}(E_2))$ is an essential subobject of $E_2$. Since 
$L_{\langle P\rangle}(M\cap i_{\langle P\rangle}(E_2))\subseteq L_{\langle P\rangle}(M)$, it follows that
$L_{\langle P\rangle}(M)$ is an essential subobject of $E_2$. 

\smallskip
(c) From part (b), we know that the injective hull commutes with the functor
$L_{\langle P\rangle}$. Since $L_{\langle P\rangle}:\mathcal A\longrightarrow \mathcal A/\langle P\rangle$ is an exact functor,
the equalities in \eqref{er7.7} are now evident from the constructions in Definition \ref{D7.1}. 

\end{proof}

\begin{thm}\label{Bass} Let $M$ be an object of $\mathcal A$.  
If
$P\in Spec(\mathcal A)$ is such that $Ext^k_{\mathcal A/\langle P\rangle}(K(P),L_{\langle P\rangle}(M))\ne 0$,
then the point $\langle P\rangle \in Ass(\mho_k(M))=Ass(E_k(M))$. 
\end{thm}

\begin{proof} Since $E_k(M)=E(\mho_k(M))$ is the injective hull of $\mho_k(M)$, it follows from Proposition \ref{P5.7} that
$Ass(\mho_k(M))=Ass(E_k(M))$. Now suppose that $P\in Spec(\mathcal A)$ is  such that 
$Ext^k_{\mathcal A/\langle P\rangle}(K(P),L_{\langle P\rangle}(M))\ne 0$. If $k=0$, we have
$Hom_{\mathcal A/\langle P\rangle}(K(P),L_{\langle P\rangle}(M))\ne 0$ and it follows from Lemma \ref{L7.4} that
$\langle P\rangle \in Ass(M)=Ass(\mho_0(M))$. 

\smallskip 
We now suppose that $k\geq  1$. By repeatedly applying \eqref{ert7.4} to the category $\mathcal A/\langle P\rangle$,
we obtain
\begin{equation}\label{te7.8}
0\ne Ext^k_{\mathcal A/\langle P\rangle}(K(P),L_{\langle P\rangle}(M))  = ... = Ext^1_{\mathcal A/
\langle P\rangle}(K(P),\mho_{k-1}(L_{\langle P\rangle}(M)))
\end{equation} We now consider the following short exact sequence in $\mathcal A/\langle P\rangle$:
\begin{equation}\label{et7.9}
0\longrightarrow \mho_{k-1}(L_{\langle P\rangle}(M))\longrightarrow E_{k-1}(L_{\langle P\rangle}(M))
\longrightarrow \mho_{k}(L_{\langle P\rangle}(M))\longrightarrow 0
\end{equation} Then \eqref{et7.9} induces the exact sequence:
\begin{equation*}
\begin{array}{c}
0\longrightarrow Hom_{\mathcal A/\langle P\rangle}(K(P),\mho_{k-1}(L_{\langle P\rangle}(M)))\longrightarrow Hom_{\mathcal A/\langle P\rangle}(K(P),E_{k-1}(L_{\langle P\rangle}(M)))
\longrightarrow \\
\qquad \qquad \qquad Hom_{\mathcal A/\langle P\rangle}(K(P),\mho_{k}(L_{\langle P\rangle}(M)))\longrightarrow Ext^1_{\mathcal A/
\langle P\rangle}(K(P),\mho_{k-1}(L_{\langle P\rangle}(M)))\longrightarrow 0
\end{array}
\end{equation*} Then, $Ext^1_{\mathcal A/
\langle P\rangle}(K(P),\mho_{k-1}(L_{\langle P\rangle}(M)))\ne 0$ implies that $Hom_{\mathcal A/\langle P\rangle}(K(P),\mho_{k}(L_{\langle P\rangle}(M)))\ne 0$. 
By Lemma \ref{L7.5}, $\mho_k(L_{\langle P\rangle}(M))=L_{\langle P\rangle}(\mho_k(M))$ and hence 
$Hom_{\mathcal A/\langle P\rangle}(K(P),L_{\langle P\rangle}(\mho_k(M)))\ne 0$. It now follows from 
Lemma \ref{L7.4} that $\langle P\rangle \in Ass(\mho_k(M))$. 

\end{proof}

The above result gives a sufficient condition for the associated points of the $k$-th cosyzygy $\mho_k(M)$ of an 
object $M\in\mathcal A$. In order to obtain a necessary condition, we will need a better understanding of the
abelian category $\mathcal A/\langle P\rangle$. 

\begin{defn}\label{D7.7} (see \cite[$\S$ 3]{Rose2}) Let $\mathcal B$ be an abelian category. A nonzero object $F\in \mathcal B$ is said
to be quasi-final if it satisfies $F\prec B$ for any nonzero object $B\in \mathcal B$. 

\smallskip
An abelian category is said to be local if it contains a quasi-final object.
\end{defn} 

\smallskip
From \cite[$\S$ 3]{Rose2}, we know that for any $P\in Spec(\mathcal A)$, the category $\mathcal A/\langle P
\rangle$ is a local abelian category with  $L_{\langle P\rangle}(P)$ as a quasi-final object. Further, since $\mathcal 
A/\langle P\rangle$ is a locally noetherian Grothendieck category, its spectral objects are  finitely
generated by  Lemma \ref{L3.2}. A local abelian category has finitely generated objects if and only if its quasi-final objects 
are semi-simple, i.e., they may be expressed as a direct sum of simple objects (see proof of \cite[Lemma 5.4.1]{Rose2}). We recall
that an object in an abelian category is said to be simple if it has no nonzero proper subobjects. 

\smallskip
Accordingly, we know that the quasi-final object $K(P)=L_{\langle P\rangle}(P)$ of $\mathcal A/\langle P\rangle$ may be expressed as a finite direct sum
of simple objects of $\mathcal A/\langle P\rangle$. Since all simple objects of a local abelian category are isomorphic 
(see \cite[$\S$ 3.1.2]{Rose2}), we may express $L_{\langle P\rangle}(P)$ as a direct sum $K(P)=L_{\langle P\rangle}(P)=S_{\langle P\rangle}^n$,
where $S_{\langle P\rangle}$ is a simple object of $\mathcal A/\langle P\rangle$. Since the category $\mathcal 
A/\langle P\rangle$ depends only on the point $\langle P\rangle \in \mathfrak Spec(\mathcal A)$, 
so does the object $S_{\langle P\rangle}$. 

\begin{lem}\label{L7.8y} Let $M\subseteq N$ be an essential extension in the abelian category $\mathcal A/\langle P\rangle$. 
Then, for any morphism $\phi:S_{\langle P\rangle}\longrightarrow N$ in $\mathcal A/\langle P\rangle$, the composition
$S_{\langle P\rangle}\overset{\phi}{\longrightarrow}N\longrightarrow N/M$ is zero.
\end{lem}

\begin{proof} For $\phi=0$, the result is obvious. We suppose therefore that $\phi\ne 0$ and set
$0\ne I:=Im(\phi)\subseteq N$. Since $M$ is an essential subobject of $N$, we get 
$0\ne I\cap M\subseteq I$. Since  $I$ is a quotient of the simple object $S_{\langle P\rangle}$, $I$
is also simple.  It follows that
$I\cap M=I$, i.e.,
$I\subseteq M$. It is now clear that the composition of $\phi$ with the canonical morphism
$N\longrightarrow N/M$ is zero. 
\end{proof}

We are now ready to show that the condition in Proposition \ref{Bass} is an if and only if condition
for the associated points of $\mho_k(M)$. 

\begin{thm}\label{Bass9}    Let $M$ be an object of $\mathcal A$. Then, 
\begin{equation*} Ass(\mho_k(M))=Ass(E_k(M))=\{\mbox{$\langle P\rangle$ $\vert$ $P\in Spec(\mathcal A)$
such that $Ext^k_{\mathcal A/\langle P\rangle}(K(P),L_{\langle P\rangle}(M))\ne 0$}\}
\end{equation*}
\end{thm}

\begin{proof}
We have already shown one direction in Proposition \ref{Bass}. For the other direction, we consider some
$P\in Spec(\mathcal A)$ with a monomorphism $P\hookrightarrow \mho_k(M)$. If $k=0$, it follows from the proof of 
Lemma \ref{L7.4} that $Ext^0_{\mathcal A/\langle P\rangle}(K(P),L_{\langle P\rangle}(M))=Hom_{\mathcal A/\langle P\rangle}(K(P),L_{\langle P\rangle}(M))\ne 0$. 

\smallskip
We suppose therefore that $k\geq 1$. As in \eqref{te7.8}, we have:
\begin{equation}\label{zer}
 Ext^k_{\mathcal A/\langle P\rangle}(K(P),L_{\langle P\rangle}(M))  = ... = Ext^1_{\mathcal A/
\langle P\rangle}(K(P),\mho_{k-1}(L_{\langle P\rangle}(M)))
\end{equation} We notice that the following is an injective resolution of $\mho_{k-1}(L_{\langle P\rangle}(M))$:
\begin{equation}\label{te7.11}
\begin{array}{l}
0\rightarrow E_{k-1}(L_{\langle P\rangle}(M))=E(\mho_{k-1}(L_{\langle P\rangle}(M)))\rightarrow  \\
\qquad \qquad\qquad\qquad\qquad \qquad E_k(L_{\langle P\rangle}(M))=E(\mho_k(L_{\langle P\rangle}(M)))\rightarrow E_{k+1}(L_{\langle P\rangle}(M))\rightarrow 
\dots 
\end{array}
\end{equation} We notice that each differential in \eqref{te7.11} can be expressed as a composition (for any $t\geq 0$):
\begin{equation}\label{te7.12}
\begin{CD}
 E_{k-1+t}(L_{\langle P\rangle}(M))=E(\mho_{k-1+t}(L_{\langle P\rangle}(M)))\\
@VVV \\
\mho_{k+t}(L_{\langle P\rangle}(M))=E(\mho_{k-1+t}(L_{\langle P\rangle}(M)))/\mho_{k-1+t}(L_{\langle P\rangle}(M)) \\
@VVV\\
E_{k+t}(L_{\langle P\rangle}(M))=E(\mho_{k+t}(L_{\langle P\rangle}(M)))\\
\end{CD}
\end{equation} We can get the groups $Ext^\ast_{\mathcal A/\langle P\rangle}(S_{\langle P\rangle},\mho_{k-1}(L_{\langle P\rangle}(M))$
by applying the functor $Hom_{\mathcal A/\langle P\rangle}(S_{\langle P\rangle},\_\_)$ to the injective
resolution in \eqref{te7.11}. Since $\mho_{k-1+t}(L_{\langle P\rangle}(M))$ is an essential subobject of  $E(\mho_{k-1+t}(L_{\langle P\rangle}(M)))$ for each $t\geq 0$, it follows from Lemma \ref{L7.8y} and the description in \eqref{te7.12} that the differentials
of the complex $\{Hom_{\mathcal A/\langle P\rangle}(S_{\langle P\rangle},E_{k-1+\ast}(L_{\langle P\rangle}(M))))\}$
are all zero. It follows that
\begin{equation}\label{frq}
\begin{array}{ll}
Ext_{\mathcal A/\langle P\rangle}^1(S_{\langle P\rangle},\mho_{k-1}(L_{\langle P\rangle}(M))&= Hom_{\mathcal A/
\langle P\rangle}(S_{\langle P\rangle},E_k(L_{\langle P\rangle}(M)))\\&=Hom_{\mathcal A/
\langle P\rangle}(S_{\langle P\rangle},L_{\langle P\rangle}(E_k(M)))\\
\end{array}
\end{equation} On the other hand, we have a monomorphism $P\hookrightarrow \mho_k(M)\hookrightarrow E_k(M)=E(\mho_k(M))$ and it
follows from the proof of  Lemma \ref{L7.4} that 
$Hom_{\mathcal A/
\langle P\rangle}(K(P),L_{\langle P\rangle}(E_k(M)))\ne 0$. Since $K(P)$ is a finite direct sum of copies of $S_{\langle P\rangle}$, this gives $Hom_{\mathcal A/
\langle P\rangle}(S_{\langle P\rangle},L_{\langle P\rangle}(E_k(M)))\ne 0$.
Combining \eqref{frq} with \eqref{zer} and using again
the fact that  $K(P)$ is a finite direct sum of copies of $S_{\langle P\rangle}$, we now obtain
$Ext^k_{\mathcal A/\langle P\rangle}(K(P),L_{\langle P\rangle}(M)) \ne 0$. 
\end{proof}

\begin{cor}\label{Cor7.10}
(a) Let $0\longrightarrow M'\longrightarrow M\longrightarrow M''\longrightarrow 0$ be a short exact sequence in 
$\mathcal A$. Then, for any $k\in \mathbb Z$, we have
\begin{equation}
\begin{array}{c}
Ass(\mho_k(M'))\subseteq Ass(\mho_k(M))\cup Ass(\mho_{k-1}(M'')) \\ Ass(\mho_k(M))\subseteq 
Ass(\mho_k(M'))\cup Ass(\mho_k(M''))\\
Ass(\mho_k(M''))\subseteq Ass(\mho_k(M))\cup Ass(\mho_{k+1}(M'))
\end{array}
\end{equation}

\smallskip
(b) Let $0\longrightarrow K_n\longrightarrow K_{n-1}\longrightarrow ...\longrightarrow K_0\longrightarrow M\longrightarrow 0$ be an
exact sequence in $\mathcal A$. Then, we have:
\begin{equation}
\begin{array}{c}
Ass(\mho_k(K_n))\subseteq \left(\bigcup_{i=0}^{n-1}Ass(\mho_{i+k-n+1}(K_{i}))\right)\cup Ass(\mho_{k-n}(M))\\
Ass(\mho_k(M))\subseteq \bigcup_{i=0}^n Ass(\mho_{k+i}(K_i))
\end{array}
\end{equation}
\end{cor}

\begin{proof} (a) Since $L_{\langle P\rangle}$ is exact, we get a short exact sequence $0\longrightarrow L_{\langle P\rangle}(M')\longrightarrow L_{\langle P\rangle}(M)\longrightarrow L_{\langle P\rangle}(M'')\longrightarrow 0$. Then, for any
$k\in \mathbb Z$, we know that the following sequence is exact
\begin{equation*}
Ext^{k-1}_{\mathcal A/\langle P\rangle}(K(P),L_{\langle P\rangle}(M''))
\longrightarrow Ext^{k}_{\mathcal A/\langle P\rangle}(K(P),L_{\langle P\rangle}(M'))
\longrightarrow Ext^{k}_{\mathcal A/\langle P\rangle}(K(P),L_{\langle P\rangle}(M))
\end{equation*} The result now follows by applying Proposition \ref{Bass9}. The result of (b) follows
by repeated application of (a). 

\end{proof}

\section{Subsets of the spectrum and $\mathbf n$-resolving subcategories}

\smallskip
In this section, we will show how certain sequences of subsets of $\mathfrak Spec(\mathcal A)$ correspond to certain $n$-resolving
subcategories of $\mathcal A$. Although not an analogue, 
our result is  motivated by the methods of \cite[Theorem 3.7]{Four}, which gives a bijection between $n$-cotilting classes of $R$-modules 
and certain sequences of specialization closed subsets of $Spec(R)$, where $R$ is a commutative noetherian ring. 
We will begin  by taking a better look at $n$-resolving subcategories. 

\smallskip
Let $\mathcal A$ be a locally noetherian Grothendieck category as before and let $\mathcal C
\subseteq \mathcal A$ be a resolving subcategory. For any object $M\in \mathcal A$ and
each $j\geq 1$,  
we define a collection $\Omega^j_{\mathcal C}(M)$ as  follows : an object $K\in \Omega^j_{\mathcal C}(M)$ if and only
if there exists an exact sequence 
\begin{equation}
0\longrightarrow K\longrightarrow C_{j-1}\longrightarrow ... \longrightarrow C_0\longrightarrow M\longrightarrow 0
\end{equation} with $C_i\in \mathcal C$ for all $0\leq i\leq j-1$. Further, we set:
\begin{equation}
\Omega_{\mathcal C}^0(M):=\{M\}\qquad \Omega_{\mathcal C}^j(M)=\phi \qquad\forall\textrm{ }j<0
\end{equation}   We now fix some $n\geq 1$ and  let $\mathcal C$ be an $n$-resolving subcategory. For $j\geq 1$,
we now set:
\begin{equation}\label{8.3pn}
\mathcal C_{(j)}:=\{\mbox{$M\in \mathcal A$ $\vert$ $\Omega_{\mathcal C}^{j-1}(M)\subseteq \mathcal C$}\}=
\{\mbox{$M\in \mathcal A$ $\vert$ $\Omega_{\mathcal C}^{j-1}(M)\cap \mathcal C\ne \phi$}\}
\end{equation} We notice that the equality in \eqref{8.3pn} follows from Proposition \ref{P6.2}, since $\mathcal C$
is a resolving subcategory. 

\begin{thm}\label{P8.1} Let $\mathcal C\subseteq \mathcal A$ be an 
$n$-resolving subcategory. Then, we have an increasing chain:
\begin{equation}
\mathcal C=\mathcal C_{(1)}\subseteq \mathcal C_{(2)}\subseteq ... \subseteq \mathcal C_{(n)}\subseteq 
\mathcal C_{(n+1)}=\mathcal A
\end{equation} where each $\mathcal C_{(j)}$ is an $(n-j+1)$-resolving subcategory. Further, for any $k\geq 1$, 
we have $(\mathcal C_{(j)})_{(k)}=\mathcal C_{(j+k-1)}$. 
\end{thm} 

\begin{proof}
We consider $j\geq 1$ and some object $M\in \mathcal C_{(j)}$. Since $\Omega_{\mathcal C}^{j-1}(M)\subseteq \mathcal C$, we have
an exact sequence
\begin{equation}\label{rv8.5}
0\longrightarrow  C_{j-1}\longrightarrow ... \longrightarrow C_0\longrightarrow M\longrightarrow 0
\end{equation} with $C_i\in \mathcal C$ for $0\leq i\leq j-1$. We now pick some $K\in\Omega^j_{\mathcal C}(M)$. Then, we have an exact
sequence 
\begin{equation}\label{rv8.6}
0\longrightarrow K\longrightarrow C'_{j-1}\longrightarrow ... \longrightarrow C'_0\longrightarrow M\longrightarrow 0
\end{equation} with $C'_i\in \mathcal C$ for $0\leq i\leq j-1$. Applying Proposition \ref{P6.2} to the exact sequences in 
\eqref{rv8.5} and \eqref{rv8.6} and using the fact that $0\in \mathcal C$, we get $K\in \mathcal C$. Hence, 
$\Omega^j_{\mathcal C}(M)\subseteq \mathcal C$ and $\mathcal C_{(j)}\subseteq \mathcal C_{(j+1)}$. From the definitions,
it is also evident that $\mathcal C=\mathcal C_{(1)}$ and $\mathcal C_{(n+1)}=\mathcal A$. 

\smallskip
From \cite[Proposition 2.3(2)]{Stov}, we know that each ${\mathcal C}_{(j)}$ is a resolving subcategory.  To show that 
it is $(n-j+1)$-resolving, we pick $M\in \mathcal A$ and construct an exact sequence
\begin{equation}\label{rv8.7}
0\longrightarrow K\longrightarrow C''_{n-j}\longrightarrow ... \longrightarrow C''_0\longrightarrow M\longrightarrow 0
\end{equation} with each $C''_i\in \mathcal C\subseteq \mathcal C_{(j)}$. We note that $K\in\Omega_{\mathcal C}^{n-j+1}(M)$ and it is clear from the definitions that
$\Omega_{\mathcal C}^{j-1}(K)\subseteq \Omega^n_{\mathcal C}(M)\subseteq \mathcal C$. It follows that
$K\in \mathcal C_{(j)}$. Hence, ${\mathcal C}_{(j)}$ is $(n-j+1)$-resolving. It is also clear from the definitions
that $(\mathcal C_{(j)})_{(k)}=\mathcal C_{(j+k-1)}$ for $k\geq 1$. 
\end{proof}

\begin{defn}\label{D8.2} Let $G$ be a generator for $\mathcal A$ and let $n\geq 1$. A $G$-sequence of
length $n$ will be an $n$-tuple $\tilde{Y}=(Y_1,Y_2,...,Y_n)$ of  subsets of $\mathfrak Spec(\mathcal A)$ satisfying
the following conditions. 

\smallskip
(a) $Y_1\supseteq Y_2\supseteq ... \supseteq Y_n$ 

(b) $Ass(\mho_{i-1}(G))\cap Y_i=\phi$ for all $1\leq i\leq n$. 
\end{defn}

\begin{thm}
\label{P8.3} Let $\tilde{Y}=(Y_1,Y_2,...,Y_n)$ be a $G$-sequence of length $n\geq 1$. We set
\begin{equation}\label{req8.8}
\mathcal C(\tilde{Y}):=\{\mbox{$M\in \mathcal A$ $\vert$ $Ass(\mho_{i-1}(M))\cap Y_i=\phi$ for all $1\leq i\leq n$}\}
\end{equation} Then, $\mathcal C(\tilde{Y})$ is an $n$-resolving subcategory of $\mathcal A$ that is closed under arbitrary direct sums and  
$G\in \mathcal C(\tilde{Y})$.
\end{thm}

\begin{proof} It is clear from the definition in \eqref{req8.8} that $\mathcal C(\tilde{Y})$ is closed
under arbitrary direct sums. 
Using Definition \ref{D8.2}, it follows  that $G\in \mathcal C(\tilde{Y})$ and hence the direct sum 
$G^{I}$ of copies of $G$ for any set $I$ lies in $\mathcal C(\tilde{Y})$. Since $G$ is a generator of $\mathcal A$, it follows that
every object $M$ in $\mathcal A$ is the target of an epimorphism from some $G^{I}\in \mathcal C(\tilde{Y})$. Further, if
$M'\in \mathcal A$ is a direct summand of $M$, then $\mho_k(M')$ is a direct summand of $\mho_k(M)$ for any
$k\in \mathbb Z$ and it follows from \eqref{req8.8} that $M\in \mathcal C(\tilde{Y})$ implies $M'\in \mathcal C(\tilde{Y})$. 

\smallskip
We now consider a short exact sequence $0\longrightarrow M'\longrightarrow M\longrightarrow M''\longrightarrow 0$ in
$\mathcal A$ 
with $M''\in \mathcal C(\tilde{Y})$. Then, if $M\in \mathcal C(\tilde{Y})$, it follows from Corollary \ref{Cor7.10} that
\begin{equation}
(Ass(\mho_{i-1}(M'))\cap Y_i)\subseteq (Ass(\mho_{i-1}(M)\cap Y_i))\cup (Ass(\mho_{i-2}(M''))\cap Y_i)=\phi
\end{equation} and hence $M'\in \mathcal C(\tilde{Y})$. Conversely, if $M'\in \mathcal C(\tilde{Y})$, we get
from Corollary \ref{Cor7.10} that
\begin{equation}
(Ass(\mho_{i-1}(M))\cap Y_i)\subseteq (Ass(\mho_{i-1}(M')\cap Y_i))\cup (Ass(\mho_{i-1}(M''))\cap Y_i)=\phi
\end{equation}  This gives $M\in \mathcal C(\tilde{Y})$. Since $\mathcal C(\tilde{Y})$ contains all  direct sums of copies of the generator $G$, we can construct
for any $M\in \mathcal A$ an exact sequence
\begin{equation}
0\longrightarrow K_n\longrightarrow K_{n-1}\longrightarrow ... \longrightarrow K_0\longrightarrow M\longrightarrow 0
\end{equation} with $K_j\in \mathcal C(\tilde{Y})$ for $0\leq j\leq n-1$. We claim that $K_n\in \mathcal C(\tilde{Y})$. 
For this, we note from Corollary \ref{Cor7.10} that 
\begin{equation}\label{8.12}
Ass(\mho_{i-1}(K_n))\subseteq \left(\bigcup_{j=0}^{n-1}Ass(\mho_{j+i-n}(K_{j}))\right)\cup Ass(\mho_{i-n-1}(M))
\end{equation} Since $K_j\in \mathcal C(\tilde{Y})$ for $j<n$, we obtain 
\begin{equation}
Ass(\mho_{j+i-n}(K_{j}))\cap Y_i\subseteq Ass(\mho_{j+i-n}(K_{j}))\cap Y_{j+i-n+1}=\phi
\end{equation} Further, since $i\leq n$, we must have $Ass(\mho_{i-n-1}(M))=\phi$ and it follows from \eqref{8.12}
that $K_n\in \mathcal C(\tilde{Y})$. 
\end{proof}

Given a $G$-sequence $\tilde{Y}=(Y_1,...,Y_n)$ and some $1\leq j\leq n$, we observe that the truncated sequence 
$\tilde{Y}_{(j)}:=(Y_j,...,Y_n)$ is a $G$-sequence of length $n-j+1$. Indeed, we have 
$Ass(\mho_{i-1}(G))\cap Y_{i+j-1}\subseteq Ass(\mho_{i-1}(G))\cap Y_{i}=\phi$. This leads to the next result. 

\begin{thm}\label{P8.5} Let 
$\tilde{Y}=(Y_1,...,Y_n)$ be a $G$-sequence of length $n\geq 1$. For each $1\leq j\leq n$, we have:
\begin{equation}
\mathcal C(\tilde{Y})_{(j)}=\mathcal C(\tilde{Y}_{(j)})=\mathcal C(Y_j,...,Y_n)
\end{equation}
\end{thm}

\begin{proof} The case of $j=1$ is obvious and so we assume that $j\geq 2$. 
We consider some $M\in \mathcal C(\tilde{Y})_{(j)}$. Then, there is an exact sequence 
$0\longrightarrow K_{j-1}\longrightarrow ... \longrightarrow K_0\longrightarrow M\longrightarrow 0$ with
each $K_i\in \mathcal C(\tilde{Y})$. From Corollary \ref{Cor7.10}, it follows that
\begin{equation}
Ass(\mho_{l-1}(M))\subseteq \bigcup_{i=0}^{j-1}Ass(\mho_{l-1+i}(K_i))
\end{equation} Since $i\leq j-1$, we notice that each $Ass(\mho_{l-1+i}(K_i))\cap Y_{l+j-1} 
\subseteq Ass(\mho_{l-1+i}(K_i))\cap Y_{l+i}=\phi$ which shows that $M\in \mathcal C(Y_j,...,Y_n)$. 

\smallskip
Conversely, consider $M\in \mathcal C(Y_j,...,Y_n)$. Since $G$ is a generator, we can form an exact sequence
$0\longrightarrow K\longrightarrow G_{j-2}\longrightarrow ... \longrightarrow G_0\longrightarrow M\longrightarrow 0$ where $G_i$ is a 
direct sum of copies of $G$ for all $0\leq i<j-1$. In order to show that $M\in \mathcal C(\tilde{Y})_{(j)}$, it suffices
to check that $K\in \mathcal C(\tilde{Y})$. Again, it follows from Corollary \ref{Cor7.10} that
\begin{equation}
Ass(\mho_{l-1}(K))\subseteq \left(\bigcup_{i=0}^{j-2}Ass(\mho_{l-j+i+1}(G_i))\right)\cup Ass(\mho_{l-j}(M))
\end{equation} Since $i\leq j-2$, we notice that $Ass(\mho_{l-j+i+1}(G_i))\cap Y_l
\subseteq Ass(\mho_{l-j+i+1}(G_i))\cap Y_{l-j+i+2}=\phi$. Since $M\in \mathcal C(Y_j,...,Y_n)$, we also
know that $Ass(\mho_{l-j}(M))\cap Y_l=\phi$. This shows that $K \in \mathcal C(\tilde{Y})$. 
\end{proof}

\begin{cor}\label{C8.4}
Let $G$ be a generator for $\mathcal A$ and let $\tilde{Y}=(Y_1,...,Y_n)$ be a $G$-sequence
of length $n$. Then, for each $1\leq j\leq n$, the subcategory $\mathcal C(\tilde{Y})_{(j)}$ is closed under taking injective hulls
and arbitrary direct sums. Additionally, the subcategory $\mathcal C(\tilde{Y})_{(n)}$ satisfies
\begin{equation}\label{shel}
\mathcal C(\tilde{Y})_{(n)}=\{\mbox{$M\in \mathcal A$ $\vert$ $Ass(M)\subseteq Ass(\mathcal C(\tilde{Y})_{(n)})$}\}
\end{equation} 
\end{cor}

\begin{proof}  If $E\in \mathcal A$ is an injective object, then $\mho_k(E)=0$ for all $k\ne 0$. Hence, $Ass(\mho_{i-1}(E))
\cap Y_i=\phi$ for any $1<i\leq n$. Suppose now that $E=E(M)$, the injective hull of some 
object $M\in \mathcal C(\tilde{Y})$. Then, $Ass(\mho_0(E))=Ass(E(M))=Ass(M)$ does not intersect 
$Y_1$ and hence $E=E(M)\in \mathcal C(\tilde{Y})$. Hence, $\mathcal C(\tilde{Y})$ is closed under taking injective hulls. 
We have already mentioned in Proposition \ref{P8.3} that $\mathcal C(\tilde{Y})$ is closed under arbitrary direct sums.
As a result of Proposition \ref{P8.5}, this also shows that each $\mathcal C(\tilde{Y})_{(j)}$ is closed under taking injective hulls
and arbitrary direct sums. 

\smallskip
Finally, for $j=n$,  we see that $\mathcal C(\tilde{Y})_{(n)}=\mathcal C(Y_n)=\{\mbox{$M\in \mathcal A$ 
$\vert$ $Ass(M)\cap Y_n=\phi$}\}$. The result of \eqref{shel} follows easily from this.

\end{proof}

\begin{lem}\label{L8.65} (a) If $0\ne E\in \mathcal A$ is an indecomposable injective, then $E$ is the injective hull
$E=E(P)$ of some $P\in Spec(\mathcal A)$.

(b) Choose a point $\langle P\rangle \in \mathfrak Spec(\mathcal A)$ and set:
\begin{equation}\label{8.18tco}
Inj(\langle P\rangle):=\{\mbox{$E(M)$ $\vert$ $M\in \mathcal A$ and $Ass(M)=\{\langle P\rangle\}$}\}
\end{equation} Then, the collection $Inj(\langle P\rangle)$ contains an indecomposable injective $E_{\langle P\rangle}$
and every object in $Inj(\langle P\rangle)$ is the direct sum of copies of $E_{\langle P\rangle}$. 
\end{lem}
\begin{proof} (a) Since $E\ne 0$, we can find some spectral object $P\in Spec(\mathcal A)$ along with an embedding
$P\hookrightarrow E$. Then, $E(P)$ is a direct summand of $E$. Since $E$ is indecomposable, we must have
$E=E(P)$. 

\smallskip
(b) We consider $M\in \mathcal A$ with $Ass(M)=Ass(E(M))=\{\langle P\rangle\}$. Since $\mathcal A$ is locally noetherian,
a well known result of Matlis (see, for instance, \cite[Proposition V.4.5]{Bo}) shows that every injective in
$\mathcal A$ can be expressed as a direct sum of indecomposable injectives. We choose one such indecomposable
injective $E_{\langle P\rangle}$ from the direct sum decomposition of $E(M)$. Then, $\phi\ne 
Ass(E_{\langle P\rangle})\subseteq Ass(E(M))=\{\langle P\rangle\}$ and hence $E_{\langle P\rangle }\in Inj(\langle P
\rangle)$. 

\smallskip
We now suppose that $E_1$ and $E_2$ are two indecomposable injectives in $Inj(\langle P\rangle)$. From part (a),
we can find $Q_1$, $Q_2\in Spec(\mathcal A)$ such that $E_1=E(Q_1)$ and $E_2=E(Q_2)$ respectively. It is clear 
that $\{\langle Q_1\rangle\} =Ass(E_1)=Ass(E_2)=\{\langle Q_2\rangle\}=\{\langle P\rangle\}$. We consider:
\begin{equation}
\begin{CD}
\mathcal A @>L:=L_{\langle Q_1\rangle}=L_{\langle Q_2\rangle}>> \mathcal A/\langle Q_1\rangle =
\mathcal A/\langle P\rangle=\mathcal A/\langle Q_2\rangle @>i:=i_{\langle Q_1\rangle}=i_{\langle Q_2\rangle}>> \mathcal A
\end{CD}
\end{equation} From \cite[Corollaire III.2]{Gab}, we know that $0\ne iL(E_1)$ (resp. $0\ne iL(E_2)$) is a direct summand of 
 the injective $E_1$ (resp. $E_2$). Since $E_1$ and $E_2$ are indecomposable, we get $E_1=iL(E_1)$
and $E_2=iL(E_2)$. It now follows from Lemma \ref{L7.5}(c) that 
\begin{equation}\label{le8.23}
E_1=iL(E_1)=iL_{\langle Q_1\rangle}(E(Q_1))=iE(L_{\langle Q_1\rangle}(Q_1))\quad 
E_2=iL(E_2)=iL_{\langle Q_2\rangle}(E(Q_2))=iE(L_{\langle Q_2\rangle}(Q_2))
\end{equation} We notice that both $L_{\langle Q_1\rangle}(Q_1)$ and 
$L_{\langle Q_2\rangle}(Q_2)$ are quasi-final objects of $\mathcal A/\langle P\rangle$. It now follows as in Section 7 that
we can write $L_{\langle Q_1\rangle}(Q_1)=S_{\langle P\rangle}^m$ and 
$L_{\langle Q_2\rangle}(Q_2)=S_{\langle P\rangle}^n$ where $S_{\langle P\rangle}$ is the simple object in $\mathcal A/\langle P\rangle$. Combining with \eqref{le8.23}, we get
\begin{equation}
E_1=(iE(S_{\langle P\rangle}))^m \qquad E_2=(iE(S_{\langle P\rangle}))^n
\end{equation} Since $E_1$ and $E_2$ are indecomposable, we now have $m=n=1$ and $E_1=iE(S_{\langle P\rangle})
=E_2$. Hence, there is a unique indecomposable injective $E_{\langle P\rangle}$ in $Inj(\langle P\rangle)$. Finally, since
every $E\in Inj(\langle P\rangle)$ decomposes as a direct sum of indecomposable injectives and every nonzero direct summand of $E\in Inj(\langle P\rangle)$ lies in $Inj(\langle P\rangle)$, the result follows. 
\end{proof}

Before proceeding further, we record here the following result: we recall that the points of the Ziegler spectrum $Zg(\mathcal A)$ of a locally noetherian Grothendieck category $\mathcal A$ are given by  isomorphism classes of indecomposable injectives in $\mathcal A$. A basis of open sets for 
$Zg(\mathcal A)$ is given by the collection of subsets of the form (see, for instance, \cite{Herz})
\begin{equation}
\mathcal U(M):=\{\mbox{$E\in Zg(\mathcal A)$ $\vert$ $Hom_{\mathcal A}(M,E)\ne 0$}\}
\end{equation}
as $M$ varies over all finitely generated objects of $\mathcal A$. 

\begin{thm}\label{P8.65jq} There is a bijection between the points of $\mathfrak Spec(\mathcal A)$ and points of the Ziegler spectrum
$Zg(\mathcal A)$.
\end{thm}

\begin{proof} Given an indecomposable injective $E$ in $\mathcal A$, it follows from Lemma \ref{L8.65}(a) that there exists $P\in Spec(\mathcal A)$ such that $E=E(P)$. We claim that this association gives a well-defined map    $Zg(\mathcal A)\longrightarrow \mathfrak Spec(\mathcal A)$. Indeed, if 
$E\in Zg(\mathcal A)$ is such that $E=E(P_1)=E(P_2)$ for $P_1$, $P_2\in Spec(\mathcal A)$, we have $\{\langle P_1\rangle \}= Ass(P_1)=Ass(E(P_1))=Ass(E(P_2))=\{\langle P_2\rangle\}$. 

\smallskip
Now suppose that we have $E_1$, $E_2\in Zg(\mathcal A)$ with $E_1=E(P_1)$ and $E_2=E(P_2)$ for $P_1$, $P_2\in Spec(\mathcal A)$ with
$\langle P_1\rangle =\langle P_2\rangle$. Then, 
$E_1$, $E_2\in Inj(\langle P_1\rangle)=Inj(\langle P_2\rangle)$ in the sense of \eqref{8.18tco}. By Lemma \ref{L8.65}(b), we know that $Inj(\langle P_1\rangle)$ contains an indecomposable injective $E_{\langle P\rangle}$ and that both $E_1$ and $E_2$ are direct sums of copies of $E_{\langle P\rangle}$. Since
$E_1$ and $E_2$ are indecomposable, it now follows that $E_1\cong E_2$, i.e., they are identical as points of $Zg(\mathcal A)$. This shows that the map
 $Zg(\mathcal A)\longrightarrow \mathfrak Spec(\mathcal A)$ is one-one. 
 
 \smallskip
 Finally, we take  $\langle P\rangle \in \mathfrak Spec(\mathcal A)$ and choose an indecomposable injective $E'_{\langle P\rangle }\in Inj(\langle P\rangle)$
 using Lemma \ref{L8.65}(b). Then, $Ass(E'_{\langle P\rangle})=\{\langle P\rangle\}$. From Lemma \ref{L8.65}(a), it now follows that there is some $P'\in Spec(\mathcal A)$ with $E'_{\langle P\rangle}=E(P')$.  Then,  $Ass(E'_{\langle P\rangle})=Ass(P')=\{\langle P'\rangle\}$.  Hence, $\langle P\rangle 
 =\langle P'\rangle$ and the map $Zg(\mathcal A)\longrightarrow \mathfrak Spec(\mathcal A)$  is onto. This proves the result.

\end{proof}

\begin{lem}\label{LL8.7}
Let $\mathcal C\subseteq \mathcal A$ be an $n$-resolving subcategory that is closed under injective hulls.  
Then, if $P\in Spec(\mathcal A)$ is such that there exists some $M\in \mathcal C$ admitting an inclusion
$P\hookrightarrow M$,  the object $E(i_{\langle P\rangle}(K(P)))\in \mathcal C$. 

\smallskip
Suppose additionally that $\mathcal C$ is closed under arbitrary direct sums. Then, the following hold:

\smallskip
(a) If $P\in Spec(\mathcal A)$ is a spectral object such that $\langle P\rangle\in Ass(\mathcal C)$, then $E(P)\in \mathcal C$.

(b) Let $N\in \mathcal A$. Then, the injective hull   $E(N)\in \mathcal C$ if and only if 
$Ass(N)\subseteq Ass(\mathcal C)$. 
\end{lem}

\begin{proof}  Suppose that we have a monomorphism
$P\hookrightarrow M$ with $P\in Spec(\mathcal A)$ and $M\in \mathcal C$. We extend this to a monomorphism
$P\hookrightarrow M\hookrightarrow E(M)$. Since $\mathcal C$ is closed under injective hulls, we know that $E(M)
\in \mathcal C$. Then, $K(P)= L_{\langle P\rangle}(P)\subseteq L_{\langle P\rangle}(E(M))$ and 
hence $i_{\langle P\rangle}(K(P))\subseteq i_{\langle P\rangle}L_{\langle P\rangle}(E(M))$. From \cite[Corollaire III.2]{Gab},
we know that $ i_{\langle P\rangle}L_{\langle P\rangle}(E(M))$ is a direct summand of the injective $E(M)$. Since 
$\mathcal C$ is a resolving subcategory, it follows from Definition \ref{D6.1} that the direct summand 
$ i_{\langle P\rangle}L_{\langle P\rangle}(E(M))\in \mathcal C$. Also $ i_{\langle P\rangle}L_{\langle P\rangle}(E(M))$ 
is injective and we see that the injective hull $E(i_{\langle P\rangle}(K(P)))$ is a direct summand of $ i_{\langle P\rangle}L_{\langle P\rangle}(E(M))$. 
It follows that $E(i_{\langle P\rangle}(K(P)))\in \mathcal C$. 

\smallskip
We now suppose additionally that $\mathcal C$ is closed under arbitrary direct sums. We consider
$P\in Spec(\mathcal A)$ with $\langle P\rangle \in Ass(\mathcal C)$. Then, there is some $Q\in Spec(\mathcal A)$ with
$\langle Q\rangle=\langle P\rangle$ such that there exists 
 a monomorphism $Q\hookrightarrow M$ with $M\in \mathcal C$. From the above, it follows that $E(i_{\langle Q\rangle}(K(Q)))\in \mathcal C$. 
We now consider an inclusion $Q'\hookrightarrow i_{\langle Q\rangle}(K(Q))$ with $Q'\in Spec(\mathcal A)$. Using adjointness, the morphism
$L_{\langle Q\rangle}(Q')\longrightarrow (K(Q))$ corresponding to this inclusion must be nonzero and hence
$L_{\langle Q\rangle}(Q')\ne 0$. Then, $L_{\langle Q\rangle}(Q')\subseteq L_{\langle Q
\rangle}i_{\langle Q\rangle}(K(Q))=K(Q)$ is an associated prime of $K(Q)=L_{\langle Q
\rangle}(Q)$. Hence,  $Ass(E(i_{\langle Q\rangle}(K(Q))))=
Ass(i_{\langle Q\rangle}(K(Q)))=\{\langle Q\rangle\}=\{\langle P\rangle\}$. 

\smallskip In the notation of Lemma \ref{L8.65}, we see that $E(i_{\langle Q\rangle}(K(Q)))\in 
\mathcal C\cap Inj(\langle P\rangle)$. Clearly,
$E(P)\in Inj(\langle P\rangle)$. Since $\mathcal C$ contains all direct summands and is closed under arbitrary
direct sums, it now follows from Lemma \ref{L8.65} that $E(P)\in \mathcal C$. This proves (a). 

\smallskip
To prove (b), we consider some $N\in \mathcal A$. Clearly, if $E(N)\in \mathcal C$, then $Ass(N)
=Ass(E(N))\subseteq \mathcal C$. Conversely, suppose that $Ass(N)\subseteq Ass(\mathcal C)$. Let $E$ be an indecomposable
injective appearing as a direct summand of $E(N)$. By Lemma \ref{L8.65}, we know that $E=E(P)$
for some $P\in Spec(\mathcal A)$ and it is clear that $\langle P\rangle \in Ass(E(N))=Ass(N)\subseteq Ass(\mathcal C)$. From 
part (a), it follows that $E=E(P)\in \mathcal C$. Since the injective $E(N)$ decomposes as a direct sum of 
indecomposable injectives and $\mathcal C$ is closed under direct sums, we get $E(N)\in \mathcal C$.  
\end{proof}

We now make a simple observation. Suppose that $\mathcal C$ is an $n$-resolving subcategory for some $n\geq 1$ and that
we have a short exact sequence 
\begin{equation}
0\longrightarrow M'\longrightarrow M\longrightarrow M''\longrightarrow 0
\end{equation} with $M\in \mathcal C$. Then, $M'\in \Omega^1_{\mathcal C}(M'')$. It follows from the definitions 
and from Proposition \ref{P6.2} that $M'\in \mathcal C$ if and only if $M''\in \mathcal C_{(2)}$. 

\begin{thm}\label{P8.7}
We fix $n\geq 1$ and let $\mathcal C\subseteq \mathcal A$ be a full subcategory satisfying the following conditions:

\smallskip
(1) $\mathcal C$ is an $n$-resolving subcategory.

(2) For each $j\geq 1$, the subcategory $\mathcal C_{(j)}$ is closed under injective hulls and arbitrary direct sums. 

(3) The subcategory $\mathcal C_{(n)}\subseteq \mathcal A$ 
satisfies $\mathcal C_{(n)}=\{\mbox{$M\in \mathcal A$ $\vert$ $Ass(M)\subseteq Ass(\mathcal C_{(n)})$}\}$.

\smallskip
For $1\leq j\leq n$, set $Y_j:=\mathfrak Spec(\mathcal A)\backslash 
Ass(\mathcal C_{(j)})$. Then:
\begin{equation*}
\begin{array}{c}
\mathcal C= \{\mbox{$M\in \mathcal A$ $\vert$ $Ass(E_{j-1}(M))\cap Y_j=\phi$ $\forall$ $1\leq j\leq n$}\}
=\{\mbox{$M\in \mathcal A$ $\vert$ $Ass(\mho_{j-1}(M))\cap Y_j=\phi$ $\forall$ $1\leq j\leq n$}\}\\
\end{array}
\end{equation*}
\end{thm}

\begin{proof} 

\smallskip We will prove the result by induction on $n$.  For the case of $n=1$, the result follows directly from 
assumption (3). We now consider $n>1$. Every object $M\in \mathcal A$ fits into a short exact sequence as follows:
\begin{equation}
0\longrightarrow M\longrightarrow E(M)\longrightarrow \mho(M)\longrightarrow 0
\end{equation} We know that $\mathcal C=\mathcal C_{(1)}$ is closed under injective hulls. Hence, from the observation above, we see that $M\in \mathcal C$ if and only 
if $E(M)\in \mathcal C$ and $\mho(M)\in \mathcal C_{(2)}$. 

\smallskip
From Lemma \ref{LL8.7}(b), we see that $E(M)\in \mathcal C$ if and only $Ass(M)\subseteq Ass(\mathcal C)$. On the other
hand, the subcategory $\mathcal C_{(2)}$ satisfies the assumptions (1)-(3) for $(n-1)$ and we obtain therefore that:
\begin{equation}
\mathcal C_{(2)}=\{\mbox{$N\in \mathcal A$ $\vert$ $Ass(\mho_{j-2}(N))\cap (\mathfrak Spec(\mathcal A)\backslash 
Ass(\mathcal C_{(j)}))=\phi$ $\forall$ $2\leq j\leq n$}\}
\end{equation} In particular, we have $\mho(M)\in \mathcal C_{(2)}$ if and only if
\begin{equation}
Ass(\mho_{j-1}(M))\cap  (\mathfrak Spec(\mathcal A)\backslash 
Ass(\mathcal C_{(j)}))=\phi \quad\forall\textrm{ }2\leq j\leq n
\end{equation} This proves the result.
\end{proof}

\begin{thm}\label{P8.9} We fix a generator $G$ of $\mathcal A$ and some $n\geq 1$. Let $\mathcal C\subseteq \mathcal A$ be an
$n$-resolving subcategory such that $G\in \mathcal C$ and each $\{\mathcal C_{(j)}\}_{1\leq j\leq n}$
is closed under injective hulls. Then, the tuple
\begin{equation}
(\mathfrak Spec(\mathcal A)\backslash Ass(\mathcal C_{(1)}), ...., \mathfrak Spec(\mathcal A)\backslash Ass(\mathcal C_{(n)}))
\end{equation} is a $G$-sequence of length $n$. 
\end{thm}
\begin{proof}
We set $Y_j:=\mathfrak Spec(\mathcal A)\backslash Ass(\mathcal C_{(j)})$ for $1\leq j\leq n$. Since each 
$\mathcal C_{(j)}\subseteq \mathcal C_{(j+1)}$, it is clear that $Y_1\supseteq ... \supseteq Y_n$. We claim that
$\mho_{j-1}(G)\in \mathcal C_{(j)}$ for each $1\leq j \leq n$. We already know that the result is true for 
$j=1$, i.e., $G\in \mathcal C=\mathcal C_{(1)}$. We proceed by induction and consider for $j>1$ the short exact sequence:
\begin{equation}\label{eqr8.26}
0\longrightarrow \mho_{j-2}(G)\longrightarrow E(\mho_{j-2}(G))\longrightarrow \mho_{j-1}(G)\longrightarrow 0
\end{equation} By induction assumption, we have $\mho_{j-2}(G)\in \mathcal C_{(j-1)}$. Since $\mathcal C_{(j-1)}$ is closed under injective hulls, we have $ E(\mho_{j-2}(G))\in \mathcal C_{(j-1)}$.  The exact sequence in \eqref{eqr8.26} now  shows that
$\mho_{j-1}(G)\in (\mathcal C_{(j-1)})_{(2)}=\mathcal C_{(j)}$. This proves the result. 
\end{proof}

We will now describe the most important assumption that we will make in this section. For any spectral 
object $P\in Spec(\mathcal A)$, we have already noted that the full subcategory $\langle P\rangle=\{\mbox{$M\in \mathcal A$
$\vert$ $P\not\prec M$}\}$ is a localizing subcategory of $\mathcal A$. In other words, there is a section 
functor $i_{\langle P\rangle}:\mathcal A/\langle P\rangle\longrightarrow\mathcal A$ that is right adjoint to the canonical functor
$L_{\langle P\rangle}:\mathcal A\longrightarrow \mathcal A/\langle P\rangle$. 

\smallskip
Being a right adjoint, we know that the section functor of any localizing subcategory is always
left exact. For the rest of this section, we will assume that the section functors
\begin{equation}\label{dreamexact}
\{\mbox{$i_{\langle P\rangle}:\mathcal A/\langle P\rangle\longrightarrow\mathcal A$ $\vert$ $P\in Spec(\mathcal A)$}\}
\end{equation}
are all exact functors.  In fact, the case of localizing subcategories with section functors that are exact is an  important special case
in the theory of localizing subcategories (see Gabriel \cite[$\S$ III.3]{Gab}).

\begin{lem}\label{L8.6} Let $P\in Spec(\mathcal A)$ be a spectral object. 

\smallskip
(a) Let $M$ be an object of $\mathcal A/\langle P\rangle$.  Let $\{\mho_k(M)\}_{k\geq 0}$ be the cosygyzies
of $M$ in $\mathcal A/\langle P\rangle$ and let $\{\mho_k(i_{\langle P\rangle}(M))\}_{k\geq 0}$ be the cosyzygies
of $i_{\langle P\rangle}(M)$ in $\mathcal A$. Then:
\begin{equation}
i_{\langle P\rangle}(\mho_k(M))=\mho_k(i_{\langle P\rangle}(M))\qquad\forall\textrm{ }k\geq 0
\end{equation}

\smallskip
(b) Let $K(P)=L_{\langle P\rangle}(P)$. For any $k\geq 0$, we have $Ass(\mho_k(i_{\langle P\rangle}(K(P))))\subseteq \{
\langle P\rangle\}$. 

\end{lem}

\begin{proof} (a) Let $E(M)$ be the injective hull of $M\in \mathcal A/\langle P\rangle$. Then $L_{\langle P\rangle}$
being exact and $i_{\langle P\rangle}$ being its right adjoint, we know
that $i_{\langle P\rangle}(E(M))$ is injective in $\mathcal A$. Also $i_{\langle P\rangle}(M)\subseteq 
i_{\langle P\rangle}(E(M))$. Let $E'\subseteq i_{\langle P\rangle}(E(M))$ be the injective
hull of $i_{\langle P\rangle}(M)$ in $\mathcal A$. Then, $i_{\langle P\rangle}(E(M))$ splits as a direct sum $i_{\langle P\rangle}(E(M))
=E'\oplus E''$. We notice that $E''\cap i_{\langle P\rangle}(M)=0$. This gives:
\begin{equation}\label{namo8}
\begin{array}{ccc}
\begin{CD}
0 @>>> E'' \\
@VVV @VVV \\
i_{\langle P\rangle}(M) @>>> i_{\langle P\rangle}(E(M)) \\
\end{CD} & \Rightarrow &  
\begin{CD}
0 @>>> L_{\langle P\rangle}(E'') \\
@VVV @VVV \\
M=L_{\langle P  \rangle}i_{\langle P\rangle}(M) @>>> L_{\langle P\rangle}i_{\langle P\rangle}(E(M)) =E(M) \\
\end{CD}\\
\end{array}
\end{equation}  The implication in \eqref{namo8} follows from the fact that $L_{\langle P\rangle}$ is exact. 
From \eqref{namo8} we have $M\cap L_{\langle P\rangle}(E'')=0$ and $M$ being an essential subobject
of $E(M)$, we obtain $L_{\langle P\rangle}(E'')=0$. We now consider the commutative diagram
\begin{equation}\label{namo9}
\begin{CD}
E'' @>>> i_{\langle P\rangle}L_{\langle P\rangle}(E'')=0 \\
@VVV @VVV \\
i_{\langle P\rangle}(E(M)) @>>>  i_{\langle P\rangle}L_{\langle P\rangle}i_{\langle P\rangle}(E(M))=i_{\langle P
\rangle}(E(M))\\
\end{CD}
\end{equation} The left vertical arrow $E''\hookrightarrow i_{\langle P\rangle}(E(M))$ in \eqref{namo9} is 
a monomorphism and the bottom horizontal arrow is the identity $i_{\langle P\rangle}(E(M))\longrightarrow i_{\langle P\rangle}(E(M))$. Hence, $E''=0$ and $i_{\langle P\rangle}(E(M))$ becomes the injective hull of $i_{\langle P\rangle}(M)$ in 
$\mathcal A$. 
 Since $i_{\langle P\rangle}$
is assumed to be exact, it is now clear from Definition \ref{D7.1} that the cosygyzies satisfy $i_{\langle P\rangle}(\mho_k(M))=\mho_k(i_{\langle P\rangle}(M))$. 

\smallskip
(b) For any object $M\in \mathcal A/\langle P\rangle$, we have a short exact sequence
\begin{equation}\label{eq8.19}
0\longrightarrow M\longrightarrow E(M)\longrightarrow \mho(M)=\mho_1(M)\longrightarrow 0
\end{equation} where $E(M)$ is the injective hull of $M$. Since $Supp(E(M))=Supp(M)$, we observe from \eqref{eq8.19}
that $Supp(\mho(M))\subseteq Supp(M)$.  More generally, since $\mho_{k+1}(M)=\mho(\mho_k(M))$ for $k\geq 0$,
it follows that all cosyzygies satisfy $Supp(\mho_k(M))\subseteq Supp(M)$.

\smallskip
In particular, we consider $M=K(P)=L_{\langle P\rangle}(P)$. Since $K(P)$ is the quasi-final object of
$\mathcal A/\langle P\rangle$, we have $Supp(K(P))=\{\langle L_{\langle P\rangle}(P)\rangle\}$. Then, for any $k\geq 0$, we have
\begin{equation}\label{eq8.20}
Ass(\mho_k(K(P)))\subseteq Supp(\mho_k(K(P)))\subseteq Supp(K(P))= \{\langle L_{\langle P\rangle}(P)\rangle\}
\end{equation}
Now, the result of (b) is obvious if $\mho_k(i_{\langle P\rangle}(K(P)))=0$ and hence we may restrict restrict to
those $k\in \mathbb Z$ for which $\mho_k(i_{\langle P\rangle}(K(P)))\ne 0$. Then, from part (a), we get
$\mho_k(K(P))\ne 0$. Combining with \eqref{eq8.20}, we see that $\phi\ne Ass(\mho_k(K(P)))=\{\langle L_{\langle P\rangle}(P)
\rangle\}$. 

\smallskip
Now suppose that we have  $Q\in Spec(\mathcal A)$ with a monomorphism $Q\hookrightarrow i_{\langle P
\rangle}(\mho_k(K(P)))$. By adjointness, the nonzero morphism $Q\hookrightarrow i_{\langle P
\rangle}(\mho_k(K(P)))$ corresponds to a nonzero morphism $L_{\langle P\rangle}(Q)\longrightarrow \mho_k(K(P))$ 
and hence $L_{\langle P\rangle}(Q)\ne 0$. It follows from \cite[$\S$ 8.5.4]{Rose2} that $L_{\langle P\rangle}(Q)
\in Ass(\mho_k(K(P)))$ and hence $\langle L_{\langle P\rangle}(Q)\rangle=\langle L_{\langle P\rangle}(P)\rangle$ in
$\mathfrak Spec(\mathcal A/\langle P\rangle)$. Treating $\mathfrak Spec(\mathcal A/\langle P\rangle)$ as a subset of 
 $\mathfrak Spec(\mathcal A)$ as before, this gives $\langle P\rangle =\langle Q\rangle\in\mathfrak Spec(\mathcal A)$. Hence,
$Ass(i_{\langle P\rangle}(\mho_k(K(P))))=Ass(\mho_k(i_{\langle P\rangle}(K(P))))=\{\langle P\rangle\}$. 
\end{proof}

\begin{thm}\label{Pe8.10} Fix a generator $G\in \mathcal A$ and let $\tilde{Y}=(Y_1,...,Y_n)$, 
$\tilde{Y}'=(Y_1',...,Y_n')$ be two $G$-sequences of length $n$. Then, if
$\mathcal C(\tilde{Y})=\mathcal C(\tilde{Y}')$, we must have $\tilde{Y}=\tilde{Y}'$. 
\end{thm}

\begin{proof}
We will prove the result by induction on $n$. Suppose that $n=1$. Then, we know that
\begin{equation}\mathcal C(\tilde{Y})=\{\mbox{$M\in \mathcal A$ $\vert$ $Ass(M)\cap Y_1=\phi$}\}
=\{\mbox{$M\in \mathcal A$ $\vert$ $Ass(M)\cap Y_1'=\phi$}\}
=\mathcal C(\tilde{Y}')
\end{equation} If $Y_1\ne Y_1'$, we can choose  $P\in Spec(\mathcal A)$ with $\langle P
\rangle \in (Y_1\backslash Y_1')\cup (Y_1'\backslash Y_1)$. Then, the object $P$ lies in exactly 
one of $\mathcal C(\tilde{Y})$ and $\mathcal C(\tilde{Y}')$, which is a contradiction.

\smallskip
We now consider $n>1$. Then, $\mathcal C(\tilde{Y})=\mathcal C(\tilde{Y}')$ implies $\mathcal C(\tilde{Y})_{(2)}=\mathcal C(\tilde{Y}')_{(2)}$ which gives $\mathcal C(Y_2,...,Y_n)=\mathcal C(Y_2',...,Y_n')$. By induction assumption for $(n-1)$,
we have $Y_j=Y_j'$ for $2\leq j\leq n$. It remains to show that $Y_1=Y_1'$. 

\smallskip
Suppose otherwise. For the sake of definiteness, we suppose that $Y_1'\backslash Y_1$ is nonempty and
choose $P\in Spec(\mathcal A)$ with $\langle P\rangle\in Y_1'\backslash Y_1$. Since $\langle P\rangle\notin Y_1$, it follows
from Lemma \ref{L8.6} that
\begin{equation}
Ass(\mho_{j-1}(i_{\langle P\rangle}(K(P)))\cap Y_j\subseteq \{\langle P\rangle\}\cap Y_1=\phi
\end{equation} Then, we see that $i_{\langle P\rangle}(K(P))\in \mathcal C(\tilde{Y})$. 
On the other hand, we have $Ass(i_{\langle P\rangle}(K(P)))\cap Y_1'=\{\langle P\rangle\}\cap Y_1'\ne \phi$ and hence
$i_{\langle P\rangle}(K(P))\notin \mathcal C(\tilde{Y}')$ which is a contradiction. 

\end{proof}

Putting together the results of Proposition \ref{P8.3}, Corollary \ref{C8.4} and Propositions \ref{P8.7}, 
\ref{P8.9} and \ref{Pe8.10} we obtain the following bijective correspondence. 

\begin{thm}\label{PR8.11}
Fix a generator $G$ of $\mathcal A$ and some $n\geq 1$. Then, the associations 
\begin{equation}
\begin{array}{c}
\tilde{Y}=(Y_1,...,Y_n)\mapsto \mathcal C(\tilde{Y}) = \{\mbox{$M\in \mathcal A$ $\vert$ $Ass(\mho_{i-1}(M))\cap Y_i=\phi$ for all $1\leq i\leq n$}\}\\
\mathcal C\mapsto (\mathfrak Spec(\mathcal A)\backslash Ass(\mathcal C_{(1)}), ...., \mathfrak Spec(\mathcal A)\backslash Ass(\mathcal C_{(n)}))\\
\end{array}
\end{equation} define mutually inverse correspondences between $G$-sequences $\tilde{Y}$ of length $n$ and 
$n$-resolving subcategories $\mathcal C\subseteq \mathcal A$ satisfying the following conditions: 

\smallskip
(1) $G\in \mathcal C$.

(2) For each $j\geq 1$, the subcategory $\mathcal C_{(j)}$ is closed under injective hulls and arbitrary direct sums. 

(3) The subcategory $\mathcal C_{(n)}\subseteq \mathcal A$ 
satisfies $\mathcal C_{(n)}=\{\mbox{$M\in \mathcal A$ $\vert$ $Ass(M)\subseteq Ass(\mathcal C_{(n)})$}\}$.
\end{thm}

\section{Subcategories closed under subobjects, direct sums and essential extensions}

\smallskip

\smallskip
In this section, we will describe a one-one correspondence between arbitrary subsets of $\mathfrak Spec(\mathcal A)$ and certain subcategories of $\mathcal A_{fg}$. 
We recall
that $\mathcal A_{fg}$ is the full subcategory of finitely generated objects of $\mathcal A$ and that 
$\mathcal A_{fg}$ is an abelian category.  We will say
that a subcategory $\mathcal C\subseteq \mathcal A_{fg}$ is closed under essential extensions if given $M\in \mathcal C$ and 
$N\in \mathcal A_{fg}$ with $M\subseteq N$ an essential extension, we must have $N\in \mathcal C$. 

\begin{lem}\label{L9.1}
Let $\mathcal C\subseteq \mathcal A$ be a full subcategory that is closed under subobjects, finite direct sums and essential extensions. 
Let $Q\hookrightarrow M$ be an inclusion with $Q\in Spec(\mathcal A)$ and $M\in \mathcal C$.  Then, every object
$N\in \mathcal A_{fg}$ such that $Ass(N)=\{\langle Q\rangle\}$ lies in $\mathcal C$.  
\end{lem}

\begin{proof}
 Since $\mathcal C$ is closed under subobjects, we notice that $Q\in \mathcal C$. We consider the injective hull
$E(Q)$. From Lemma \ref{L8.65}, it follows that we can find $P\in Spec(\mathcal A)$ such that 
$E(P)$ is an indecomposable injective and $E(Q)$ is a direct sum of copies of $E(P)$. It is also clear that
$\{\langle P\rangle\}=Ass(E(P))=Ass(E(Q))=\{\langle Q\rangle\}$. 

\smallskip
Since $Q$ is an essential subobject of $E(Q)$ and we have $P\subseteq E(P)\subseteq E(Q)$, we can find
$0\ne P'\subseteq P$ such that $P'\subseteq Q$. Then, $P'\in \mathcal C$ and since $P$ is spectral, we have
$\langle P'\rangle =\langle P\rangle$.  Since $E(P)$ is an indecomposable injective, we notice that $E(P)=E(P')$. 

\smallskip
We now consider some $N\in \mathcal A_{fg}$ with $Ass(N)=\{\langle Q\rangle\}=\{\langle P\rangle\}=\{\langle P'\rangle\}$.  Again
from Lemma \ref{L8.65}, it follows that the injective hull $E(N)$ is a direct sum of copies of $E(P)=E(P')$, say
$E(N)=E(P')^{I}$. Since $N$ is finitely generated, the set $I$ is finite. Then, $E(N)$ is an essential extension of $P'^{I}$ and hence $N\subseteq E(N)$ 
is an essential extension of $N\cap P'^{I}$. 

\smallskip
Finally, since $\mathcal C$ is closed under finite direct sums,  we have $P'^{I}\in \mathcal C$. Then, the subobject
$N\cap P'^{I}$ of $P'^{I}$ lies in $\mathcal C$. Since $N$ is an essential extension of $N\cap P'^{I}$, we get
$N\in \mathcal C$. 
\end{proof}

\begin{lem}\label{L9.2} If $M\in \mathcal A_{fg}$ is a finitely generated object, then $M$ has only finitely many 
associated points.
\end{lem}

\begin{proof}
We know that the injective hull $E(M)$ admits a decomposition into indecomposable injectives, say $E(M)=\bigoplus_{i\in I}E_i$. 
Since $M$ is finitely generated, it must be contained inside the direct sum of some finite subcollection of
$\{E_i\}_{i\in I}$ and hence $I$ is finite. From Lemma \ref{L8.65}, we know that each indecomposable injective
$E_i=E(P_i)$ for some $P_i\in Spec(\mathcal A)$. Then, $Ass(M)=Ass(E(M))=\bigcup_{i\in I}Ass(E(P_i))=
\bigcup_{i\in I}\{\langle P_i\rangle\}$. 
\end{proof}

\begin{thm}\label{P9.3} 
Let $\mathcal C\subseteq \mathcal A$ be a full subcategory that is closed under subobjects, finite direct sums and essential extensions. Let $M\in \mathcal A_{fg}$ be such that $P\in \mathcal C$ for any $P\in Spec(\mathcal A)$ such that
there exists an inclusion $P\hookrightarrow M$. Then, $M\in \mathcal C$. 
\end{thm}

\begin{proof}
From Lemma \ref{L9.2}, we know that $M\in \mathcal A_{fg}$ has only a finite number of associated points. Suppose therefore
that $Ass(M)=\{\langle P_1\rangle, ...., \langle P_k\rangle\}$ where each $P_i$, $1\leq i\leq k$ is a spectral object
admitting an inclusion into $M$. We know that $Ass(P_i)=\{\langle P_i\rangle\}$.  Since $M$ is finitely generated 
(hence Noetherian), we can choose for each $i$ a maximal subobject $N_i\subseteq M$ such that
$Ass(N_i)=\{\langle P_i\rangle\}$. It is also clear that each $N_i\subseteq M$ is finitely generated. 

\smallskip
By the assumption on $M$, we know that each $P_i\in \mathcal C$. Since $Ass(N_i)=\{\langle P_i\rangle\}$, it now
follows from Lemma \ref{L9.1} that
each $N_i\in \mathcal C$. 

\smallskip
We now claim that $N_1\cap N_2=0$. Indeed, if there exists $0\ne K\subseteq N_1\cap N_2$, we have
$Ass(K)\subseteq Ass(N_1)\cap Ass(N_2)=\phi$ which is a contradiction. Hence, we can form the direct sum 
$N_1\oplus N_2$ as a subobject of $M$. Similarly, we note that $N_3\cap (N_1\oplus N_2)=0$ and so on, which
allows us to form the direct sum $N_1 \oplus ... \oplus N_k$ as a subobject of $M$. Since $\mathcal C$ 
is closed under finite direct sums, we have $N_1 \oplus ... \oplus N_k\in \mathcal C$. 

\smallskip
We will complete the proof by showing that $M$ is an essential extension of $N_1 \oplus ... \oplus N_k$. Otherwise,
there is some $0\ne N\subseteq M$ such that $N\cap (N_1 \oplus ... \oplus N_k)=0$. We pick some inclusion
$Q\hookrightarrow N$ with $Q\in Spec(\mathcal A)$. Then, $\langle Q\rangle\in Ass(N)\subseteq Ass(M)=\{\langle P_1\rangle, ...., \langle P_k\rangle\}$. For the sake of definiteness, we suppose that $\langle Q\rangle =\langle P_1\rangle$. We notice
that $Q\cap N_1\subseteq Q\cap (N_1 \oplus ... \oplus N_k)\subseteq N\cap (N_1 \oplus ... \oplus N_k)=0$. Hence, we can form
the direct sum $Q\oplus N_1$ as a subobject of $M$. We notice that $Q\oplus N_1\supsetneq N_1$ and 
$Ass(Q\oplus N_1)=\{\langle P_1\rangle\}$ which contradicts the assumption that $N_1$ is a maximal element 
among subobjects of $M$ having $\langle P_1\rangle$ as their only associated point. 
\end{proof}

\begin{thm}\label{Dr9.4} The associations:
\begin{equation*}
\begin{array}{c}
\mathcal C\mapsto \Phi(\mathcal C):=\underset{M\in \mathcal C}{\bigcup}\textrm{ }Ass(M)\\
S\mapsto \Psi(S):=\{\mbox{$M\in \mathcal A_{fg}$ $\vert$ $Ass(M)\subseteq S$}\}
\end{array}
\end{equation*} induce mutually inverse bijections between:

\smallskip
(a) The full and replete subcategories $\mathcal C\subseteq \mathcal A_{fg}$ that are closed under 
subobjects, finite direct sums and essential extensions.

\smallskip
(b) The subsets $S\subseteq \mathfrak Spec(\mathcal A)$ of the spectrum $\mathfrak Spec(\mathcal A)$
of $\mathcal A$. 
\end{thm}

\begin{proof}
It is immediate from the definitions that $\Phi\Psi(S)\subseteq S$. Conversely, if $P\in Spec(\mathcal A)$ is such that
$\langle P\rangle\in S$, we know that $P\in \Psi(S)$. Then, $\{\langle P\rangle\} = Ass(P)\subseteq \Phi\Psi(S)$ and hence
we have $S=\Phi\Psi(S)$. 

\smallskip
It is also clear that $\mathcal C\subseteq \Psi\Phi(\mathcal C)$. Conversely, we take $M\in \Psi\Phi(\mathcal C)$
and consider some $P\in Spec(\mathcal A)$ such that there is an inclusion $P\hookrightarrow M$. Then, $\langle P\rangle 
\in Ass(M)\subseteq \Phi(\mathcal C)$. Accordingly, we can find $N\in \mathcal C$ and an inclusion $Q\hookrightarrow N$
with $Q\in Spec(\mathcal A)$ such that $\langle Q\rangle =\langle P\rangle$. Since 
$Ass(P)=\{\langle P\rangle\}=\{\langle Q\rangle\}$, it follows from Lemma \ref{L9.1} that $P\in \mathcal C$. Applying 
Proposition \ref{P9.3}, we now obtain $M\in \mathcal C$. Hence, $\mathcal C= \Psi\Phi(\mathcal C)$. 

\end{proof}

\section{Examples} 

\medskip

\medskip
We will now present examples of abelian categories $\mathcal A$ whose subcategories may be studied using
the theory in the previous sections. For this, we will first make a list of the conditions that have been  applied to the locally noetherian Grothendieck category
$\mathcal A$ at various points in the paper:

\smallskip
(1) Every nonzero object of $\mathcal A$ has an associated point, i.e., for every $0\ne M\in \mathcal A$, we can find
a monomorphism $P\hookrightarrow M$ with $P\in Spec(\mathcal A)$. 

\smallskip
(2) For every spectral object $Q\in Spec(\mathcal A)$
and every $M\in \mathcal A$, we have
$
Ass(L_{\langle Q\rangle}(M))=Ass(M)\cap \mathfrak Spec(\mathcal A/\langle Q\rangle)
$ (see \eqref{tcondc1}).

\smallskip
(3) For any $P\in Spec(\mathcal A)$, the right adjoint $i_{\langle P
\rangle}:\mathcal A/\langle P\rangle\longrightarrow \mathcal A$ to the localization $L_{\langle P
\rangle}:\mathcal A\longrightarrow \mathcal A/\langle P\rangle$ is an exact functor. (see \eqref{dreamexact}).

\smallskip
These assumptions are easily verified in the case of $R$ being a commutative noetherian ring. 
In order to understand the  case of noncommutative $R$, we will discuss these assumptions one by one. 
We know that 
the category $R-Mod$ is locally noetherian exactly when $R$ is a left noetherian ring (see, for instance, \cite[(3.3)]{Ver}). 

\smallskip
In order to study the existence of associated points, Rosenberg  \cite[Chapter I]{R3} considers the following
preorder on the collection of left ideals : for left ideals $\mathfrak{m}$, $\mathfrak{n}\subseteq R$, set :
\begin{equation}\label{eq10.1}
\mathfrak{m}\leq \mathfrak{n} \quad\mbox{if}\quad(\mathfrak{m}:Z)=\{\mbox{$x\in R$ $\vert$ $xZ\subseteq 
\mathfrak{m}$ }\}\subseteq\mathfrak{n}\textrm{ }\mbox{for 
some $Z\in P(R)$}
\end{equation} where $P(R)$ denotes the collection of finitely generated  $\mathbb Z$-submodules 
of $R$. Then, Rosenberg \cite[$\S$ I.6.5.6]{R3} shows  that  if $R$ is ``left $\leq$-noetherian'', i.e., every
nonempty set of left ideals in $R$ has a maximal element with respect to $\leq$, then any nonzero left $R$-module has
an associated point. From 
\eqref{eq10.1}, it is clear that if $R$ is left noetherian and every left ideal in $R$ is also two-sided, then
$R$ is left $\leq$-noetherian and therefore satisfies (1). Further, we know from \cite[V.C2.3]{R3}  that for a left $\leq$-noetherian ring $R$, every
left $R$-module $M$ satisfies $Ass(M)=LAss(M)$. Applying Proposition \ref{sIPx5.10l}, it now follows that the category of modules
over any left $\leq$-noetherian ring $R$ satisfies (2).

\smallskip
We now come to condition (3). As  mentioned before, each spectral object $P\in Spec(\mathcal A)$ corresponds
to a localizing subcategory $\langle P\rangle$ and a flat localization:
\begin{equation}
L_{\langle P\rangle}:\mathcal A\longrightarrow \mathcal A/\langle P\rangle
\end{equation} We recall that the right adjoint $i_{\langle P\rangle}$ of $L_{\langle P\rangle}$ is full and faithful
and hence $\mathcal A/\langle P\rangle$ may be viewed as a subcategory of $\mathcal A$. When $\mathcal A=R-Mod$, we know (see \cite[I.0.4.1]{R3}) that each such subcategory arising from a flat localization is identical to a subcategory $(R-Mod)/F$ constructed from a  ``radical filter $F$ of left ideals in $R$.'' Further, any such filter $F$ defines a functor $\mathbb G_F: R-Mod\longrightarrow R-Mod$ 
known as the ``Gabriel functor''   (see \cite[I.0.4]{R3}), with $\mathbb G_F(R)$ carrying the structure of an $R$-algebra.
Every object in $(R-Mod)/F$ carries the structure of a left $\mathbb G_F(R)$-module. 

\smallskip This situation becomes rather simple
when $R$ is a left principal ideal domain. In that case every radical filter $F$ of left ideals gives a left Ore multiplicative 
set $S$ (see \cite[I.A.1.2]{R3}) and we can construct the left ring of fractions $S^{-1}R$. Then, the flat localization of $R-Mod$ is described
by the extension of scalars $R-Mod\longrightarrow S^{-1}R-Mod$. The right adjoint to this is the restriction 
of scalars $S^{-1}R-Mod\longrightarrow R-Mod$, which is clearly an exact functor. 

\smallskip
\begin{exm} 
\emph{Let $D$ be a
commutative principal ideal domain  which contains a copy $\rho: K\longrightarrow D$ of its quotient field $K$. We assume that
there is some $d\in D$ such that $\rho(d)\ne d$. The simplest case would be to take $D=K$ and $\rho$ to be a nontrivial
automorphism of $K$. For more examples, the reader may see, for instance, \cite[$\S$ 4]{Jat}. 
We  define the skew formal power series ring $R=D\langle x,\rho\rangle$ as follows : as an abelian group, $D\langle 
x,\rho\rangle$ consists of all formal power series $\sum_{i\geq 0}d_ix^i$ with all coefficients $d_i\in D$. The multiplication 
is defined by the rule:}
\begin{equation}
xd =\rho(d)x \qquad\forall\textrm{ }d\in D
\end{equation} \emph{Then, an interesting result of Jategaonkar \cite[Theorem 1]{Jat} states that the skew
formal power series ring $D\langle x,\rho\rangle$ is a left principal ideal domain in which every left ideal is two-sided. Since 
a left principal ideal domain is automatically  left noetherian, it is clear from the discussion above that the category
of left modules over $R=D\langle x,\rho\rangle$ satisfies all the conditions (1), (2),  and (3) above.}
\end{exm}

\begin{exm}\emph{
 We know that the category $QCoh(X)$ of quasi-coherent sheaves on a separated noetherian scheme $
(X,\mathcal O_X)$. 
 is a locally noetherian Grothendieck category (see \cite[$\S$ II.7]{Hart} and \cite[Tag 077P]{Stacks}). 
We now consider a point $x\in X$. The point $x\in X$ determines a quasi-coherent sheaf $P_x\in QCoh(X)$ as follows : if $U$
is an affine open containing $x$, we set $P_x(U)=\mathcal O_X(U)/p_x(U)$, where $p_x(U)$ is the prime ideal of $\mathcal O_X(U)$
corresponding to $x$. On all other affines, $P_x$ is set to be zero. Then, Rosenberg \cite[$\S$ 7.2]{Rose1} has shown that every
spectral object in $QCoh(X)$ is equivalent to some such $P_x$. Then, the right adjoint of the localization 
$QCoh(X)\longrightarrow QCoh(X)/\langle P_x\rangle$ is the pushforward on quasi-coherent sheaves along the morphism:
$
j_x: Spec(\mathcal O_{X,x})\longrightarrow X
$. If $U$ is any affine containing $x$, the morphism $j_x$ factors as:
\begin{equation}\label{eq10.5}
Spec(\mathcal O_{X,x})=Spec(\mathcal O_X(U)_{p_x(U)}){\longrightarrow} Spec(\mathcal O_X(U))=U{\hookrightarrow} X
\end{equation} The pushforwards on quasi-coherent sheaves along both the morphisms in \eqref{eq10.5} are exact and hence
so is the pushforward $j_{x,\ast}$. This proves condition  (3) for the category $Qcoh(X)$.}

\smallskip
\emph{We now come to condition (1) : we have show that every nonzero quasi-coherent sheaf $\mathcal M$ on $X$ has an associated point, i.e., $\mathcal M$
has a quasi-coherent subsheaf that is a spectral object in $QCoh(X)$. It follows from \cite[Tag 01PF, 01XZ, 01YF]{Stacks} that we can find an integral closed subscheme $i:Z\hookrightarrow X$
and a sheaf $0\ne \mathcal I\subseteq\mathcal O_Z$ of ideals on $Z$ such that $i_\ast(\mathcal I)\subseteq \mathcal M$.  It suffices
therefore to prove the result for $\mathcal M=i_\ast(\mathcal I)$.}

\smallskip
\emph{We now consider an affine open $j_U:U=Spec(S)\hookrightarrow X$ such that $\mathcal M(U)\ne 0$. Then, there is a point $x\in Spec(S)
\subseteq X$ such that there is a monomorphism $S/p_x(U)\hookrightarrow \mathcal M(U)$, where $p_x(U)$
is the prime ideal in $S$ corresponding to $x\in Spec(S)$. This shows that $(P_x|U)\hookrightarrow
(\mathcal M|U)$ where $P_x\in QCoh(X)$ is the spectral object of $QCoh(X)$ corresponding to $x\in X$ as constructed above. 
Applying the right adjoint functor $j_{U*}:QCoh(U)\longrightarrow QCoh(X)$, we obtain a monomorphism
$j_{U*}(P_x|U)\hookrightarrow
j_{U*}(\mathcal M|U)$. }

\smallskip
\emph{We now claim that the canonical morphism $\mathcal M\longrightarrow j_{U*}(\mathcal M|U)$ 
is a monomorphism for $\mathcal M=i_*(\mathcal I)$ as above. In other words,
we claim that for each affine $V\subseteq X$, the natural morphism:
\begin{equation}\label{eq10.61e}
\mathcal M(V)=\mathcal I(V\cap Z)\longrightarrow  \mathcal I(U\cap V\cap Z)=\mathcal M(U\cap V)=(\mathcal M|U)(U\cap V)=j_{U*}(\mathcal M|U)(V)
\end{equation} is a monomorphism. Since closed immersions are affine and $X$ is separated, the inclusion $U\cap V
\cap Z\hookrightarrow V\cap Z$ is a Zariski immersion of affine schemes $Spec(B)=U\cap V
\cap Z\hookrightarrow V\cap Z=Spec(A)$. Since $Z$ is an integral scheme, we know that $A$ is an integral domain
and hence the ideal $\mathcal I(V\cap Z)\subseteq A$ is a torsion free $A$-module. Then, for any nonzero element 
$f\in A$, the morphism $\mathcal I(V\cap Z)\longrightarrow \mathcal I(V\cap Z)\otimes_AA_f$ is a monomorphism. Since
the Zariski open $Spec(B)\hookrightarrow  Spec(A)$ contains at least one basic open set of the form $Spec(A_f)$, $0\ne f\in A$, 
it follows that $\mathcal I(V\cap Z)\longrightarrow \mathcal I(V\cap Z)\otimes_AB=\mathcal I(U\cap V\cap Z)$
is a monomorphism. }

\smallskip
\emph{We now form the fiber square:
\begin{equation}\label{eq10.7}
\begin{CD}
\mathcal N @>>> \mathcal M=i_*(\mathcal I)\\
@VVV @VVV\\
j_{U*}(P_x|U) @>>> j_{U*}(\mathcal M|U)
\end{CD}
\end{equation} Then, we have monomorphisms $\mathcal N\hookrightarrow \mathcal M$ and 
$\mathcal N\hookrightarrow j_{U*}(P_x|U)$. We claim that the quasi-coherent subsheaf $\mathcal N\hookrightarrow \mathcal M$ is a spectral object
of $QCoh(X)$. First of all, if we restrict the diagram \eqref{eq10.7} to the open set $U$, the right vertical arrow becomes
an isomorphism and hence $\mathcal N|U=P_x|U$. This gives $\mathcal N\ne 0$. Finally, the proof of \cite[$\S$ 7.2]{Rose1} shows
that $j_{U*}(P_x|U)$ is also a spectral object with $\langle j_{U*}(P_x|U)\rangle =\langle P_x\rangle\in \mathfrak Spec
(QCoh(X))$. Then $\mathcal N$ which is a non-trivial subobject of $j_{U*}(P_x|U)$ is also spectral and we have obtained
an associated point of the object $\mathcal M\in QCoh(X)$ in the sense of Section 2. }

\smallskip
\emph{It remains to check  (2) for the category $QCoh(X)$. By Proposition \ref{sIPx5.10l}, it suffices to show that $LAss(\mathcal M)=Ass(\mathcal M)$
for each $\mathcal M\in QCoh(X)$. We first suppose that $\mathcal M=i_*(\mathcal I^{\oplus r})$ for some $r\geq 0$, where $i:Z\hookrightarrow X$ is an integral
closed subscheme and $\mathcal I\subseteq \mathcal O_Z$ is a sheaf of ideals. We choose a point $\langle P_x\rangle \in LAss(\mathcal M)$ and consider
the  stalk $\mathcal M_x$ at $x\in X$. Then, $\mathcal M_x$ is an $\mathcal O_{X,x}$-module. If 
$m_x$ is the maximal ideal in $\mathcal O_{X,x}$,  it follows that $m_x$ is an associated prime of $\mathcal M_x$.   If 
$j_U:Spec(S)=U\hookrightarrow X$ is an affine open with $x\in U$, we use the fact that $LAss(\mathcal M(U))=Ass(\mathcal M(U))$ to conclude
that there is a monomorphism $P_x|U\hookrightarrow \mathcal M|U$. Since $\mathcal M=i_*(\mathcal I^{\oplus r})$, it now follows as in \eqref{eq10.61e}
and \eqref{eq10.7} that $\langle P_x\rangle \in Ass(\mathcal M)$. }

\smallskip
\emph{More generally, suppose that $\mathcal M$ is a coherent sheaf on $X$, i.e., quasi-coherent and of finite type. Choose $\langle P_x\rangle \in LAss(\mathcal M)$, i.e., $m_x$ is an associated prime of $\mathcal M_x$.  We consider the closed integral
subscheme $i:Z_x:=\overline{\{x\}}\hookrightarrow X$ and its corresponding sheaf of ideals $\mathcal J_x\subseteq \mathcal O_X$. We now consider $\mathcal M'\subseteq \mathcal M$ defined by setting
\begin{equation}
\mathcal M'(V):=\{\mbox{$n\in \mathcal M(V)$ $\vert$ $\mathcal J_x(V)n=0$}\}
\end{equation} for each affine open $V\subseteq X$. It follows from \cite[Tag 01PO]{Stacks} that $\mathcal M'$ is quasi-coherent and that
\begin{equation}
\mathcal M'_x=\{\mbox{$n\in \mathcal M_x$ $\vert$ $sn=0$ for all $s\in m_x$}\}
\end{equation} It is immediate that $\mathcal M'\ne 0$ and that 
the maximal ideal $m_x$ is an associated prime of $\mathcal M'_x$.   Actually, since $\mathcal M'_x$ is annihilated by $m_x$, it follows
from \cite[Tag 01YE]{Stacks} that there is an integer $r\geq 0$, a sheaf of ideals $\mathcal I\subseteq \mathcal O_{Z_x}$ and an affine open $W$ containing $x$ such that we have a monomorphism 
$
i_*(\mathcal I^{\oplus r})\hookrightarrow \mathcal M'
$  which restricts to an isomorphism on $W$. Then,  $m_x$ is an associated prime of $(i_*(\mathcal I^{\oplus r}))_x=\mathcal M'_x$, i.e.,
$\langle P_x\rangle\in LAss(i_*(\mathcal I^{\oplus r}))$. From the above, we know that $LAss(i_*(\mathcal I^{\oplus r}))=Ass(i_*(\mathcal I^{\oplus r}))$. Hence, $\langle P_x\rangle \in Ass(i_*(\mathcal I^{\oplus r}))\subseteq Ass(\mathcal M')\subseteq Ass(\mathcal M)$. }

\smallskip
\emph{Finally, we know that any quasi-coherent sheaf $\mathcal M$ on $X$ may be expressed as a directed colimit of its quasi-coherent submodules of finite type. Applying \cite[III.8.2, V.C2.3.1]{R3}, it now follows that
\begin{equation}
LAss(\mathcal M)=\underset{\tiny \begin{array}{c}\mbox{$\mathcal N\subseteq \mathcal M$} \\ \mbox{q.c. \& f.g.}\\
\end{array}}{\bigcup}\textrm{ }LAss(\mathcal N)=\underset{\tiny \begin{array}{c}\mbox{$\mathcal N\subseteq \mathcal M$} \\ \mbox{q.c. \& f.g.}\\
\end{array}}{\bigcup}\textrm{ }Ass(\mathcal N)=Ass(\mathcal M)
\end{equation} This proves assumption (2) for the category $QCoh(X)$.  }
\end{exm}


\small
\begin{thebibliography}{99}

\bibitem{AdRos} J. Ad\'{a}mek, J. Rosick\'{y}, Locally presentable and accessible categories, {\it London Mathematical
Society Lecture Note Series}, {\bf 189}, Cambridge University Press, Cambridge,
1994.

\bibitem{ABJp} A.~Banerjee,  
Two criteria for locally noetherian Grothendieck categories, 
{\it J. Pure Appl. Algebra} {\bf 227} (2023), no. 3, Paper No. 107233.

\bibitem{Bazz}  S.~Bazzoni. A characterization of $n$-cotilting and $n$-tilting modules, {\it J. Algebra}, {\bf 273(1)}, 359--372, (2004).

\bibitem{BI} A.~Beligiannis, I.~Reiten, 
Homological and homotopical aspects of torsion theories, 
{\it Mem. Amer. Math. Soc.} {\bf 188} (2007), no. 883. 

\bibitem{BOPP} D.~Bravo,  S.~Odaba\c{s}i,  C.~E.~Parra,  M.~A.~P\'{e}rez,  
Torsion and torsion-free classes from objects of finite type in Grothendieck categories, 
{\it J. Algebra} {\bf 608} (2022), 412--444.

\bibitem{BH} W.~Bruns, J.~Herzog. Cohen-Macaulay rings, Vol {\bf 39}, {\it Cambridge Studies in Advanced
Mathematics}, Cambridge University Press, Cambridge, 1993.

\bibitem{CT} R.~Colpi, J.~Trlifaj. Tilting modules and tilting torsion theories. {\it J. Algebra}, {\bf 178(2)}, 614--634, (1995).

\bibitem{Colpi} R.~Colpi, Tilting in Grothendieck categories, {\it Forum Math}, {\bf 11}, No. 6, 735--759 (1999). 

\bibitem{EGRO} E.~E.~Enochs,  J.~R.~Garcia Rozas, J L.~Oyonarte,  
Finitely generated cotorsion modules, 
{\it Proc. Edinb. Math. Soc. (2)} {\bf 44} (2001), no. 1, 143--152.

\bibitem{Gab} P.~Gabriel, Des cat\'{e}gories ab\'{e}liennes, {\it Bull. Soc. Math. France} 
{\bf 90} 1962 323--448.

\bibitem{GP1} G.~Garkusha, M.~Prest, Classifying Serre subcategories of finitely presented modules, 
{\it Proc. Amer. Math. Soc.} {\bf 136} (2008), no. 3, 761--770.

\bibitem{GP2}  G.~Garkusha, M.~Prest,  Torsion classes of finite type and spectra,
{\it Proceedings of the Conference on K-theory and noncommutative geometry (Valladolid, 2006)}, 393--412, 
{\it EMS Ser. Congr. Rep., Eur. Math. Soc.}, Z\"{u}rich, 2008. 


\bibitem{Tohoku} A.~Grothendieck, Sur quelques points d'alg\`{e}bre homologique, I. {\it Tohoku Math. J. (2)} 
{\bf  9} (1957), no. 2, 119--221.

\bibitem{GAIT} P.~A.~Guil Asensio,  M.~C.~Izurdiaga,  B.~Torrecillas, Blas, 
On the existence of non-trivial finitely injective modules, 
{\it Forum Math.} {\bf 26} (2014), no. 6, 1629--1633.



\bibitem{Hart} R.~Hartshorne, Residues and duality, Lecture notes of a seminar on the work of A. Grothendieck, given at Harvard 1963/64. With an appendix by P. Deligne, {\it Lecture Notes in Mathematics}, No. {\bf 20}, Springer-Verlag, Berlin-New York 1966.

\bibitem{Herz} I.~Herzog,  The Ziegler spectrum of a locally coherent Grothendieck category, {\it Proc. London Math. Soc. (3)} {\bf 74} (1997), no. 3, 503--558. 


\bibitem{Four} L.~A.~H\"{u}gel,  D.~Posp\'{i}\v{s}il, J.~\v{S}\v{t}ov\'{i}\v{c}ek,  J.~Trlifaj, 
Tilting, cotilting, and spectra of commutative Noetherian rings, 
{\it Trans. Amer. Math. Soc.} {\bf 366} (2014), no. 7, 3487--3517. 

\bibitem{Izur} M.~C.~Izurdiaga,  
The cotorsion pair generated by the class of flat Mittag-Leffler modules, 
{J. Algebra}  {\bf 479} (2017), 203--215.
 

\bibitem{Jat} A.~V.~Jategaonkar, Left principal ideal domains, {\it J. Algebra}, {\bf 8}, 148--155 (1968).  

\bibitem{Kanda1} R.~Kanda, Classifying Serre subcategories via atom spectrum, {\it Adv. Math.} 
{\bf  231} (2012), no. 3-4, 1572--1588.

\bibitem{Kanda2} R.~Kanda, Extension groups between atoms and objects in locally noetherian Grothendieck category, 
{\it J. Algebra}  {\bf 422} (2015), 53--77.

\bibitem{Kanda3} R.~Kanda, Specialization orders on atom spectra of Grothendieck categories, {\it J. Pure Appl. Algebra}
{\bf  219} (2015), no. 11, 4907--4952.

\bibitem{Kanda4} R.~Kanda, Classification of categorical subspaces of locally Noetherian schemes,
{\it  Doc. Math.}  {\bf 20} (2015), 1403--1465.

\bibitem{KR1} M.~Kontsevich, A.~L.~Rosenberg,  
Noncommutative smooth spaces, The Gelfand Mathematical Seminars, 1996–1999, 85--108,
{\it Gelfand Math. Sem.}, Birkhauser Boston, Boston, MA, 2000.

\bibitem{KR2} M.~Kontsevich, A.~L.~Rosenberg,  
Noncommutative spaces, MPI-M Preprint (2004).

\bibitem{KR3} M.~Kontsevich, A.~L.~Rosenberg,  
Noncommutative spaces and flat descent, MPI-M Preprint (2004).

\bibitem{KR4} M.~Kontsevich, A.~L.~Rosenberg,  
Noncommutative stacks, MPI-M Preprint (2004).

\bibitem{Hayd} H.~Lindo,  
Trace ideals and centers of endomorphism rings of modules over commutative rings,
{J. Algebra} {\bf 482} (2017), 102--130.


\bibitem{MaMu} S.~Mohamed,  B.~J.~M\"{u}ller,  
Continuous and discrete modules, {\it 
London Mathematical Society Lecture Note Series,} {\bf 147.} Cambridge University Press, Cambridge, 1990.

\bibitem{Nee} A.~Neeman, The chromatic tower for D(R). With an appendix by Marcel B\"{o}kstedt,
{\it  Topology} 
{\bf 31} (1992), no. 3, 519--532.

\bibitem{PS} C.~ E.~Parra, M.~Saor\'{i}n, Manuel, 
Hearts of t-structures in the derived category of a commutative Noetherian ring, 
{\it Trans. Amer. Math. Soc.} {\bf 369} (2017), no. 11, 7789--7827.


\bibitem{Rose2} A.~L.~Rosenberg, Noncommutative local algebra, {\it Geometry and Functional Analysis}, Vol. {\bf 4},  (1994) 
No. 5, 545--585.

\bibitem{R3} A.~L.~Rosenberg, Noncommutative algebraic geometry and representations of quantized
algebras, Kluwer Academie Publishers, Mathematies and Its Applications, Vol. {\bf 330} (1995). 

\bibitem{Rose1}  A.~L.~Rosenberg, Reconstruction of Schemes, MPI Preprints Series, 1996
(108).

\bibitem{SW} D.~Stanley, B.~Wang, Classifying subcategories of finitely generated modules over a
Noetherian ring, {\it Journal of Pure and Applied Algebra}, {\bf 215} (2011) 2684--2693. 

\bibitem{Stacks} The Stacks project, Available online from stacks.math.columbia.edu. 

\bibitem{Stov}  J.~\v{S}\v{t}ov\'{i}\v{c}ek, Derived equivalences induced by big cotilting modules, {\it Adv. Math.} {\bf 263} (2014), 45--87.


\bibitem{Bo} B. Stenstr\"{o}m, Rings of Quotients, {\it Die Grundlehren der Mathematischen Wissenschaften}, vol. {\bf 217}, Springer-Verlag, New York, 1975.

\bibitem{Taka} R.~Takahashi, Classifying subcategories of modules over a commutative Noetherian ring,
{\it  J.
Lond. Math. Soc. (2)}, {\bf 78}(3):767--782, 2008.


\bibitem{Ver} A.~Verschoren, Compatibility and stability, {\it Notas de Matem\'{a}tica}, {\bf 3}. Universidad de Murcia, Secretariado de Publicaciones e Intercambio Cient\'{i}fico, Murcia, 1990. 



\end{thebibliography}
\end{document}